\documentclass[aoas,preprint]{imsart}

\RequirePackage[OT1]{fontenc}
\RequirePackage{amsthm,amsmath}
\RequirePackage[numbers]{natbib}
\RequirePackage[colorlinks,citecolor=blue,urlcolor=blue]{hyperref}

\usepackage[all,cmtip]{xy}

\usepackage{amsmath}
\usepackage{amsthm}
\usepackage{amsfonts}
\usepackage{amssymb}
\usepackage{graphicx}
\usepackage{subfigure}
\usepackage{multirow}
\usepackage{color}
\usepackage{blkarray}
\usepackage{natbib}
\usepackage[dvipsnames]{xcolor}

\newcommand{\bm}[1]{\boldsymbol{#1}}
\newcommand{\E}{\mathbb{E\,}}

\renewcommand{\d}{\mathrm{d}}

\newcommand{\Y}{\bm{Y}}

\renewcommand{\H}{\bm{H}}
\newcommand{\Q}{\bm{Q}}

\newcommand{\wh}{\widehat}

\def\C {\,|\:}

\def\C {\,|\:}

\def\B{\bm{B}}
\def\Y{\bm{Y}}

\def\X{\bm{X}}
\def\x{\bm{x}}
\def\y{\bm{y}}
\def\by{\bm{y}}
\def\bg{\bm{\gamma}}

\def\z{\bm{z}}

\def\b{\bm{\beta}}

\renewcommand{\d}{\mathrm{d}\,}

\newcommand{\bdm}{\begin{displaymath}}
\newcommand{\edm}{\end{displaymath}}

\usepackage{amssymb,tikz}

\newtheorem{definition}{Definition}[section]
\newtheorem{lemma}{Lemma}[section]

\newtheorem{theorem}{Theorem}[section]
\newtheorem{example}{Example}[section]
\newtheorem{remark}{Remark}[section]
\newtheorem{corollary}{Corollary}[section]

 \theoremstyle{assumption}

\startlocaldefs
\numberwithin{equation}{section}
\theoremstyle{plain}

\endlocaldefs

\begin{document}

\begin{frontmatter}
\title{The Median Probability Model and Correlated Variables}
\runtitle{Median Probability Model}
\begin{aug}

\author{\fnms{Marilena} \snm{Barbieri}\thanksref{m1}\ead[label=e1]{first@somewhere.com}},
\author{\fnms{James O.} \snm{Berger}\thanksref{m2}\ead[label=e2]{second@somewhere.com}}\\
\author{\fnms{Edward I.} \snm{George and}\thanksref{m3}\ead[label=e1]{first@somewhere.com}},
\author{\fnms{Veronika} \snm{Ro\v{c}kov\'{a}}\thanksref{t4,m4}\ead[label=e1]{first@somewhere.com}},\\
{\normalsize \sl 17 August  2018}


\thankstext{t4}{
Veronika.Rockova@ChicagoBooth.edu; \hspace{12cm}
 This This work was supported by the James S. Kemper Foundation Faculty Research Fund at  the University of Chicago Booth School of Business.}


\affiliation{Universit\`a Roma Tre \thanksmark{m1}}
\affiliation{Duke University \thanksmark{m2}}
\affiliation{University of Pennsylvania \thanksmark{m3}}
\affiliation{University of Chicago \thanksmark{m4}}

\address{}

\end{aug}

\begin{abstract}
The median probability model (MPM) \citep{barbieri} is defined as the model  consisting of those variables  whose marginal posterior  probability of inclusion is at least $0.5$. The MPM rule yields  the best single model for prediction  in orthogonal and nested correlated designs. This result was  originally conceived under a specific class of priors, such as the point mass mixtures of non-informative  and $g$-type priors.
The MPM rule, however, has become  so very popular that it is now being deployed for a wider variety of priors and under correlated designs, where the properties of MPM are not yet completely understood.
The main thrust of this work is to shed light on properties of MPM
in these contexts by (a) characterizing situations when MPM is still safe under correlated designs, (b) providing significant generalizations of MPM to a broader class of priors (such as continuous spike-and-slab priors). We also provide new supporting evidence for the suitability of $g$-priors, as opposed to independent product priors, using new predictive matching arguments. Furthermore, we  emphasize the importance of  prior model probabilities and highlight
the merits of non-uniform prior probability assignments using the notion of model aggregates.\end{abstract}



\end{frontmatter}

\section{Introduction}
This paper investigates to which extent the median probability model rule of  Barbieri and Berger \cite{barbieri} can be used for variable selection when the covariates are correlated.
To this end, we consider the usual linear model
\vspace{-0.1cm}
\begin{equation}\label{model}
\bm Y_{n \times 1} \sim  \mathcal{N}_{n}\left(\bm X \bm \beta ,\sigma^2 \bm I \right),
\end{equation}
where $\Y$ is the $n\times 1$ vector of responses, $\X$ is the $n\times q$ design matrix of covariates, $\b$ is a $q\times 1$ vector of unknown coefficients, and $\sigma^2$ is a known scalar.
The equation \eqref{model} corresponds to the full model and we are interested in selecting a submodel indexed by $\bg=(\gamma_1,\dots,\gamma_{q})'$, where $\gamma_i\in\{1,0\}$ for whether the $i^{th}$ covariate is in or out of the model. We tacitly assume that the response has been centered and thereby omit the intercept term.

For prediction of a new observation $y^\star$ from $\x^\star$ under squared error loss, the optimal model $\bg^{o}$   is known to  satisfy (Lemma 1 of \cite{barbieri})
\begin{equation}\label{Rcrit}
\bg^{o}=\arg\min_{\bg} R(\bg)\quad\text{with}\quad R(\bg)\equiv \left(\H_{\bg}\wh\b_{\bg}-\bar\b\right)'\Q\left(\H_{\bg}\wh\b_{\bg}-\bar\b \right),
\end{equation}
where $\bar{\b}=\E[\b\C\Y]$ is the overall posterior mean of $\b$ under the hierarchical prior $\pi(\bg)$ and $\pi(\b\C\bg)$, $\H_{\bg}$ is a $q\times |\bg|$ stretching matrix (defined in Section 2.2 of \cite{barbieri}) whose $(i,j)$ entry is  $1$ if $\gamma_i=1$ and $j=\sum_{r=1}^i\gamma_r$ and $0$ otherwise, and $\wh{\b}_{\bg}$ is the conditional posterior mean under $\bg$ and where $\Q=\E[\x^\star\x^{\star\prime}]=\X'\X$.

Contrary to what might be commonly conceived as an optimal predictive model,  $\bg^o$ is {\sl not} necessarily the {modal} highest posterior probability model.  In orthogonal and nested correlated designs, \cite{barbieri} show that the optimal model {$\bg^o$} is the {\sl median probability model} $\bg^{MP}$. This is defined to be the model consisting of  variables whose marginal inclusion probability $\pi(\gamma_i=1\C\Y)$ is  at least $0.5$. This model can be regarded as the best single-model approximation to model averaging.

{Compared to the often targeted highest posterior model (HPM), a major attraction of the median probability model (MPM) is the {relative} ease with which it can be found via MCMC.  Whereas finding the HPM requires identification of the largest of the $2^p$ distinct model probabilities, finding the MPM { is much less costly, requiring identification only of those of the $p$ marginal covariate probabilities which are greater $0.5$.  And when the HPM and MPM are identical, which may often be the case, the MPM offers a much faster route to computing them both.

The {MPM} is now routinely used for distilling posterior evidence towards variable selection; \cite{scott2008feature,clyde2011bayesian,feldkircher2012forecast,garcia2013sampling,ghosh2015bayesian,piironen2017comparison} and \cite{drachal2018comparison} are
some of the articles that have used and discussed the performance of the MPM. Despite its widespread use in practice, however, the optimality of MPM has so far been shown under comparatively limited circumstances. In particular, the priors  $\pi(\b\C\bg)$ are required to be such that the MPM estimator $\wh{\b}_{\bg}$ is proportional to the MLE estimator under $\bg$. This property will be satisfied by e.g. the {\sl point-mass} spike-and-slab $g$-type priors \citep{zellner,liang2008mixtures}. However, the also very popular {\sl continuous} spike-and-slab mixtures \citep{GM93, ishrawan_rao1,RG14,rockova2017}  will fail to satisfy this requirement.  Here, we will show that this condition is not necessary for MPM to be predictive optimal.   In particular, we provide significant generalizations of the existing MPM optimality results for a wider range of priors such as the continuous spike-and-slab mixtures and, more generally, independent product priors.

Barbieri and Berger \cite{barbieri} presented a situation with correlated covariates (due to Merlise Clyde) in which the MPM was clearly not optimal. Thus
there has been a concern that correlated covariates (reality) might make the MPM practically irrelevant.
Hence another purpose of this paper is to explore the extent to which correlated covariates can degrade the performance of the MPM.
We address this with theoretical studies concerning the impact of correlated covariates, and numerical studies; the magnitude of the
scientific domain here limits us (in the numerical studies) to consider a relatively exhaustive study of the two variable case, made
possible by geometric considerations.
The overall conclusion is that (in reality) there can be a small degradation of performance, but the degradation is less than that
experienced by the HPM in correlated scenarios.

First, using predictive matching arguments \citep{Berg:Peri:01,bayarri_etal,fouskakis2018power}, we provide new arguments for the
suitability of $g$-type priors as opposed to independent product priors. Going further, we highlight the importance of prior model probabilities assignments and discuss their {``dilution"} issues \citep{george2010}  in highly collinear designs.  Introducing the notion of model aggregates, we showcase the somewhat peculiar behavior of separable model priors
obtained with a fixed prior inclusion probability.  We show that the beta-binomial prior copes far better with variable redundancy.
We also characterize the optimal predictive model and relate it to the MPM through relative risk comparisons. We also provide several ``mini-theorems"  showing  predictive (sub)optimality of the MPM when $q=2$.

The paper is structured as follows. Section 2 introduces the notion of model collectives and looks into some interesting limiting behaviors of the MPM when the predictors are correlated.
Section 3 delves into a special case with 2 collinear predictors. Section 4 generalizes the optimality of the MPM to other priors and Section 5 wraps up with a discussion.



\section{The effect of many highly correlated variables and $g$-priors on the median probability model}

\subsection{The marginal likelihood in the presence of many highly correlated variables}
One  reasonable requirement for objective model selection priors is that they be properly matched across models that are indistinguishable from a predictive point of view.
Recall that two models are regarded as {\sl predictive matching} \citep{bayarri_etal} if their marginal likelihoods are close in terms of some distance.
In this section, we take a closer look at the marginal likelihood for the model \eqref{model} under the celebrated $g$-priors \citep{zellner}, assuming that the design matrix  satisfies
\begin{equation}\label{Xmatrix}
\bm X_{n\times (p+k)} = [\bm B_{n\times p}, \ \bm x + \epsilon\, \bm \delta_1, \ \cdots, \ \bm x + \epsilon\, \bm \delta_k]
\end{equation}
for some $\epsilon>0$, where $\bm B$ consists of $p$ possibly correlated regressors and where $\bm \delta_1,\ldots, \bm\delta_k$ are $(n \times 1)$ perturbation vectors.
We assume that $\bm \delta_1,\ldots, \bm\delta_k$ are orthonormal and orthogonal to $\bm x$ and $\bm B,${}\footnote{This assumption is not necessary, but greatly simplifies the illustration.} while $\bm x$ and $\bm B$ are not necessarily orthogonal.
 We will be letting $\epsilon$ be very small to model the situation of having $k$ highly correlated variables.
For the full model \eqref{model}, the $g$-prior is
$$
\b_{(p+k) \times 1} \sim\mathcal{N}_{p+k}\left(\bm{0},g\,\sigma^2(\X'\X)^{-1}\right)
$$
for some $g>0$ (typically $n$) and the corresponding marginal likelihood is
$$ \bm Y \sim \mathcal{N}_{n}\left(\bm{0},\sigma^2 \bm I+g\,\sigma^2\X(\X'\X)^{-1}\X'\right) \,.$$
Note that $\X(\X'\X)^{-1}\X'$ is the projection matrix onto the column space of $\bm X$. Hence, having near duplicate columns in $\bm X$ should not change this matrix much at all. Indeed, the following Lemma shows that, as $\epsilon\rightarrow 0$,
this is  a fixed matrix (depending only on $\bm B$ and $\bm x$).

\begin{lemma}\label{lemma:projection}
Denote with
$
\bm P = \lim\limits_{\epsilon \rightarrow 0} \,\, \X(\X'\X)^{-1}\X'.
$
Then
$$ \bm P= \bm P_{\bm B} + \frac{1}{(\|\bm x\|^2 - \bm x' \bm P_{\bm B} \ \bm x)}
[\bm P_{\bm B} \ \bm x -\bm x][\bm x'\bm P_{\bm B} -\bm x']    \,,$$
where $\bm P_{\bm B} = \bm B (\bm B' \bm B)^{-1}\bm B'$.
\end{lemma}
\begin{proof}
Let $\bm 1$ be the $k$-column vector of ones, so that $\bm 1\bm 1'$ is the $k \times k$ matrix of ones, and let $ \bm v = \bm B' \bm x$.
Note first that
$$ \bm X' \bm X = \left(
                    \begin{array}{cc}
                     \bm B' \bm B & \bm v \bm 1' \\
                      \bm 1 \bm v'  & \|\bm x\|^2 \bm 1\bm1' +  \epsilon \bm I \\
                    \end{array}
                  \right)
$$
and, letting $C=(\|\bm x\|^2 - \bm v' (\bm B' \bm B)^{-1}\bm v)$, we have $ (\bm X' \bm X)^{-1} =$
$$ \left(
                    \begin{array}{cc}
                   \left( ( \bm B' \bm B)^{-1} +
                    \frac{k}{\epsilon +kC}( \bm B' \bm B)^{-1}\bm v \bm v'( \bm B' \bm B)^{-1}\right) & -(\epsilon+ k C)^{-1} (\bm B' \bm B)^{-1}\bm v\bm 1'  \\
                     -(\epsilon+ k C)^{-1} \bm 1 \bm v' (\bm B' \bm B)^{-1} & \frac{1}{\epsilon}\left(\bm I - \frac{C}{\epsilon +kC}\bm 1\bm 1' \right) \\
                    \end{array}
                  \right) \,.
$$
The result follows by multiplying this matrix with $\X$ and $\X'$, and taking the limit as $\epsilon\rightarrow 0$.
\end{proof}

One important conclusion from Lemma \ref{lemma:projection} is that no matter how many columns of highly correlated variables are present in the model, the marginal likelihood will essentially be $$ \bm Y \sim \mathcal{N}_{n}\left(\bm{0},\sigma^2 \bm I+g\,\sigma^2 \bm P \right)$$
as $\epsilon\rightarrow0$.
Thereby all models including all predictors in $\B$ and at least one replicate of $\x$ can be essentially regarded as predictive matching.

We let $\bg=(\bg_1',\bg_2')'$ denote the global vector of inclusion indicators, where $\bg_1$ is associated with $\B$ and $\bg_2$ is associated with the $k$ near duplicates. The same analysis holds for any sub-model $\bg_1\in\{0,1\}^p$, defined by the design matrix $\bm B_{\bm \gamma_1}$ consisting of the active variables corresponding to the $1$'s in $\bm \gamma_1$.   Before proceeding, we introduce the notion of a  {\sl model collective}
which will be useful for characterizing the properties of $g$-priors and the median probability model in collinear designs.
 \begin{definition}(A model collective)\label{def:aggregate}
 Let $\bg_1\in\{0,1\}^p$ be a vector of inclusion indicators associated with the $p$ variables in $\B$.  Denote by $M_{\bg_1, \bm x}$ the {\sl model collective} comprising all models consisting of
  the $\bm \gamma_1$ variables together with one or more of the (near) duplicates of $\bm x$.
 \end{definition}

 Let $\bm P_{\bm \gamma_1}$ be the
 limiting projection matrix corresponding to any of the
 models inside the model collective $M_{\bg_1,\x}$.
 The limiting marginal likelihood of such models is
 \begin{equation}\label{m_function}
 m(\bm y \mid \bm \gamma_1,\bm x) =\phi\left(\bm y \mid \bm{0},\sigma^2 \bm I+g\sigma^2 \bm P_{\bm \gamma_1} \right),
 \end{equation}
where $\phi(\bm y\mid \bm \mu,\Sigma)$ denotes a multivariate Gaussian density with mean vector $\bm \mu$ and covariance matrix $\Sigma$.

\begin{lemma}\label{lemma:marg}  Let $  m(\bm y \mid \bm \gamma_1)$ denote the marginal likelihood under the model $\bg_1$. Then we have
$ m(\bm y \mid \bm \gamma_1,\bm x) =  m(\bm y \mid \bm \gamma_1) \times m(\bm y\mid \x)$, where
$$m(\bm y\mid \x) = \frac{1}{\sqrt{1+g}}
 \exp\left\{\frac{g}{2\sigma^2(1+g)}\times
 \frac{[{\bm y}'(\bm I- {\bm P}_{{\bm B}_{\gamma_1}}) \bm x]^2}{{\bm x}'(\bm I- {\bm P}_{{\bm B}_{\gamma_1}}) \bm x)} \right\} $$
 and  ${\bm P}_{{\bm B}_{\gamma_1}} = {\bm B}_{\gamma_1} ({\bm B}_{\gamma_1}' {\bm B}_{\gamma_1})^{-1}{\bm B}_{\gamma_1}'$.
Note that, if $\bm x$ is orthogonal to $\bm B$, then
\begin{equation}
\label{eq.orthmarg}
 m(\bm y \mid \bm x) =  \frac{1}{\sqrt{1+g}}
 \exp\left\{\frac{g}{2\sigma^2(1+g)}\times
 \left[{\bm y}' \frac{\bm x}{\|\bm x\|}\right]^2\right\} \,.
 \end{equation}
 \end{lemma}

\begin{proof}
Letting $\bm P$ denote ${\bm P}_{{\bm B}_{\gamma}} $ and $\bm z = (\bm I- {\bm P}) \bm x / \sqrt{{\bm x}'(\bm I- {\bm P})\bm x}$ , this follows from the identities
\begin{eqnarray*}
(\bm I + g \bm P + g \bm z {\bm z}') ^{-1} &=& (\bm I + g \bm P) ^{-1} - \frac{(\bm I + g \bm P) ^{-1} \bm z {\bm z}' (\bm I + g \bm P) ^{-1}}{[g^{-1} + {\bm z}' (\bm I + g \bm P) ^{-1} \bm z]} \,,\\
|\bm I + g \bm P + g \bm z {\bm z}'| &=& |\bm I + g \bm P| \, (1+ g{\bm z}' (\bm I + g \bm P) ^{-1} \bm z), \, \\
(\bm I + g \bm P)(\bm I - \bm P) &=& (\bm I - \bm P)\,, \, \mbox{so that} \,\, (\bm I + g \bm P)^{-1}(\bm I - \bm P)= (\bm I - \bm P) \,.
\qedhere\end{eqnarray*}

\end{proof}

\begin{remark}
 If $\bm x$ is orthogonal to $\bm B$, the corresponding Bayes estimates are just the usual $g$-prior posterior means
 \begin{equation}
 \frac{g}{1+g}\left(\bm {\widehat \beta}^{MLE}_{\bm \gamma_1} ,  \frac{{\bm x}' \bm y}{\|\bm x\|^2}  \right) \,.
  \end{equation}
Moreover, adding at least one of the near-identical predictors multiplies the limiting marginal likelihood by a constant factor that does not depend on the number of copies.
\end{remark}

\subsection{Dimensional predictive matching}
 As a first application of Lemma \ref{lemma:marg}, we note that the (limiting) marginal likelihood under the $g$-prior is the same, no matter how many replicates of $\bm x$ are in the model.
 This property can be regarded as a variant of  {\sl dimensional predictive matching}, one of the desiderata relevant for the development of objective model selection priors (\cite{bayarri_etal}).
This type of predictive matching across dimensions is, however,  new in the sense that the matching holds for all training samples, not only the minimal ones.

 \begin{corollary}
Mixtures of $g$-priors are {\sl dimensional predictive matching} in the sense that the limiting marginal likelihood of all models within the model collective is the same, provided that the mixing distribution over $g$ is the same across all models.
\end{corollary}
\begin{proof} Follows directly from Lemma \ref{lemma:marg}. \end{proof}




In contrast, it is of interest to look at what happens with an alternative prior for $\bm \beta$ such as a $N(\bm 0,\bm I)$ prior. If a model has $j$ near-replicates of $\bm x$, the effective parameter for $\bm x$ in that model is the sum of the $j$ $\beta$'s, which will each have a $N(0,j)$ prior. So the marginal likelihoods will depend strongly on the number of replicates, even though there is no difference in the models.

\subsection{When all non-duplicated covariates are orthogonal}
To get insights into the behavior of the median probability model for correlated predictors, we consider {an instructive} example obtained by  setting $\epsilon=0$ and $\B'\B=\bm I$  in \eqref{Xmatrix}.
 In particular, we will be working with an orthogonal design that has been augmented with multiple copies of one predictor
 \begin{equation}\label{X_aug}
 \X_{n\times (p+k)}=[\x_1,\dots,\x_p,\underbrace{\x,\dots,\x}_{k}],
\end{equation}
where $\x_1,\dots,\x_p,\x$ are orthonormal. A few points are made with this  toy example. First, we want to characterize the optimal predictive model and  generalize the MPM rule when the designs
  have blocks of (nearly) identical predictors.
Second, we want to understand  how close to the optimal predictive model  the MPM actually is  in this limiting case. Third, we want to highlight the benefits of the $g$-prior correlation structure.
 We denote by $z=\x'\y$, $z_i=\x_i'\y$ for $i=1,\dots,p$ and $\z=(\z_1',\z_2')'$, where  $\z_{1}=(z_1,\dots,z_p)'$ and $\z_2=z\bm{1}_{k}$. We will again split the variable inclusion indicators into two groups $\bg=(\bg_1',\bg_2')'\in\{0,1\}^{p+k}$, where
     $\bg_1$ is attached to the first $p$ and $\bg_2$ to the last $k$ predictors.
 To begin, we assume the
 generalized $g$-prior  on regression coefficients, given the model $\bg$,
\begin{equation}\label{gprior}
\b_{\bg}\sim\mathcal{N}_{|\bg|}\left(\bm{0},g\,\sigma^2(\X_{|\bg|}'\X_{|\bg|})^+\right),
\end{equation}
where $(\X_{|\bg|}'\X_{|\bg|})^+$ is the Moore-Penrose pseudo-inverse.
The following lemma characterizes the optimal predictive model under \eqref{gprior} and \eqref{X_aug}.

\begin{lemma}\label{lemma:copies}
Consider the model \eqref{model} where $\X$  satisfies \eqref{X_aug} and where $\x_1,\dots,\x_p,\x$ are orthonormal. Under the prior \eqref{gprior}, any model $\bg^o=(\bg_1^{o\prime},\bg_2^{o\prime})'$ that satisfies
 \begin{align}
 \gamma_{1i}^o&=1\quad\text{iff}\quad \pi(\gamma_{1i}=1\C\Y)>0.5,\quad i=1,\dots,p\label{opt1}\\
 |\bg_2^o|&\geq1\quad\text{iff}\quad \pi(\bg_2\neq \bm{0}\C\Y)>0.5\label{opt2}
 \end{align}
is predictive optimal.
\end{lemma}
\begin{proof}
Due to the block-diagonal matrix $\X'\X=\left(\begin{matrix}\mathrm{I}_{p}& \bm{0}_{p\times k}\\ \bm{0}_{k\times p}& \bm{1}_k\bm{1}_k'\end{matrix}\right)$, the posterior mean under the non-null model $\bg$ satisfies
$$
\H_{\bg}\widehat{\b}_{\bg}= \frac{g}{(1+g)}\left(
\begin{matrix}
\mathrm{diag}\{\bg_1\}&\bm{0}\\
\bm{0}&\frac{1}{|\bg_2|}\mathrm{diag}\{\bg_2\}
\end{matrix}
\right)
\z
$$
The overall posterior mean $\bar\b=\E (\b\C\Y)=\sum\limits_{\bg}\pi(\bg\C\Y)\H_{\bg}\widehat{\b}_{\bg}$
then satisfies $\H_{\bg}\widehat{\b}_{\bg}-\bar{\b}= $
$$
\frac{g}{(1+g)}\left(
\begin{matrix}
\mathrm{diag}\{\bg_1-\E[\bg_1\C\Y]\}&\bm{0}\\
\bm{0}&\mathrm{diag}\left\{\frac{\bg_2}{|\bg_2|}-\E\left[\frac{\bg_2}{|\bg_2|}\C\Y,\bg_2\neq \bm{0}\right]\pi(\bg_2\neq\bm{0}) \right\}
\end{matrix}\right)\z
$$
The optimal predictive model minimizes $R(\bg)$ defined in \eqref{Rcrit}.
Due to the fact that $\Q$ is block-diagonal, the criterion $R(\bg)$ separates into two parts, one involving the first $p$ independent variables and the second involving the $k$ identical copies. In particular,
$
R(\bg)=R_1(\bg_1)+R_2(\bg_2)
$
where
\begin{align}
 R_1(\bg_1)=& \frac{g^2}{(1+g)^2}\sum_{i=1}^pz_i^2(\gamma_{1i}-\E[\gamma_{1i}\C\Y])^2\label{R1}\\
R_2(\bg_2)=&\frac{g^2 z^2}{(1+g)^2}\left\{\left[1-\pi(\bg_2\neq \bm{0}\C\Y)\right]^2\mathbb{I}(\bg_2\neq \bm{0})\right.\notag\\
&\left.+\pi(\bg_2\neq \bm{0}\C\Y)^2\mathbb{I}(\bg_2=\bm{0}) \right\}.\label{R2}
\end{align}
The statement then follows from \eqref{R1} and \eqref{R2}.
With duplicate columns, the optimal predictive model $\bg^o\equiv\arg\min\limits_{\bg}R(\bg)$ is not unique. Any model $\bg^o=(\bg_1^{o\prime},\bg_2^{o\prime})'$ defined through \eqref{opt1} and \eqref{opt2}
will minimize the criterion $R(\bg)$.
 \end{proof}

The last $k$ variables in the optimal predictive model thus act jointly as one variable, where the decision to include $\x$ is based on a {\sl joint} posterior probability $\pi(\bg_2\neq \bm{0}\C\Y)$. This intuitively appealing treatment of $\x$ is
an elegant byproduct of the $g$-prior. We will see in the next section that such clustered inclusion no longer occurs in the optimal predictive model under independent product priors.
The risk of the optimal model is
\begin{align*}
R(\bg^o)=&\frac{g^2}{(1+g)^2}\sum_{i=1}^pz_i^2\min\{\E[\gamma_{1i}\C\Y],1-\E[\gamma_{1i}\C\Y]\}^2\\
&+\frac{g^2z^2}{(1+g)^2}\min\{\pi[\bg_2\neq \bm{0}\C\Y],1-\pi[\bg_2\neq \bm{0}\C\Y]\}^2
\end{align*}
Contrastingly,
recall that the median probability model $\bg^{MP}=(\bg_1^{MP\prime},\bg_2^{MP\prime})$ is defined through
$$\gamma^{MP}_i=1\quad \text{iff}\quad\pi(\gamma_i=1\C\Y)>0.5\quad\text{for}\quad i=1,\dots, p+k.$$
The  median probability model  $\bg^{MP}$ thus behaves as the optimal model $\bg^o$ for the first $p$ variables.  For the $k$ duplicate copies, however,  $\bg^{MP}_2$ consists of either {\sl all ones or all zeros}. The MP rule correctly recognizes that the decision to include $\x$ is ultimately dichotomous: either all $\x$'s in or all $\x$'s out.
Moreover, when the median model decides ``all in", it  {\sl will be predictive optimal}. Indeed,  $\pi(\gamma^{MP}_{2i}=1\C\Y)>1/2$ for $i\in\{1,\dots, k\}$ implies $\pi(\bg_2=\bm{0}\C\Y)<1/2$.
The MP model will  deviate from the optimal model only when  $\pi(\gamma^{MP}_{2i}=1\C\Y)<1/2$ and $\pi(\bg_2\neq\bm{0}\C\Y)>1/2$ in which case
\begin{align*}
\frac{R(\bg^{MP})-R(\bg^o)}{R(\bg^o)}&=\frac{R_2(\bg^{MP}_2)-R_2(\bg^{o}_2) }{R(\bg^o)}=\frac{g^2z^2[1-2\pi(\bg_2\neq \bm{0}\C\Y)]}{(1+g)^2[R_1(\bg_1^o)+R_2(\bg^{o}_2)]}.
\end{align*}
The term $R_1(\bg_1^o)$ in \eqref{R1} can be quite large when $p$ is large, implying that the relative risk can be quite small. The MP model is thus not too far away from the optimal predictive model in this scenario.

Several conclusions can be drawn from our analysis of this toy example. First, Lemma \ref{lemma:copies} shows that, in the presence of  perfect correlation, it is the {\sl joint inclusion} rather than marginal inclusion probabilities that guide the optimal predictive model selection. We will elaborate on this property in Section \ref{sec:equicor}, showing that optimal predictive model can be characterized using both posterior means and covariances of $\bg$ (in equicorrelated designs). Second, the clone variables ultimately act  collectively as one variable, which has important implications on the assignment of prior model probabilities. We will elaborate on this important issue in Section \ref{sec:prior_prob}, \ref{sec:posterior_prob} and \ref{sec:dillution}.
Third, purely from a predictive point of view, all  models in the model collective (including at least one $\x$) are equivalent. The $g$-prior here appears to be egualitarian in the sense that it (rightly) treats all these models equally. This property is not retained under  independent product priors, as shown below.

\begin{remark}(Independent Product Priors) Let us replace \eqref{gprior} with an independent prior covariance structure
\begin{equation}\label{indep}
\b_{\bg}\sim\mathcal{N}_{|\bg|}\left(\bm{0},\sigma^2\mathrm{I}_{|\bg|}\right),
\end{equation}
 The posterior mean $\bar\b$  then satisfies $H_{\bg}\widehat{\b}_{\bg}-\bar{\b}=$
{\small$$
\left(
\begin{matrix}
\mathrm{diag}\{\bg_1-\E[\bg_1\C\Y]\}&\bm{0}\\
\bm{0}& \mathrm{diag}\{\bg_2-\E[\bg_2\C\Y]\}-\frac{1}{1+|\bg_2|}\bg_2\bg_2'+\E\left[\frac{1}{1+|\bg_2|}\bg_2\bg_2'\biggl|\,\Y\right]
\end{matrix}\right)\z
$$}
The criterion $R(\bg)=R_1^\star(\bg_1)+R_2^\star(\bg_2)$ again separates into two parts, where
\begin{align}
R_1^\star(\bg_1)&= \sum_{i=1}^pz_i^2(\gamma_{1i}-\E[\gamma_{1i}\C\Y])^2\\
R_2^\star(\bg_2)&=z^2\left[\frac{|\bg_2|}{1+|\bg_2|} - \E\left(\frac{|\bg_2|}{1+|\bg_2|}\,\biggl|\,\Y\right)\right]^2 \,.
\end{align}
The optimal predictive model for the last $k$ variables now has a bit less intuitive explanation. It  consists of any collection of variables of size $|\bg_2^o|$ for which $|\bg_2^o|/[1+|\bg_2|^o]$ is as close as possible to the posterior mean of $|\bg_2|/[1+|\bg_2|]$. It is worthwhile to note that this {\sl does not need to be} the null or the full model. For instance, one can show that $\bg_2^o=\bm{0}$ when $\E[|\bg_2|\C\Y]<1/3$ (or more generally when $\pi(\bg_2=0\C\Y)>2/3)$).
The full model $\bg_2^o=\bm{1}$ will be optimal when $\E[|\bg_2|\C\Y]>k-\frac{k+1}{2k+1}$. Besides these narrow situations, the optimal model $\bg^{o}_2$ will   have a nontrivial size (other than $0$ or $k$).
The median probability model will still maintain the dichotomy by either including all or none of the $x$'s. However, contrary to the $g$-prior it is not guaranteed to be ``optimal" when, for instance, $\bg_2^{MP}=\bm{1}$. 
It seems that the mission of  the optimal model under the independent prior is a bit obscured. It is not obvious why models in the same model collective should be treated differentially and ranked based on their size. The independence prior correlation structure thus induces what seems as an arbitrary identifiability constraint.

\end{remark}

\subsection{Prior Probabilities on Model Collectives}\label{sec:prior_prob}
It has been now standard to assume that each model of dimension $|\bg|$ has an equal prior probability
\begin{equation}\label{eq:prior_equal}
\pi(\bg)=\pi(|\bg|) / {p+k \choose |\bg|},
\end{equation} with $\pi(|\bg|)$ being the prior probability (usually $1/(p+k+1)$) of the collection of models of dimension $|\bg|$.
One of the observations from Lemma \ref{lemma:copies} is that it is the {\sl aggregate} posterior probability $\pi(\bg_2\neq 0\C\Y)$ rather than individual inclusion probabilities $\pi(\gamma_{2i}=1\C\Y)$ that drive the optimal predictive model $\bg^o_2$ in our collinear design. Thereby, it is natural to inspect the aggregate prior probability $\pi(\bg_2\neq \bm{0})$.
We will be using the notion of model collectives  introduced earlier in Definition \ref{def:aggregate}. The number of models of size $j > |\bm \gamma_1|$ in the model collective $M_{\bg_1, \bm x}$ is ${k \choose j-|\bm \gamma_1|}$, so that the prior probability of the model collective $M_{\bg_1, \bm x}$ is
\begin{equation}
\label{eq.probmass}
 \pi(M_{\bg_1, \bm x}) = \sum_{j=|\bm \gamma_1|+1}^{|\bm \gamma_1|+k} \frac{\pi(j)}{{p+k \choose j}} {k \choose j-|\bm \gamma_1|} \,.
 \end{equation}
We investigate the prior probability of the model collective under two usual choices: fixed prior inclusion probability (the separable case) and the random (non-separable) case.

\bigskip \noindent
 {
 {\em The Separable Case:}
 Suppose that all variables have a known and equal prior inclusion probability $\theta=\pi(\gamma_j=1\C\theta)$ for $j=1,\dots, p+k$. Then the probability of the model aggregate, given $\theta$, is
\begin{align}
\label{eq.probmass2}
 \pi(M_{\bg_1, \bm x};\theta) &= \theta^{|\bg_1|}(1-\theta)^{p-|\bg_1|}  \sum_{j=1}^{k} \theta^{j}(1-\theta)^{k-j} {k \choose j}  \\
 &=\theta^{|\bg_1|}(1-\theta)^{p-|\bg_1|}\left[1-(1-\theta)^{k}\right]
 \end{align}
and the  prior probability  of the ``null model" $M_{\bg_1, \bm 0}$ (not including any of the correlated variables) is
\begin{equation}
\label{eq.probmass2}
 \pi(M_{\bg_1, \bm 0};\theta)=\theta^{|\bg_1|}(1-\theta)^{p-|\bg_1|}(1-\theta)^{k}.
 \end{equation}
 The ratio satisfies
 $$
 \frac{ \pi(M_{\bg_1, \bm x};\theta)}{ \pi(M_{\bg_1, \bm 0};\theta)}=\left[\left(\frac{1}{1-\theta}\right)^k-1\right].
 $$
 This analysis reveals a rather spurious property of the separable prior: regardless of the choice $\theta\in (0,1)$, the model aggregate  $M_{\bg_1, \bm x}$ will always have a higher prior probability than the model $M_{\bg_1, \bm 0}$ without any $\x$ in it. Such a preferential treatment for $\x$ is generally unwanted.  We illustrate this issue with the uniform model prior (obtained with $\theta=0.5$) which is still widely used in practice.

\smallskip \noindent
With fixed $\theta=0.5$,  all models have an equal prior probability of $2^{-(p+k)}$. The number of models in the collective $M_{\bg_1, \bm x}$ is
$2^k-1$, and so
$$\pi(M_{\bg_1, \bm x};1/2) = (2^k-1)2^{-(p+k)} = (2^k-1) \, \pi(M_{\bg_1, \bm 0};1/2).$$
The collective can thus have much more prior probability than $M_{\bg_1, \bm 0}$. Furthermore, the marginal prior probability of inclusion of $\x$ is $\sum_{\bm \gamma_1} \pi(M_{\bg_1, \bm x};1/2)= 1- 2^{-k}$. Hence, if $k$ is even moderately large, the prior mass is concentrated on the models which include $\bm x$ as a covariate, and the posterior mass will almost certainly also be concentrated on those models. The model-averaged $\bar {\bm \beta}$ will reflect this, essentially only including models that have $\bm x$ as a covariate.

{\em Beta-binomial Prior:}  It is generally acknowledged \citep{leysteel2007,cui2008empirical,Scott_berger_2010} that assigning equal prior probability to all models is a poor choice, since it  does not adjust for the multiple testing that is effectively being done in variable selection. The common alternative (which does adjust for multiple testing), is replace the separable prior with $\theta\sim\mathcal{B}(a,b)$. Then the prior probability of the model aggregate satisfies
\begin{align*}
\pi(M_{\bg_1, \bm x})&=\int_0^1 \pi(M_{\bg_1, \bm x};\theta)\d\pi(\theta)
=\int_0^1\theta^{|\bg_1|+a-1}(1-\theta)^{p-|\bg_1|+b-1}\left[1-(1-\theta)^{k}\right]\d\theta\\
 &=\mathcal{B}(|\bg_1|+a,p-|\bg_1|+b)-\mathcal{B}(|\bg_1|+a,p+k-|\bg_1|+b)\\
 &=\mathcal{B}(|\bg_1|+a,p+k-|\bg_1|+b)\left[\prod_{j=1}^k\frac{a+b+p+j-1}{p+b-|\bg_1|+j-1}-1\right]
 \end{align*}
and
\begin{equation}
 \pi(M_{\bg_1, \bm 0})=\mathcal{B}(|\bg_1|+a,p+k-|\bg_1|+b).
  \end{equation}
Then
\begin{equation}
\label{eq.priorodds}
\frac{\pi(M_{\bg_1, \bm x})}{\pi(M_{\bg_1, \bm 0})}=\left[\prod_{j=1}^k\left(1+\frac{|\bg_1|+a}{p+b-|\bg_1|+j-1}\right)-1\right].
\end{equation}
This ratio is guaranteed to be smaller than under the separable case with a fixed $\theta$ when $|\bg_1|<(p+b)\theta-a(1-\theta)$. This suggests that the beta-binomial prior can potentially cope better with variable redundancy. We elaborate on this point in the next section. In the forthcoming Lemma \ref{lemma}, we provide an approximation to \eqref{eq.priorodds} as $p$ gets large.

\subsection{Posterior inclusion probabilities}\label{sec:posterior_prob}
In the previous section, we have shown that equal prior model probabilities can be problematic because each model collective $M_{\bg_1,\x}$ receives much more prior mass relative to
$M_{\bg_1,\bm 0}$, essentially forcing the inclusion of $\x$. Going further, we show how this is reflected in the posterior inclusion probabilities.

\begin{lemma}\label{post_inclusion}
Consider the model \eqref{model}, where $\X$ satisfies \eqref{X_aug} and where $\x_1,\dots,\x_p,\x$ are orthonormal. Denote by $z= {\bm y}' \bm x$ and consider the prior \eqref{gprior} with $g=n$ and equal prior model probabilities $\pi(\bg)=1/2^{p+k}$. Then we have
\begin{equation}
\label{eq.equalprobcoll}
\pi(\bg_2\neq \bm 0\C \Y)>1/2\quad\text{iff}\quad z ^2 > \log \left(\frac{\sqrt{1+n}}{2^k-1} \right) 2 \sigma^2 \left(1+\frac{1}{n}\right) \,.
\end{equation}
\end{lemma}
\begin{proof}
The posterior  probability of joint inclusion $\pi(\bg_2\neq \bm 0\C \Y)$ (noting that $\pi(M_{\bg_1, \bm x})$ and $\pi(M_{\bg_1, \bm 0})$
depend only on $| \bm \gamma_1|$) equals
\begin{eqnarray}
\label{eq.collective-inclusion}
\pi(\bg_2\neq \bm 0\C \Y) &=& \frac{ \sum_{\bg_1} \pi(M_{\bg_1, \bm x})  m(\bm y \mid \bg_1,\bm x)}
{\sum_{\bg_1} \pi(M_{\bg_1, \bm x})  m(\bm y \mid \bg_1,\bm x)+ \sum_{\bg_1} \pi(M_{\bg_1, \bm 0})  m(\bm y \mid \bg_1,\bm 0)} \nonumber \\
&=& \frac{ \sum_{\bg_1} \pi(M_{\bg_1, \bm x})  m(\bm y \mid \bg_1)m(\bm y\mid\bm x) }
{ \sum_{\bg_1} \pi(M_{\bg_1, \bm x})  m(\bm y \mid \bg_1)m(\bm y\mid\bm x) + \sum_{\bg_1} \pi(M_{\bg_1, \bm 0})  m(\bm y \mid \bg_1)  } \nonumber \\
&=& \frac{ m (\bm y \mid \bm x)\sum_{i=0}^p \pi_{i,\x}^\star}
{\sum_{i=0}^p\left[m (\bm y \mid \bm x)\pi_{i,\x}^\star+\pi_{i,\bm 0}^\star\right]} \,,
\end{eqnarray}
with  $\pi_{i,\x}^\star= \sum_{\bg_1:|\bg_1|=i} m(\bm y \mid \bg_1) \,\, \pi(M_{\bg_1, \bm x})$ and  $\pi_{i,\bm 0}^\star= \sum_{\bg_1:|\bg_1|=i} m(\bm y \mid \bg_1) \,\, \pi(M_{\bg_1, \bm 0})$,
and where
$$m (\bm y \mid \bm x) = \frac{1}{\sqrt{1+g}}
 \exp{\left\{\frac{g}{2\sigma^2(1+g)}\times
 \left[{\bm y}' {\bm x}\right]^2 \right\}}
 $$
 and  $m(\y\C\bg_1,\x)=m(\y\C\bg_1)m(\y\C\x)$ was defined in \eqref{m_function}.
When each model has equal prior probability,  this simplifies to
$$
\pi(\bg_2\neq \bm 0\C \Y)  = \frac{(2^k-1) m (\bm y \mid \bm x)}{1+(2^k-1) m (\bm y \mid \bm x)} \,.
$$
With the usual choice $g=n$,  it follows that
$\pi(\bg_2\neq \bm 0\C \Y)  > 0.5$ iff
$
z^2 > \log \left(\frac{\sqrt{1+n}}{2^k-1} \right) 2 \sigma^2 (1+\frac{1}{n}).
$
\end{proof}
From Lemma \ref{post_inclusion} it follows that
 the optimal predictive model (characterized in Lemma \ref{lemma:copies}) will include $\x$ if the number of duplicates  $k$ is large enough, even when $\bm x$ has a small effect ($z$ is small).
Thus, the choice of equal prior model probabilities for optimal predictive model, in the face of replicate covariates, is potentially quite problematical.
If one is only developing a model for prediction in such a situation, such forced inclusion of $\bm x$ is probably suboptimal, but it is only one covariate and so will not typically have a large effect, unless only very small models have significant posterior probability.  For prediction, one could presumably do somewhat better by only considering the first $p+1$ variables in the model uncertainty problem, finding the model averaged $\bar {\bm \beta}$ for this subset of variables.

\
This statement at first seems odd, because we `know' the model averaged answer in the original problem is optimal from a Bayesian perspective. But that optimality is from the internal Bayesian perspective, assuming we believe that the original model space and assignment of prior probabilities is correct. If we really believed -- e.g., that any of $k$ highly correlated genes could be in the model with prior inclusion probabilities each equal to $1/2$ (equivalent to the assumption that all models have equal prior probability) -- then the original model averaged answer would be correct and we should include $\bm x$ in the prediction. At the other extreme, if we felt that only the collection of all $k$ genes has prior inclusion probability of $1/2$, then the result will be like the model averaged answer for the first $p+1$ variables.

To get some feel for things in the general case (non-uniform model prior), suppose $\pi_{i,\bm \x}^\star$ for some $0\leq i\leq p$ is much bigger than the others, so that (\ref{eq.collective-inclusion}) becomes
$$
\pi(\bg_2\neq \bm 0\C \Y)  \approx \frac{ \pi_{i, \bm x}^\star m (\bm y \mid \bm x) }
{ \pi_{i, \bm x}^\star m (\bm y \mid \bm x) +  \pi_{i,\bm 0}^\star }  \,.
$$
Using (\ref{eq.priorodds}), it is immediate that this is bigger than 0.5 if
\begin{equation}
\label{eq.betabinomialodds} 1< \frac{ \pi_{i, \bm x}^\star}{\pi_{i,\bm 0}^\star} \ m (\bm y \mid \bm x) = \left[\prod_{j=1}^k\left(1+\frac{i+a}{p+b-i+j-1}\right)-1\right] m (\bm y \mid \bm x) \,.
\end{equation}
The following Lemma characterizes the behavior of $\frac{ \pi_{i, \bm x}^\star}{\pi_{i,\bm 0}^\star}$ when $p$ gets large.

\begin{lemma}\label{lemma} Suppose $a$ and $b$ are integers. As $p$  gets large with $i$ fixed,
$$ \prod_{j=1}^k\left(1+\frac{i+a}{p+b-i+j-1}\right) = \left(1+\frac{k}{p}\right)^{i+a}
\left(1+\frac{C k}{p(p+k)}\right)\left(1+ O\left(\frac{1}{p^2}\right)\right) \,.
$$
where $C=-(i+a)[b-1-i+(i+a+1)/2]$ ($C=(i^2-i-2)/2$ if $a=b=1$).
To first order,
$$ \prod_{j=1}^k\left(1+\frac{i+a}{p+b-i+j-1}\right) = \left(1+\frac{k}{p}\right)^{i+a}
\left(1+ O\left(\frac{1}{p}\right)\right) \,.
$$
\end{lemma}
\begin{proof}
Defining $d=b-1$ and $c=d+a$,
\begin{eqnarray*}
 \prod_{j=1}^k\left(1+\frac{i+a}{p+b-i+j-1}\right) &=& \prod_{j=1}^k \frac{p+c+j}{p+d-i+j} = \frac{(p+c+k)!/(p+c)!}{(p+d+k-i)!/(p+d-i)!}     \\
&=& \frac{(p+c+k)!/((p+d+k-i)!}{(p+c)!/(p+d-i)!} = \prod_{j=1}^{i+a} \frac{p+d+k-i+j}{p+d-i+j} \\
&=& \left(\frac{p+k}{p}\right)^{i+a}\prod_{j=1}^{i+a} \frac{\left(1+ \frac{d-i+j}{p+k}\right)}{\left(1+ \frac{d-i+j}{p}\right)} \,.
\end{eqnarray*}
The first order result follows immediately and the second order result follows from expanding the products in the last term above.
\end{proof}

Utilization of the first order term in  (\ref{eq.betabinomialodds}), and again choosing $g=n$ and assuming $\|\bm x\| = 1$, yields
that the collective has posterior inclusion probability greater than 0.5 if
$$ z^2 > \log \left(\frac{\sqrt{1+n}}{(1+k/p)^{[i+a]}-1} \right) 2 \sigma^2 \left(1+\frac{1}{n}\right) \,.
$$
Note that this is much less likely to be satisfied than (\ref{eq.equalprobcoll}), when $k$ grows, since $(1+k/p)^{[i+a]}$ is then much smaller than $2^k$; thus having many duplicate $\bm x$'s does not ensure that $\bm x$ will be included in the model, as it was in the equal model probability case.


\subsection{The Dilution Problem}\label{sec:dillution}

{Sets of predictors, which are highly correlated with each other, become proxies for one another in our linear model (1).  This quickly leads to an excess of redundant models, each of which is distinguished only by including a different subset of these.  To prevent this cluster of redundant models from accumulating too much posterior probability, dilution priors may be considered, \cite{george2010}.  Such priors first assign a reasonable amount of prior mass to the entire cluster, and then dilute this mass uniformly across all subset models within this cluster.

For example,} to smear out the prior aggregation on  $M_{\bg_1,\bm x}$, one might like to consider different inclusion probabilities.
 Let $[\bm{x}_1,\dots,\bm{x}_p]$ have a prior inclusion probability $\theta_1$ and each of the $\x$ clones have a prior inclusion probability $\theta_2$.
 With
 \begin{equation}\label{theta2}
 \theta_2=1- (1-\theta_1)^{1/k}
\end{equation}
 we have
\begin{equation}
 \pi(M_{\bm \gamma_1, \bm x}) =\theta_1^{|\bg_1|}(1-\theta_1)^{p-|\bg_1|}\left[1-(1-\theta_2)^{k}\right]=\theta_1^{|\bg_1|+1}(1-\theta_1)^{p-|\bg_1|}.
 \end{equation}
 and
 $$
  \pi(M_{\bm \gamma_1, \bm 0})  =\theta_1^{|\bg_1|}(1-\theta_1)^{p-|\bg_1|+1}.
 $$
Assuming \eqref{theta2}, variables with correlated copies have smaller inclusion probabilities (the more copies, the smaller the probability). This may correct the imbalance between  $ \pi(M_{\bm \gamma_1, \bm x}) $ and $ \pi(M_{\bm \gamma_1, \bm 0}) $ by treating the  multiple copies of $\x$ essentially as one variable. This prior allocation would put $\x$ on an equal footing with $\x_1,\dots,\x_p$
in the optimal predictive model rule (based on $\pi(\bg_2\neq \bm 0\C\Y)$), but would disadvantage $\x$ in the median probability model.
From our considerations above, it would seem that  there is a fix to the dilution problem in our synthetic example (with clone $\x$'s). However,   general recommendations for other correlation patterns are far less clear.

\section{ The Case of Two Covariates}

\subsection{The geometric representation}

The situations analyzed in previous sections may also be considered from a geometric perspective.
Define
$\mbox{\boldmath $\alpha$}_{\mbox{\footnotesize\boldmath $\gamma$}}$ as the projection of $\Y$ onto the space
spanned by the columns of ${\bf X}_{\mbox{\footnotesize\boldmath
$\gamma$}}$,
$\mbox{\boldmath $\alpha$}_{\mbox{\footnotesize\boldmath
$\gamma$}}={\bf X} \, {\bf H}_{\mbox{\footnotesize\boldmath $\gamma$}} \,
\widehat{\mbox{\boldmath $\beta$}}_{\mbox{\footnotesize\boldmath
$\gamma$}}= {\bf X}_{\mbox{\footnotesize\boldmath $\gamma$}}({\bf
X}_{\mbox{\footnotesize\boldmath $\gamma$}}^{'} {\bf
X}_{\mbox{\footnotesize\boldmath $\gamma$}})^{-1} \, {\bf
X}_{\mbox{\footnotesize\boldmath $\gamma$}}^{'} \, \mbox{\boldmath
$Y$}$,  and $\bar{\mbox{\boldmath
$\alpha$}}={\bf X}\,\bar{\mbox{\boldmath $\beta$}}=
\sum_{\mbox{\footnotesize\boldmath $\gamma$}}
p_{\mbox{\footnotesize\boldmath $\gamma$}} \, {\bf X}\, {\bf
H}_{\mbox{\footnotesize\boldmath $\gamma$}} \,
\widehat{\mbox{\boldmath $\beta$}}_{\mbox{\footnotesize\boldmath
$\gamma$}}$, where $p_{\bg}=\pi(\bg\C\Y)$.
The expected posterior loss
(\ref{Rcrit}) to be minimized may be written as
\begin{displaymath}
R(\bg)= \left({\bm{\alpha}}_{\bg}-\bar{\bm{\alpha}}\right)'\left({\bm{\alpha}}_{\bg}-\bar{\bm{\alpha}} \right).
\end{displaymath}

This implies that the preferred model will be the one whose
corresponding $\mbox{\boldmath
$\alpha$}_{\mbox{\footnotesize\boldmath $\gamma$}}$ is nearest to
$\bar{\mbox{\boldmath $\alpha$}}$ in terms of Euclidean distance.

To geometrically formulate the predictive problem, each model $M_{\mbox{\footnotesize\boldmath $\gamma$}}$ may be represented by
the point $\mbox{\boldmath $\alpha$}_{\mbox{\footnotesize\boldmath
$\gamma$}}$ and the set of models becomes a collection of
points in $q$-dimensional space. The convex hull of these points
is a polygon representing the set of possible model averaged
estimates $\bar{\mbox{\boldmath $\alpha$}}$, as the
$p_{\mbox{\footnotesize\boldmath $\gamma$}}$ vary over their range.
Any point in this polygon is a possible optimal predictive
model, depending on $p_{\mbox{\footnotesize\boldmath $\gamma$}}$'s.
The goal is to geometrically characterize when each single
model is optimal, given that a single model must be used.

Consider the simple situation in which we have two covariates
$x_1$ and $x_2$ and four possible models: \bdm M_{10}:\{x_1\}
\hspace{2cm} M_{01}:\{x_2\} \hspace{2cm} M_{11}:\{x_1,x_2\}, \edm
and the null model $M_{00}$. These can be
represented as four points in the plane.

Depending on the sample correlation structure, the polygon region, whose
vertices are $\mbox{\boldmath $\alpha$}_{00}$, $\mbox{\boldmath $\alpha$}_{10}$, $\mbox{\boldmath $\alpha$}_{01}$ and $\mbox{\boldmath $\alpha$}_{11}$ (i.e. the convex hull of
all possible posterior means $\bar{\mbox{\boldmath $\alpha$}}$), can have
four  distinct forms. Each situation may be characterized in terms of the correlations between the variables involved, as summarized in Table \ref{tab:char}, where $r_{12}=Corr(x_1, x_2)$, $r_{1y}=Corr(x_1, Y)$ and $r_{2y}=Corr(x_2, Y)$.

\begin{table}[ht]
\begin{center}
\begin{tabular}{||c|c|c|c||}\hline\hline
$r_{12}=0$ & $r_{12}\frac{r_{1y}}{r_{2y}} <0$ & \multicolumn{2}{c|}{$r_{12}\frac{r_{1y}}{r_{2y}} >0$}\\\hline
 && $|r_{12}|< \min
\{|\frac{r_{1y}}{r_{2y}}|, |\frac{r_{2y}}{r_{1y}}|\}$ &
$|\frac{r_{1y}}{r_{2y}}|<|r_{12}|$\\\cline{3-4}
 orthogonal & case 1 & case 2 & case 3\\
\hline\hline
\end{tabular}
\end{center}
\caption{Possible scenarios in terms of $r_{12}$ and the ratio
$\frac{r_{1y}}{r_{2y}} $.} \label{tab:char}
\end{table}

\begin{figure}[!ht]
     \subfigure[Orthogonal]{
     \begin{minipage}[t]{0.22\textwidth}
       \hskip-1pc  \scalebox{0.23}{\includegraphics{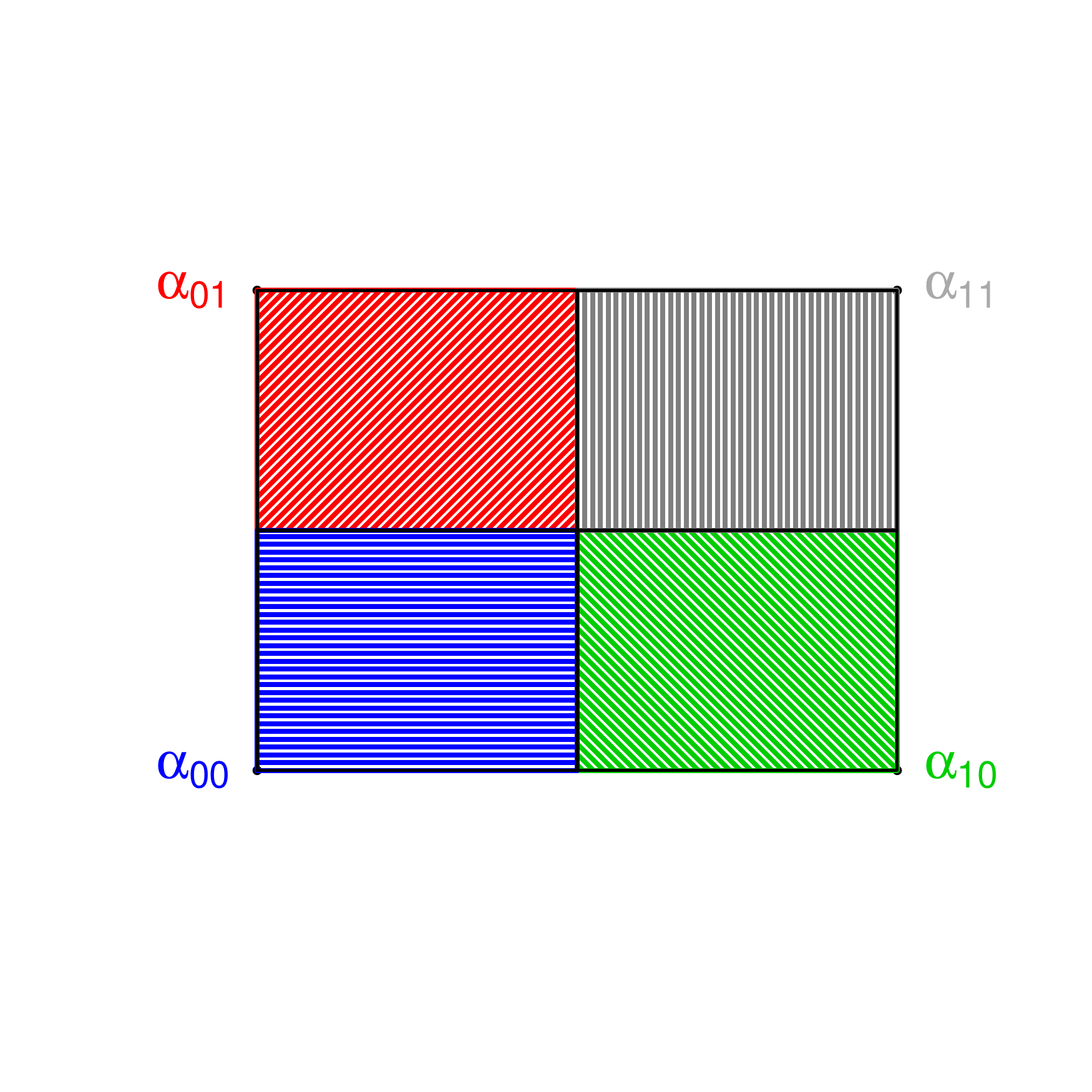}}  \label{fig:morthogonal}
    \end{minipage}}
     \subfigure[Case 1]{
    \begin{minipage}[t]{0.22\textwidth}
    \hskip-1pc  \scalebox{0.23}{\includegraphics{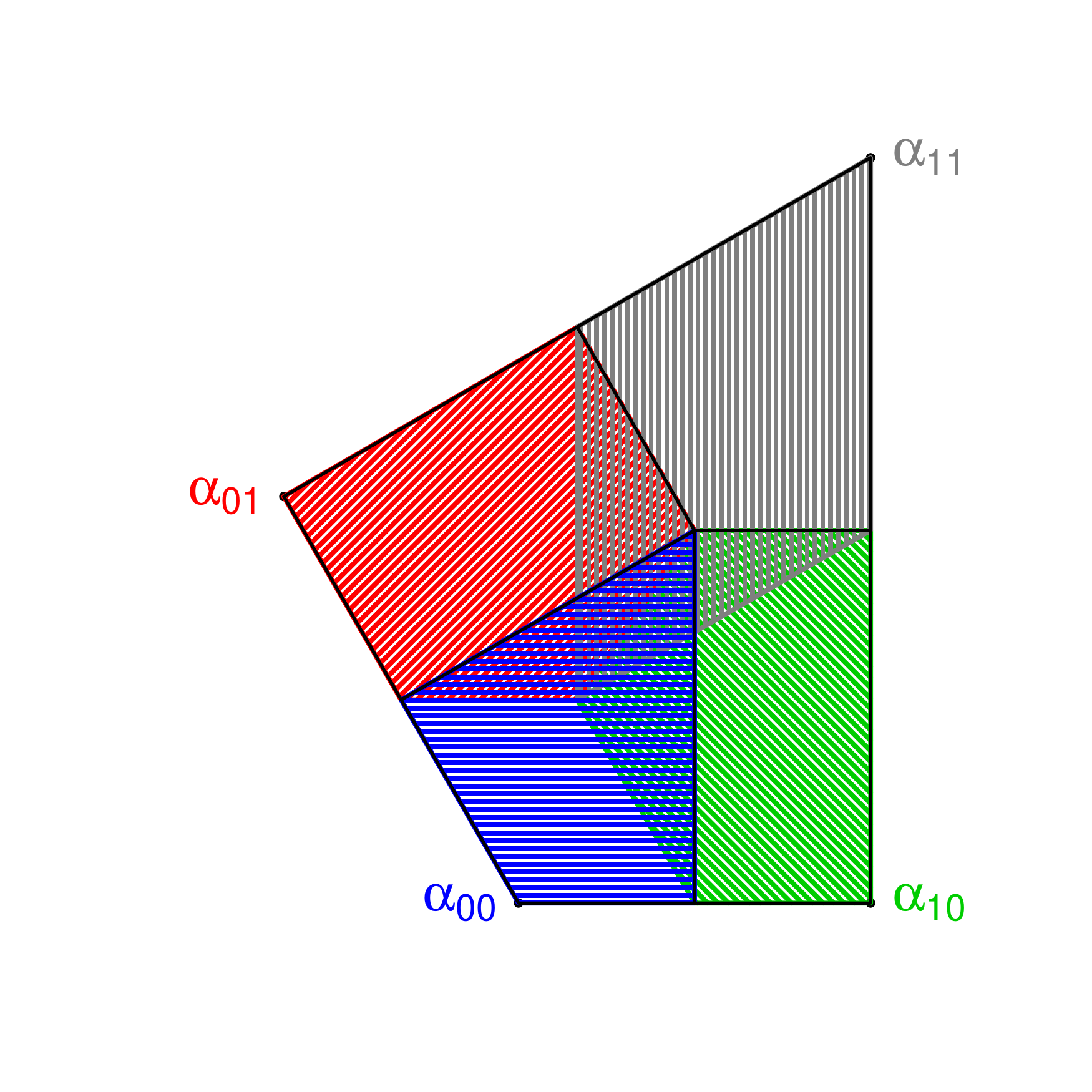}}     \label{fig:mcase1}
    \end{minipage}}
     \subfigure[Case 2]{
    \begin{minipage}[t]{0.22\textwidth}
     \hskip-1pc \scalebox{0.23}{\includegraphics{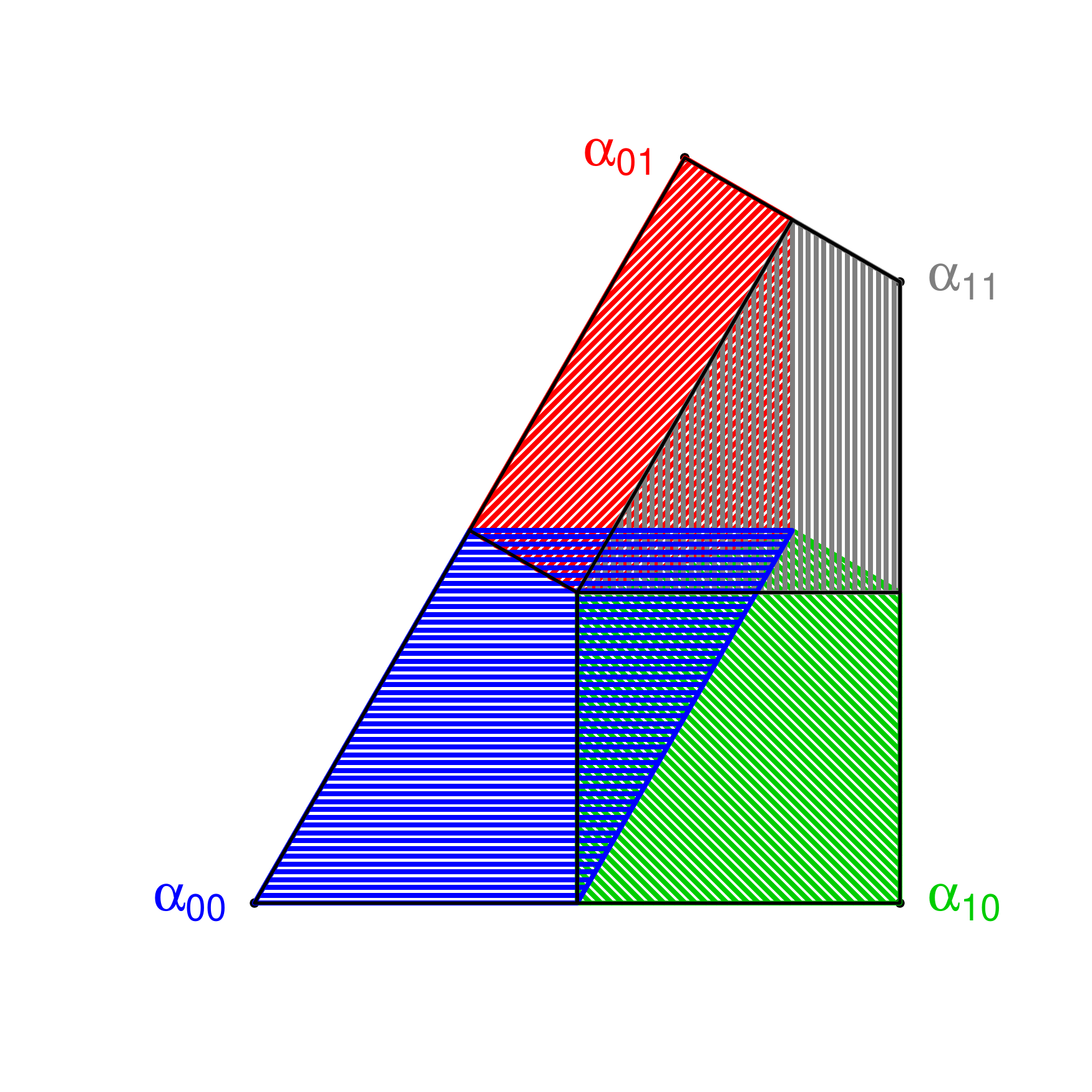}}   \label{fig:mcase2}
    \end{minipage}}
    \subfigure[Case 3]{
    \begin{minipage}[t]{0.22\textwidth}
     \hskip-1pc \scalebox{0.23}{\includegraphics{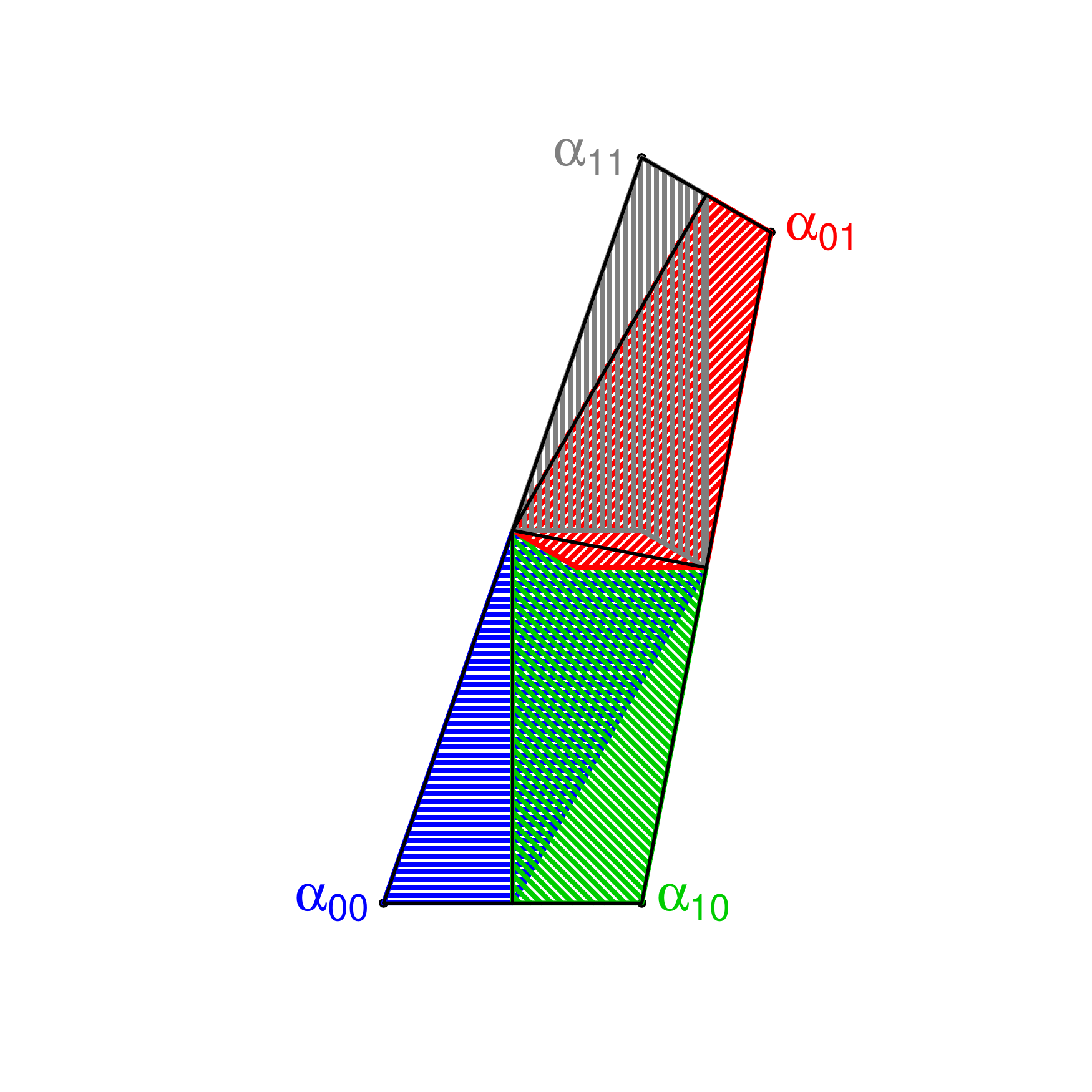}}   \label{fig:mcase3}
    \end{minipage}}
    \caption{The median posterior probability model}    \label{fig:med}
\end{figure}

In Figure \ref{fig:med} the four forms are plotted for the case $|r_{12}|=0.5.$ (Ignore the colors for now.)
In particular, the values of the correlations here are:

\begin{tabular}{llll}
Case 1 & $r_{12}=-0.5$ & $r_{1y}=0.3$ & $r_{2y}=0.4$\\
Case 2 & $r_{12}=0.5$ & $r_{1y}=0.3$ & $r_{2y}=0.4$\\
Case 3 & $r_{12}=0.5$ & $r_{1y}=0.1$ & $r_{2y}=0.3$\\
\end{tabular}

The angles ${\bf \alpha}_{00}{\bf \alpha}_{10}{\bf
\alpha}_{11}$ and ${\bf \alpha}_{00}{\bf \alpha}_{01}{\bf
\alpha}_{11}$ are always right angles, since $({\bf
\alpha}_{10}-{\bf \alpha}_{00})^{'}({\bf \alpha}_{11}-{\bf
\alpha}_{10})=0$ and $({\bf \alpha}_{01}-{\bf
\alpha}_{00})^{'}({\bf \alpha}_{11}-{\bf \alpha}_{01})=0$ [the projection of $\mbox{\boldmath $\alpha$}_{11}$ on the line spanned by $\mbox{\boldmath $\alpha$}_{10}$ is $\mbox{\boldmath $\alpha$}_{10}$ itself and similarly for $\mbox{\boldmath $\alpha$}_{01}$].

The solid lines divide the figures into the four optimality subregions associated with the four models, namely the sets of those $\bar{\mbox{\boldmath $\alpha$}}$ which are closer to one of the $
\bm{\alpha}_{\bg}$.

The colors in Figure \ref{fig:med} indicate the regions where the average model point (the best model averaged answer) could lie if the model with the corresponding color is the median posterior probability. In the orthogonal case, the model averaged answer and model optimality regions always
coincide, i.e., the MPM is always optimal. In the other cases, this need not be so. In Case 1, for instance, the red region extends into the (blue)
null model's optimality region; thus $M_{01}$ could be the MPM, even when the null model is optimal. Likewise the green region extends into the
optimality region of the null model, and the grey region (corresponding to the full model) extends into the optimality regions of all other models. Only the null model is fine here; if the null model is the MPM, it is guaranteed to be optimal.

\subsection{Characterizations of the Optimal Model}

For the case of two correlated covariates, we obtained partial characterizations of the optimal predictive model and the median probability model. These are summarized by the ``mini-theorems" below.
\begin{theorem}(``Mini-theorems")\label{thm:mini}
Consider the model \eqref{model} with $q=2$ and the three cases described in Table \ref{tab:char}. Then the following statements hold.
\begin{enumerate}
\item In Case 1, if $M_{00}$ is the median, it is optimal.
\item In Case 2, if $M_{11}$ is the median, it is optimal.
\item In Case 1 and 3, if at most one variable has posterior inclusion probability
larger than $1/2$, $M_{11}$ cannot be optimal.
\item In Case 2 and 3, if at least one variable has posterior inclusion probability
larger than $1/2$, $M_{00}$ cannot be optimal.
\item In Cases 1 and 2, if $M_{00}$ or $M_{11}$ has posterior inclusion probability larger than
$0.5$, it is optimal.
\item In any case, if $M_{00}$ or $M_{11}$ has posterior inclusion probability larger than
$0.5$, the other cannot be optimal.
\item In Cases 3, if $M_{00}$ or $M_{11}$ has posterior inclusion probability smaller than
$0.5$ it cannot be optimal.
\end{enumerate}
\end{theorem}
\proof Appendix 1.

The motivation for developing these mini-theorems was to generate possible theorems that might hold
in general. Unfortunately, in going to the three-dimensional problem, we were able to develop
counterexamples (not shown here) to each of the mini-theorems.

\subsection{Numerical Study of the Peformance of the MPM}

{\begin{table}[h!]
\centering
\medskip
\scalebox{0.8}{ \begin{tabular}{|l|r|r|r|r|r|r|r|r|} \hline
  & M=H=O & M=H$\neq$O & M=O$\neq$H & H=O$\neq$M &H$>$M &  M$>$H & GM $\frac{R(\bg^{M})}{R(\bg^o)}$& GM $\frac{R(\bg^{H})}{R(\bg^o)}$ \\
   &   &  &  & & both $\neq$ O  & both $\neq$ O & &  \\ \hline\hline
       \multicolumn{9}{c}{Cases combined: Full model scenario}\\ \hline\hline
n=10&404 &93 &27 &3 &4$^\star$&3$^\star$ &1.08 &1.10\\
n=50&505 &8 &13 &0 &8$^\star$ &0 &1.02 &1.06 \\
n=100&512 &8 &14 &0 &0 &0 &1.02 &1.06\\
    \hline \hline
    Overall & 1421 &109 &54 &3 &12 &3 &1.03 &1.07  \\  \hline \hline
      & 88.7$\%$ & 6.8$\%$ & 3.4$\%$ & 0.2$\%$ & 0.7$\%$ & 0.2$\%$ &&\\ \hline \hline
             \multicolumn{9}{c}{Cases combined: $\beta_1 = 0$ and $\beta_2 \neq 0$ scenario}\\ \hline\hline
n=10&424 &169 &18 &2 &3$^\star$&4$^\star$ &1.104 &1.114\\
n=50&661 &57 &3 &1 &9 &0 &1.054 &1.049 \\
n=100&682 &65 &1 &0 &0 &1$^\star$ &1.045 &1.046\\
    \hline \hline
    Overall & 1767 &291 &22 &3 &12 &5 &1.065 &1.067  \\  \hline \hline
    & 84.1$\%$ & 13.9$\%$ & 1.1$\%$ & 0.1$\%$ & 0.6$\%$ & 0.2$\%$ &&\\\hline \hline
      \multicolumn{9}{c}{Cases combined: Null model scenario}\\ \hline\hline
n=10&470 &178 &50 &2 &8&7$^\star$ &1.106 &1.120\\
n=50&682 &98 &13 &3 &0 &2$^\star$ &1.039 &1.045 \\
n=100&735 &60 &5 &1 &0 &1 &1.023 &1.024\\
    \hline \hline
    Overall & 1887 &336 &68 &6 &8 &10 &1.054 &1.060  \\  \hline \hline
    & 81.6$\%$ & 14.5$\%$ & 2.9$\%$ & 0.3$\%$ & 0.3$\%$ & 0.4$\%$ &&\\\hline \hline
 \end{tabular}}
\caption{A summary of the numerical study in the two variable case. Legend: H = HPM, M = MPM, O = optimal predictive model;  M$>$H means that MPM has smaller \eqref{Rcrit} than HPM; H$>$M means that HPD has smaller \eqref{Rcrit} than MPM; and GM is the geometric mean of relative risks (to the optimal model) when MPM or HPM is not optimal.} \label{table-summary}
{$^*$ Curiously, the optimal model, $O$, is the {\em lowest} probability model in these cases. }
\end{table}}

We present a numerical study that investigates the extent to which the MPM and HPM agree, and how often they differ from the optimal predictive model. The goal was to devise a study that effectively spans the entire range of correlations that are possible and this
was easiest to do by limiting the study to the two-dimensional case. The study considered
the following correlations and sample sizes:
\begin{itemize}
\item $r_{12}$ varies over the grid \\
{\small $\{-0.9, -0.8, -0.7, -0.6, -0.5, -0.4, -0.3, -0.2, -0.1, 0.1, 0.2, 0.3, 0.4, 0.5, 0.6, 0.7, 0.8, 0.9\}$.} \\
($r_{12}=0$ was not considered because the MPM is guaranteed to be optimal then.)
\item $r_{1y}$ and $r_{2y}$ vary over ranges meant to span the range of likely data
under either the full model, one-variable model, or null model; the
description (and derivation)
of the various correlation ranges is given in Appendix 2.
\item Sample sizes $n=10,50$ and $100$ are considered.
\item Equal prior probabilities are assumed for the four models.
\item The unit information $g$-prior is used for the parameters.
\item We consider the more realistic scenario where the variance $\sigma^2$ is unknown and assigned the usual objective prior $1/\sigma^2$,
risks being computed in this setting.
\end{itemize}
The reason the numerical study is conducted in this way is to reduce the dimensionality of the problem. In terms of ordinary inputs,
one would have to deal with a study over the space of $\bm x_1$, $\bm x_2$, $\bm \beta_1$, $\bm \beta_2$, and the random error vector $\bm  \varepsilon$ (or $\bm Y$).
But, because the predictive Bayes risks only depend on $r_{12}$, $r_{1y}$ and $r_{2y}$, we can reduce the study to a three dimensional problem.
And, since these are simply correlations, we can choose a grid of values for each that essentially spans the space of
possibilities in the 5-dimensional problem. The details of this are given in Appendix 2.

 Tables \ref{table-full}, \ref{table-onevariable} and \ref{table-null}, in Appendix 2,  summarize some features of the simulation study, for the
   correlation scenarios under the full model, the one-variable model, and the null model, respectively. Those tables present the results separately for the Case 1, Case 2, and  Case 3 situations. It is very clear from these tables that the Case 1 scenario is very
   favorable for the MPM -- it is then virtually always the optimal model -- while, in Cases 2 and 3, the MPM fails to be the optimal model
   in roughly 12\% of the cases. This is a useful result if one is in the two-variable situation, since it is easy to determine if one
   is in Case 1 or not. Alas, it is not known how to generalize this to larger dimensions.

  Table \ref{table-summary} summarizes the results, over the three cases, for each of the model correlation scenarios (full, one-variable, and null). The table reports how often the MPM and HPM equal the
  optimal predictive model ($O$), i.e., the model minimizing \eqref{Rcrit},
  and presents geometric averages of relative risks of the MPM and HPM to $O$.

  Here are some observations from Table \ref{table-summary}:
  \begin{itemize}
  \item Simpler models are more challenging for the MPM (and HPM); indeed, $MPM=O$ in 92.1\%, 85.2\%, and 84.5\% of
  the cases for the full, one-variable, and null model, respectively; still, these are high success rates, given that
  correlations vary over the full feasible spectrum.
  \item As would be expected, both the MPM and HPM do better with larger sample sizes.
  \item The vast majority of the time, the MPM and HPM are the same model but, when they differ, the MPM is typically better:
  \begin{itemize}
  \item On average, the MPM does better than the HPM (from the $M=O \neq H$ and $M > H$ columns) in 2.7\% of the cases; while the HPM does better than the MPM in 0.7\% of the cases.
      \item When the MPM and HPM are not optimal, the geometric average of the MPM risk (relative to that of $O$) is smaller than the
      geometric average for the HPM.
      \end{itemize}
  \end{itemize}

    \begin{figure}[!t]
     \subfigure[Case 1]{
     \begin{minipage}[t]{0.31\textwidth}
       \hskip-1pc  \scalebox{0.26}{\includegraphics{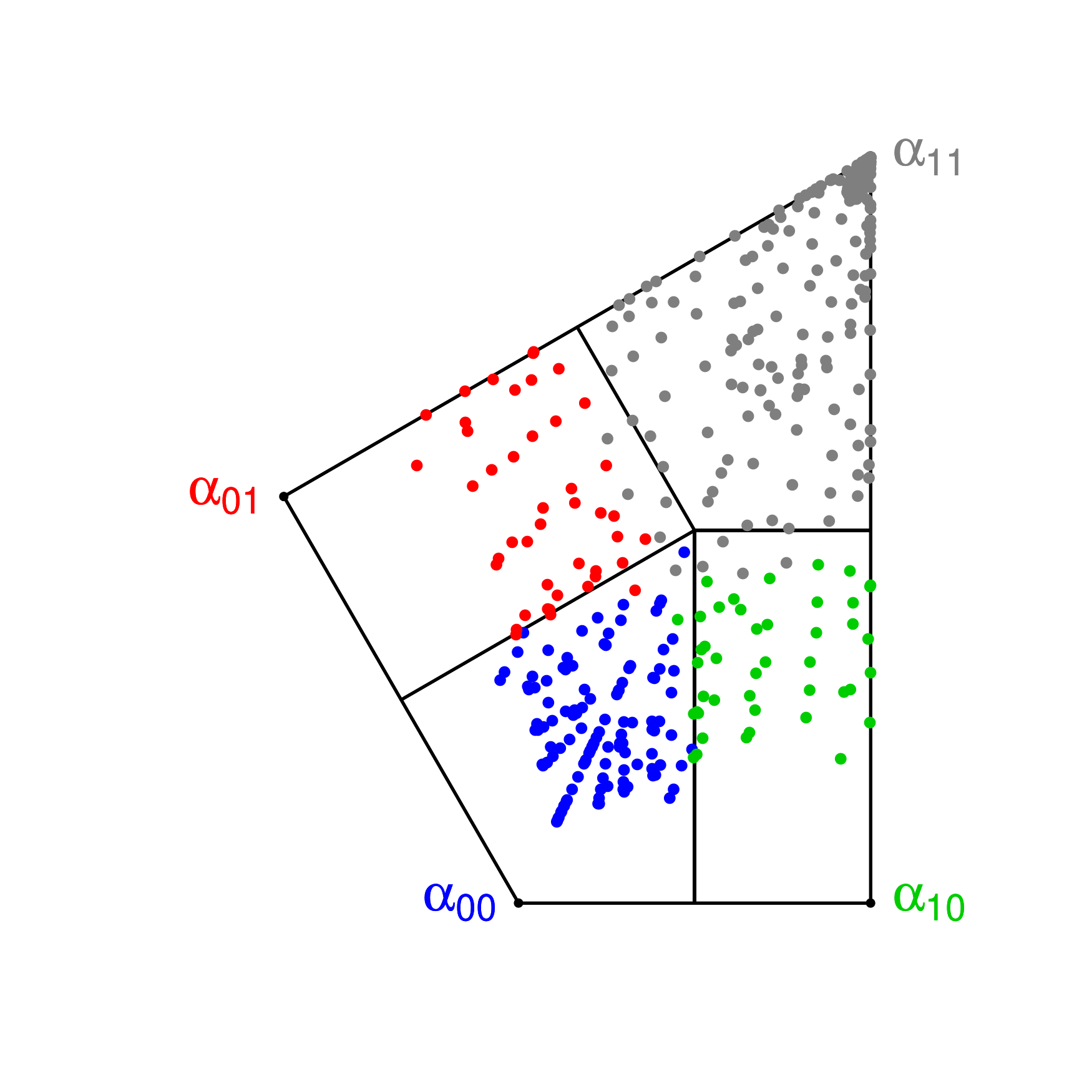}}
    \end{minipage}}
     \subfigure[Case 2]{
    \begin{minipage}[t]{0.31\textwidth}
    \hskip-1pc  \scalebox{0.26}{\includegraphics{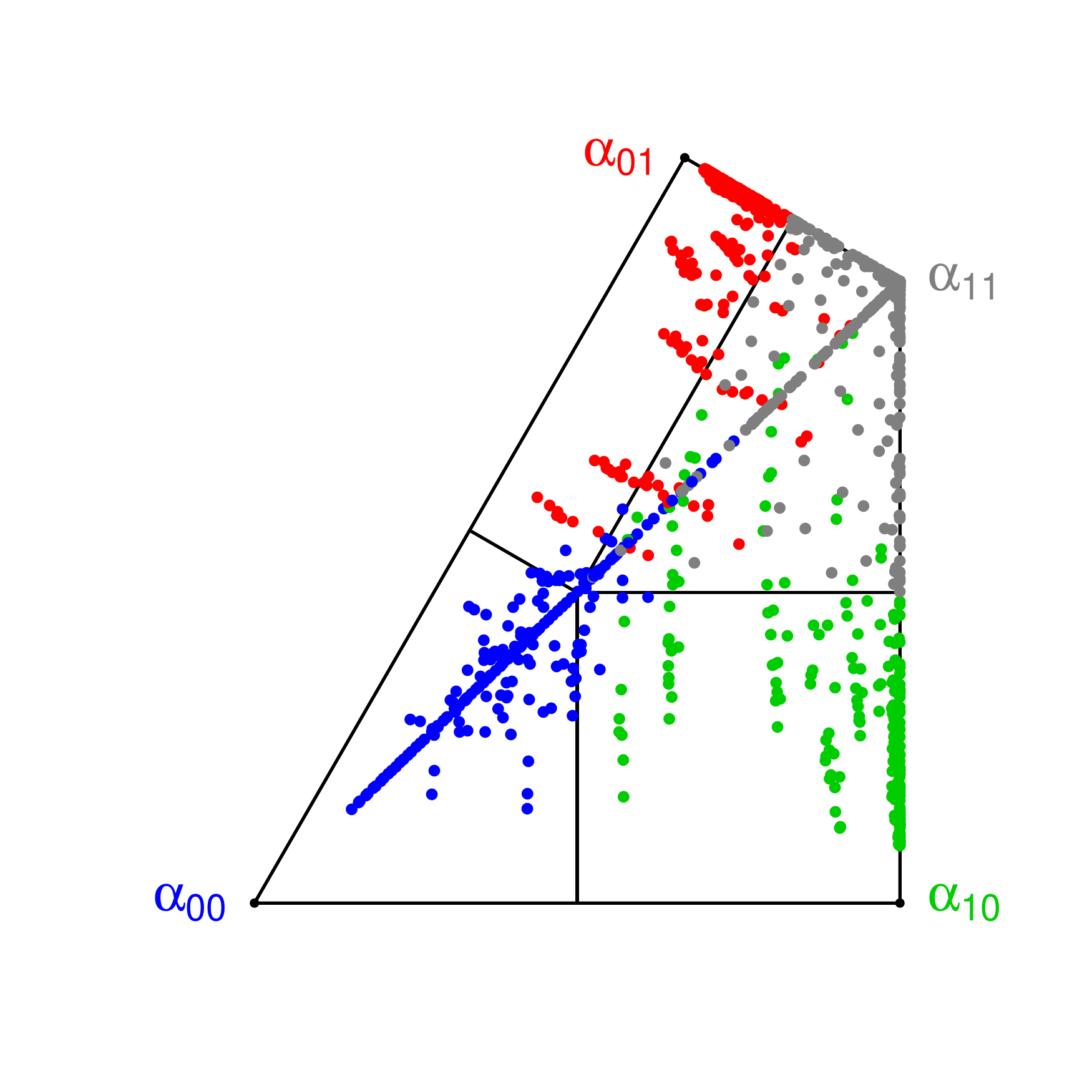}}
    \end{minipage}}
     \subfigure[Case 3]{
    \begin{minipage}[t]{0.31\textwidth}
     \hskip-1pc \scalebox{0.26}{\includegraphics{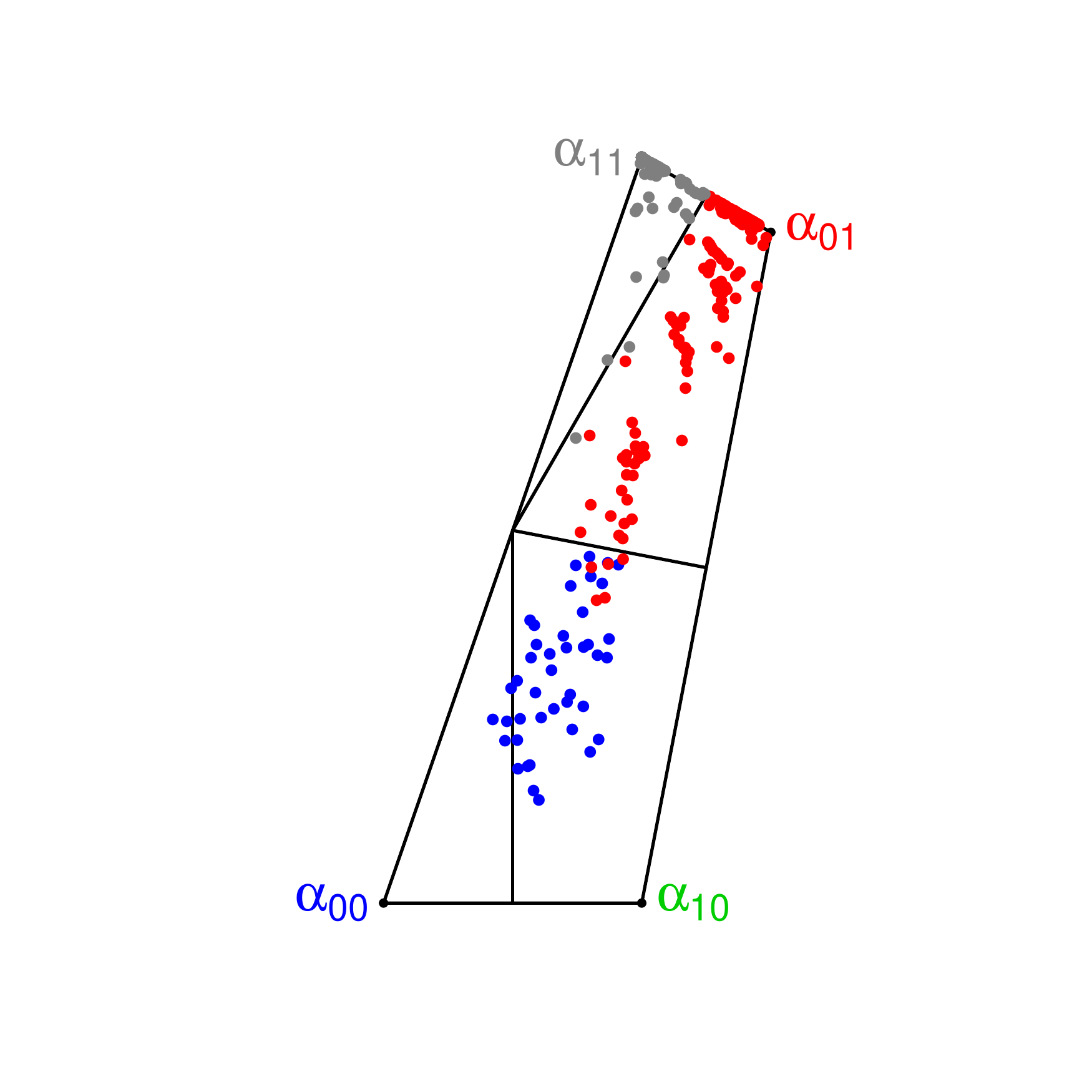}}
    \end{minipage}}
    \caption{Results from the numerical study under the full model correlation scenario: each dot has the color of the MPM,
     with the MPM being optimal (or not) if it lies within (or outside) the quadrilateral with external vertex of the same color.}    \label{fig:med-fulltot}
\end{figure}
    \begin{figure}[!t]
     \subfigure[Case 1]{
     \begin{minipage}[t]{0.31\textwidth}
       \hskip-1pc  \scalebox{0.26}{\includegraphics{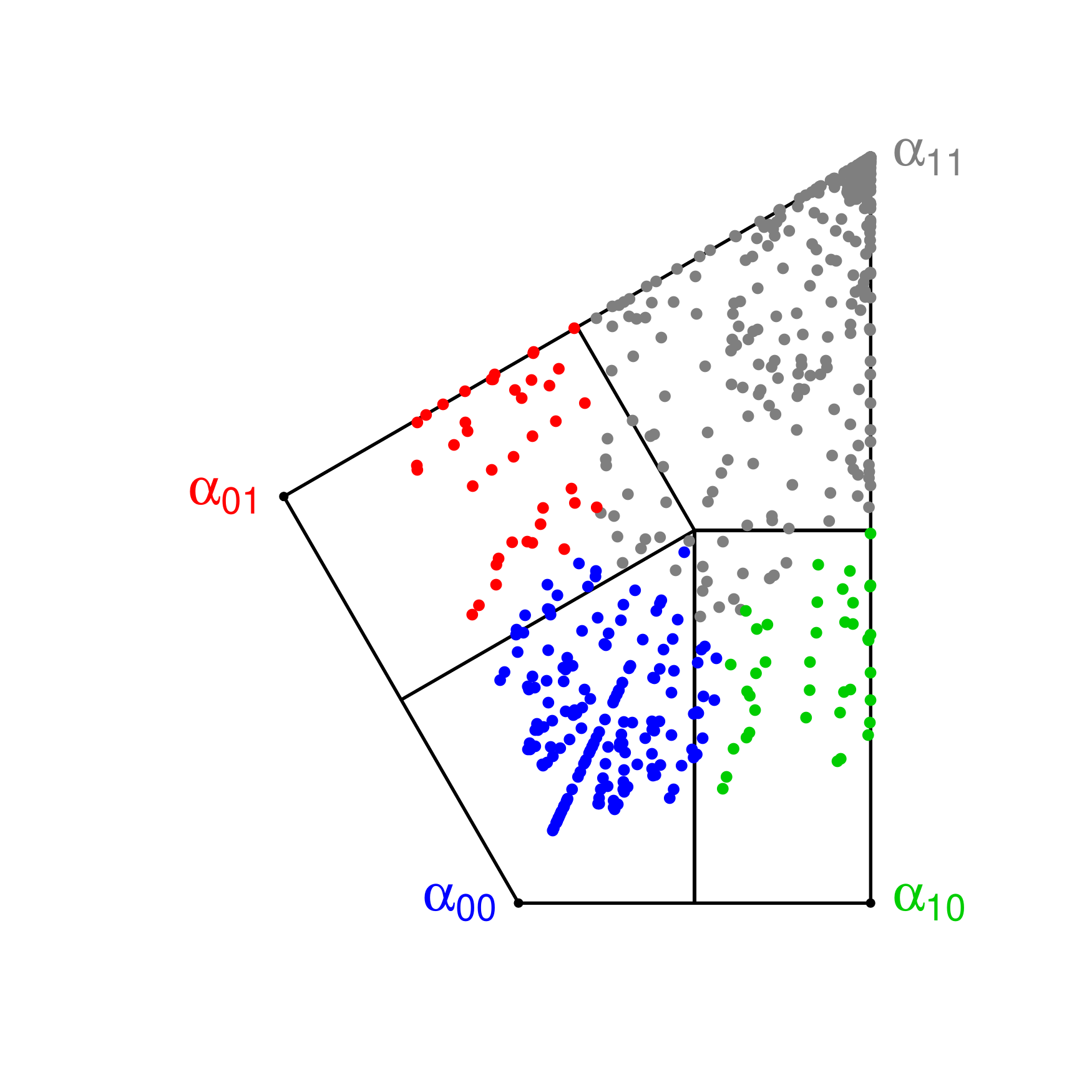}}
    \end{minipage}}
     \subfigure[Case 2]{
    \begin{minipage}[t]{0.31\textwidth}
    \hskip-1pc  \scalebox{0.26}{\includegraphics{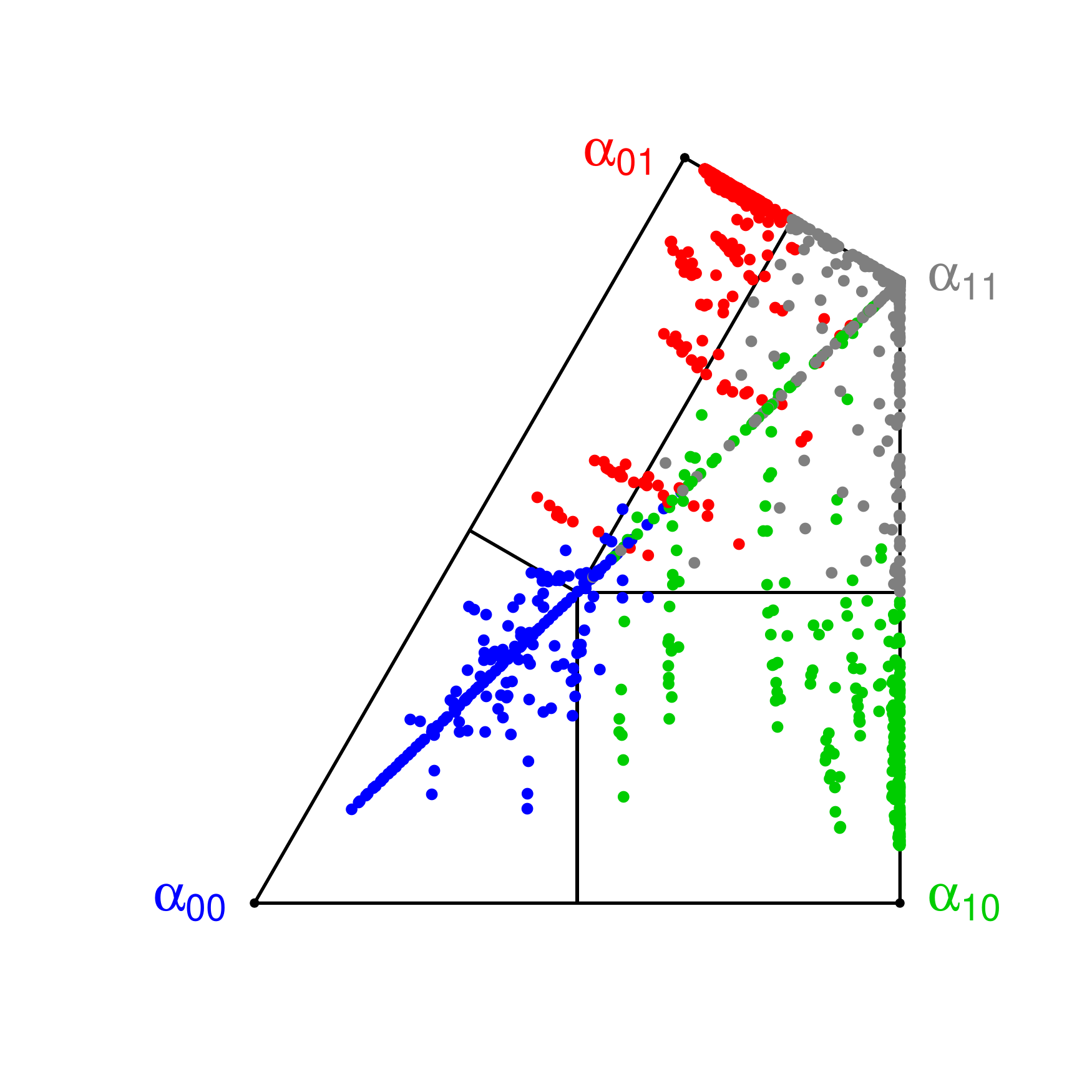}}
    \end{minipage}}
     \subfigure[Case 3]{
    \begin{minipage}[t]{0.31\textwidth}
     \hskip-1pc \scalebox{0.26}{\includegraphics{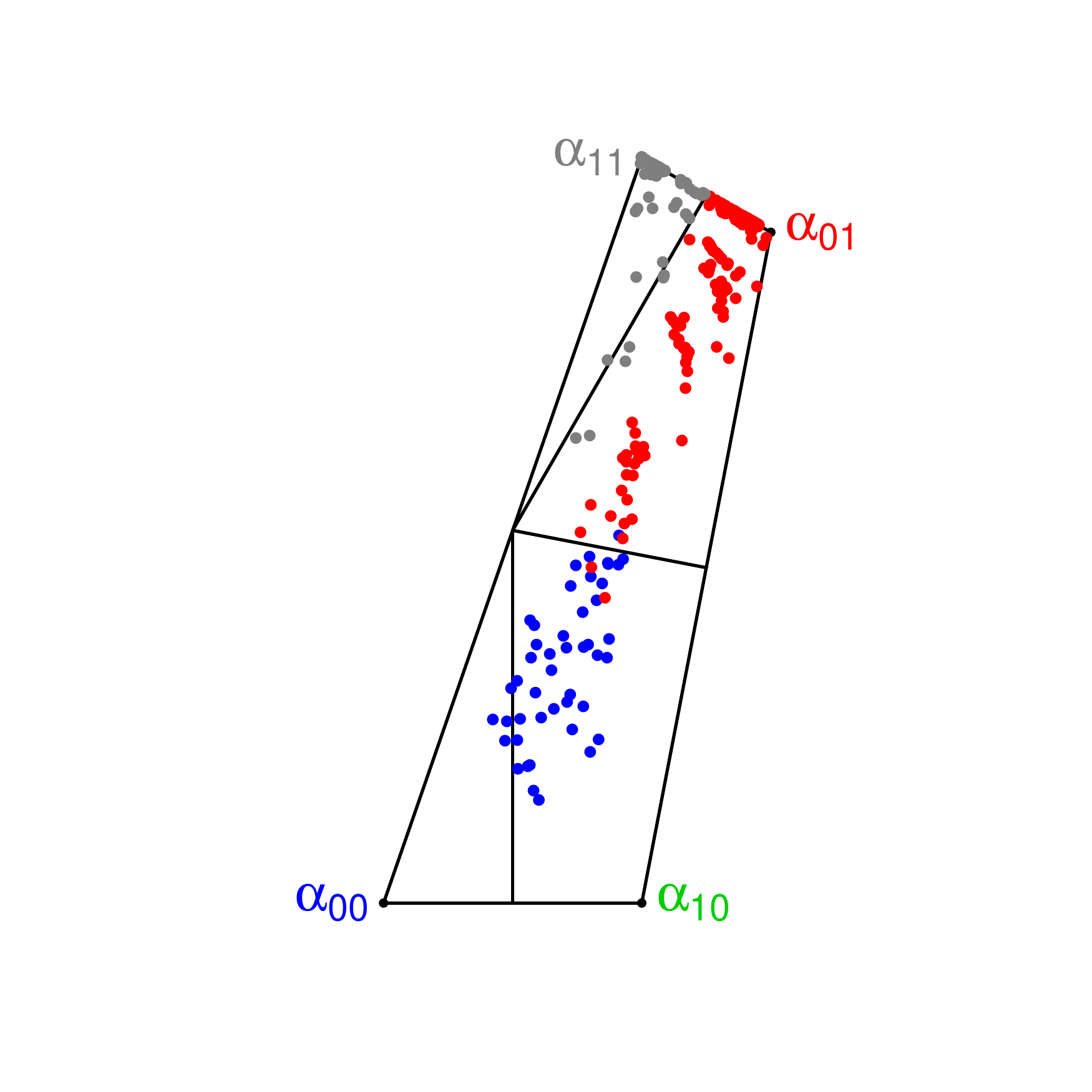}}
    \end{minipage}}
    \caption{Results from the numerical study under the full model correlation scenario: each dot has the color of the HPM,
     with the HPM being optimal (or not) if it lies within (or outside) the quadrilateral with external vertex of the same color.}    \label{fig:high-fulltot}
\end{figure}

  Additional insight can be gained by looking at the nature of the `failures' of the MPM and HPM. Figure \ref{fig:med-fulltot}, for the MPM, and
  Figure \ref{fig:high-fulltot}, for the HPM, show the errors being made, in the numerical study, for each of Case 1, Case 2 and Case 3, under the the full model correlation scenario. Focusing
  on the MPM for explanation, the color of the dots in Figure \ref{fig:med-fulltot} indicates which model was the median probability model; thus a blue dot indicates that the median probability model was $M_{00}$, because that is the color of $\alpha_{00}$. As before, the true optimal
  model for a dot is the external vertex defining the quadrilateral in which the dot lies; thus, if the blue dot lies within the quadrilateral with
  $\alpha_{00}$ as the external vertex, the MPM is the optimal model, while if the blue dot lies within the quadrilateral for which $\alpha_{10}$
  is the external vertex,
  the MPM is incorrectly saying that $M_{00}$ is optimal, when actually $M_{10}$ is optimal.

  The figures reinforce the earlier messages; Case 1 is nice for the MPM and HPM (almost all the colored dots are in the quadrilateral with
  the external vertex being of the same color), while Case 2 and, especially, Case 3 here are problematical -- in Case 3, the MPM is typically $M_{00}$ when $M_{10}$ is optimal. Careful
  examination of the figures shows that the MPM is slightly better than the HPM, but the improvement is not dramatic.

  The interesting feature revealed by the figures is that, essentially always, when the MPM and HPM fail, they do so by selecting
  a model of smaller dimension than the optimal model. There are a handful of dots going the other way, but they are hard to find.
  (This same feature was present in the corresponding figures for the one-variable and null model correlation scenarios, so
  those figures are omitted.)
  We highlight this feature because it potentially generalizes; if the MPM and HPM fail, they may typically do so by choosing too-small models.

\section{Generalizations of the optimality of the median probability model}
\subsection{More general priors in the orthogonal design case}\label{sec:general}
In orthogonal designs, the primary condition for optimality of the median probability model   is that $\wh{\b}_{\bg}$ is obtained by taking the relevant coordinates of the overall posterior mean
$\bar\b$ (condition (17) of \cite{barbieri}). With  ${\bm X}' {\bm X} = \bm D=\mathrm{diag}\{d_i\}_{i=1}^q$, the likelihood factors into independent likelihoods for each $\beta_i$ and thereby {\sl any} independent product prior
\begin{equation}\label{indep_prod}
\pi(\b\C\bg) = \prod_{i=1}^{q} \pi_i(\beta_i) \,,
\end{equation}
 will  satisfy the condition (17). This is a very important extension  because priors that are fat-tailed are often recommended over sharp-tailed priors, such as the $g$-prior (for which the optimality results of the MPM were originally conceived).

\begin{example}(Point-mass Spike-and-Slab Priors)
As an example of \eqref{indep_prod}, consider
the point-mass mixture prior $\pi(\b\C\bg)=\prod_{i=1}^q[\gamma_i\widetilde{\pi}_i(\beta_i)+(1-\gamma_i)\delta_{0}(\beta_i)]$, where $\widetilde{\pi}_i(\beta_i)$ could be e.g. the unit-information Cauchy priors, as recommended by Jeffreys.
\end{example}

\begin{example} (Continuous Spike-and-Slab Priors)
The point-mass spike is not needed for the MPM to be optimal. Consider  another example of \eqref{indep_prod}, the Gaussian mixture prior
of \cite{GM93}:
$
\pi(\b\C\bg)=\mathcal{N}_q(\bm{0}_q,\bm{V}_{\bm{\gamma}}),
$ where $\bm{V}_{\bm{\gamma}}=\mathrm{diag}\{\gamma_iv_1+(1-\gamma_i)v_0\}_{i=1}^q$ with $v_1>>v_0$.
 While the MPM was originally studied for point-mass spike-and-slab mixtures,  it is optimal {\sl also} under the continuous mixture priors.
  Indeed, to give an alternative argument, note that the posterior mean under a given model $\bg$ satisfies
\begin{eqnarray*}\label{mean_gamma}
\wh{\bm{\beta}}_{\bm{\gamma}}&=&(\bm{X}'\bm{X}+\bm{V}_{\bm{\gamma}}^{-1})^{-1}\bm{X}'\bm{Y}=\\&=&\mathrm{diag}\left\{\frac{1}{d_i+v_0^{-1}}\right\}\bm{X}'\bm{Y}+\mathrm{diag}\left\{\left(\frac{1}{d_i+v_1^{-1}}-\frac{1}{d_i+v_0^{-1}}\right)\gamma_i\right\}\bm{X}'\bm{Y},
\end{eqnarray*}
where $\bm{V}_{\bm{\gamma}}^{-1}=\mathrm{diag}
\left\{\frac{\gamma_i}{v_1}+\frac{(1-\gamma_i)}{v_0}\right\}_{i=1}^q$.  Then the posterior mean vector appears to be
$$
\bar{\bm{\beta}}=\mathrm{diag}\left\{\frac{1}{d_i+v_0^{-1}}\right\}\bm{X}'\bm{Y}+\mathrm{diag}\left\{\left(\frac{1}{d_i+v_1^{-1}}-\frac{1}{d_i+v_0^{-1}}\right)\pi_i\right\}\bm{X}'\bm{Y}.
$$
The criterion $R(\bm{\gamma})$ in \eqref{Rcrit} can be then written  as
$$
R(\bm{\gamma})=\sum_{i=1}^q\left(\frac{1}{d_i+v_1^{-1}}-\frac{1}{d_i+v_0^{-1}}\right)^2(\gamma_i-p_i)^2 d_i z_i^2.
$$
which easily seen to be minimized by the MPM model.
\end{example}


\subsection{More flexible priors in nested correlated designs}
\cite{barbieri} show that  the MPM is optimal also for correlated regressors, when considering a sequence of nested models.  Here, we generalize the class of priors under which such a statement holds.
Assume $q<n$ and denote with $\bm T$ the upper Cholesky triangular matrix such that $\bm X' \bm X  = {\bm T}' {\bm T}$. Then transform the linear model to
\begin{eqnarray*}
 \bm Y &=& {\bm X}^* {\bm \beta}^* + \bm \varepsilon \\
 &=& ({\bm X} {\bm T}^{-1}) ({\bm T} \bm \beta) + \bm \varepsilon,
\end{eqnarray*}
where $\bm\varepsilon\sim\mathcal{N}(0,\sigma^2\mathrm{I}_n)$.
Note first that, since $\bm T^{-1}$ is upper triangular, the nested sequence of models is unchanged; the parameterizations within each model have changed, but only by transforming the variables inside the model. We thus have the same nested model selection problem.

Next note that $({\bm X}^*)' {\bm X}^* = \mathrm{I}_n$, so the likelihood factors into independent likelihoods for the $\beta^*_i$; and this independence holds within each of the nested models, since the columns of ${\bm X}^*$ are orthonormal.
Thus, if the prior is chosen to be
$$\pi({\bm \beta}^*) = \prod_{i=1}^q \pi_i(\beta_i^*) \,,$$
then it follows from Section \ref{sec:general} that the median probability model is optimal.


\begin{example}{(Rescaled $g$-priors)}
Suppose the prior for $\bm \beta^*$ is $\mathcal{N}_p(\bm 0_p, \bm D)$, where $\bm D$ is diagonal. Any such prior results in optimality of the median probability model. If one transforms back to $\b$, the prior is $\mathcal{N}_p(\bm 0_p, {\bm T}^{-1} \bm D ({\bm T}^{-1})') $ which is considerably richer than the $g$-type priors considered in \cite{barbieri} (which would be these priors with $\bm D = g \bm I$).
As a specific illustration, suppose
$$\bm X = \left(
            \begin{array}{cc}
              1 & 1 \\
              1 & 1\\
              1+\epsilon & 1 \\
            \end{array}
          \right) \,,$$
with $\epsilon>0$ small. Computation then yields that the prior covariance matrix is
$${\bm T}^{-1} \bm D ({\bm T}^{-1})' = \frac{1}{2 \ \epsilon^2} \left(
            \begin{array}{cc}
            3d_1 & -(3+\epsilon) d_1 \\
             -(3+\epsilon) d_1 & [(3+2\epsilon + \frac{1}{3} \epsilon^2)d_1 + \frac{2}{3} \epsilon^2 d_2] \\
            \end{array}
          \right) \,.$$
The $g$-prior choice is $d_1=d_2=g$, resulting in
$${\bm T}^{-1} \bm D ({\bm T}^{-1})' = \frac{g}{2 \ \epsilon^2} \left(
            \begin{array}{cc}
            3 & -(3+\epsilon)  \\
             -(3+\epsilon)  & [(3+2\epsilon + \frac{1}{3} \epsilon^2) + \frac{2}{3} \epsilon^2 ] \\
            \end{array}
          \right) \,.$$
The alternative choice $d_1 = g  \epsilon^2$, $d_2=g$ yields
$${\bm T}^{-1} \bm D ({\bm T}^{-1})' = \frac{g}{2} \left(
            \begin{array}{cc}
            3 & -(3+\epsilon)  \\
             -(3+\epsilon)  & [(3+2\epsilon + \frac{1}{3} \epsilon^2) + \frac{2}{3}  d_2] \\
            \end{array} \right)
            \approx \frac{g}{2} \left(
            \begin{array}{cc}
            3 & -3  \\
             -3  &  \frac{11}{3} \\
            \end{array}
          \right) \,.$$
Thus the $g$-prior assigns a variance of $3g/[2 \epsilon^2]$ to $\beta_1$, while the new prior assigns a variance of $1.5 g$. This may
be much more reasonable in certain contexts.

\end{example}

\begin{figure}[!t]
     \subfigure[$r_{12}=0$]{
     \begin{minipage}[t]{0.31\textwidth}
       \hskip-1pc  \scalebox{0.25}{\includegraphics{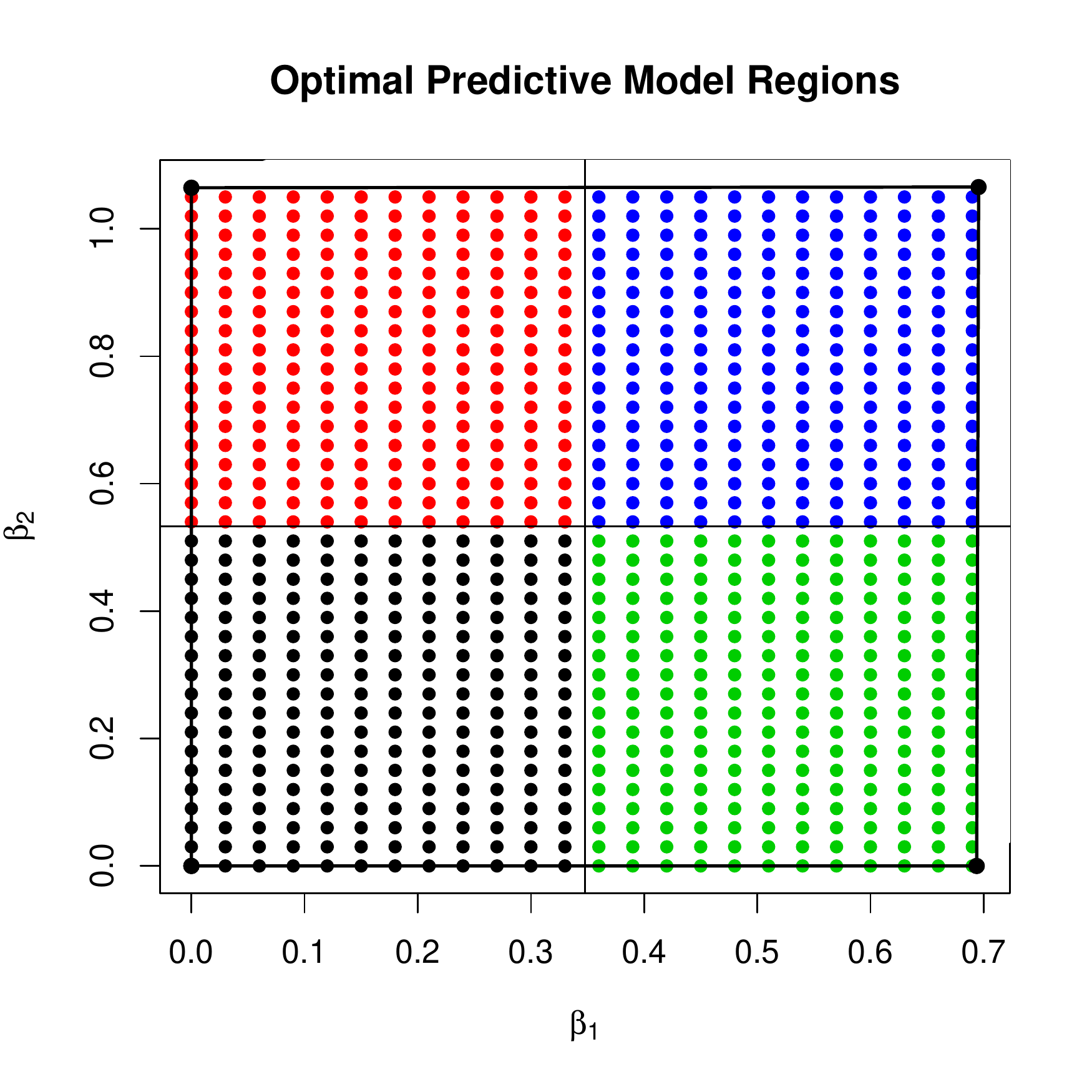}}  \label{pic:orthogonal}
    \end{minipage}}
     \subfigure[$r_{12}=0.5$]{
    \begin{minipage}[t]{0.31\textwidth}
    \hskip-1pc  \scalebox{0.25}{\includegraphics{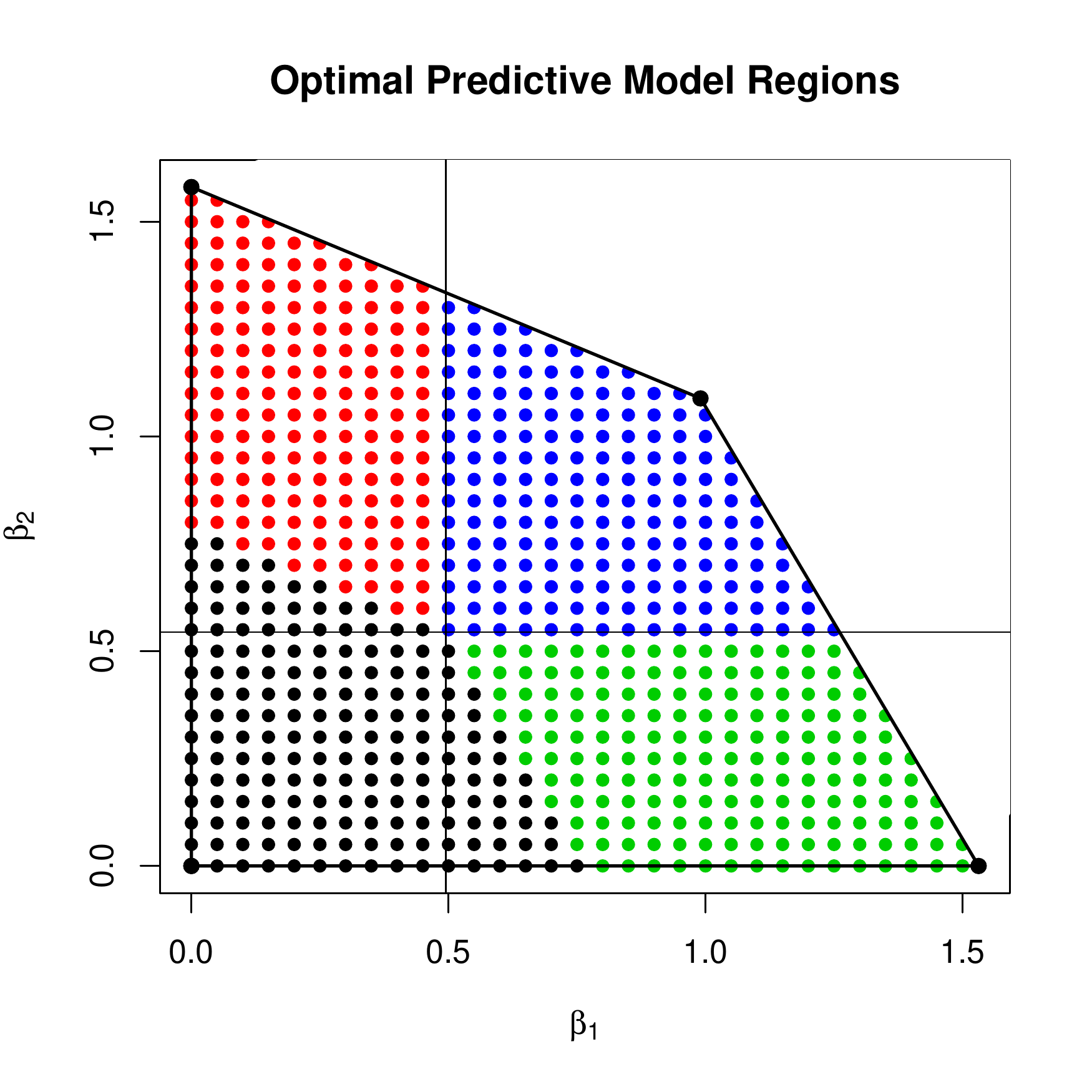}}     \label{fig1b}
    \end{minipage}}
     \subfigure[$r_{12}=0.9$]{
    \begin{minipage}[t]{0.31\textwidth}
     \hskip-1pc \scalebox{0.25}{\includegraphics{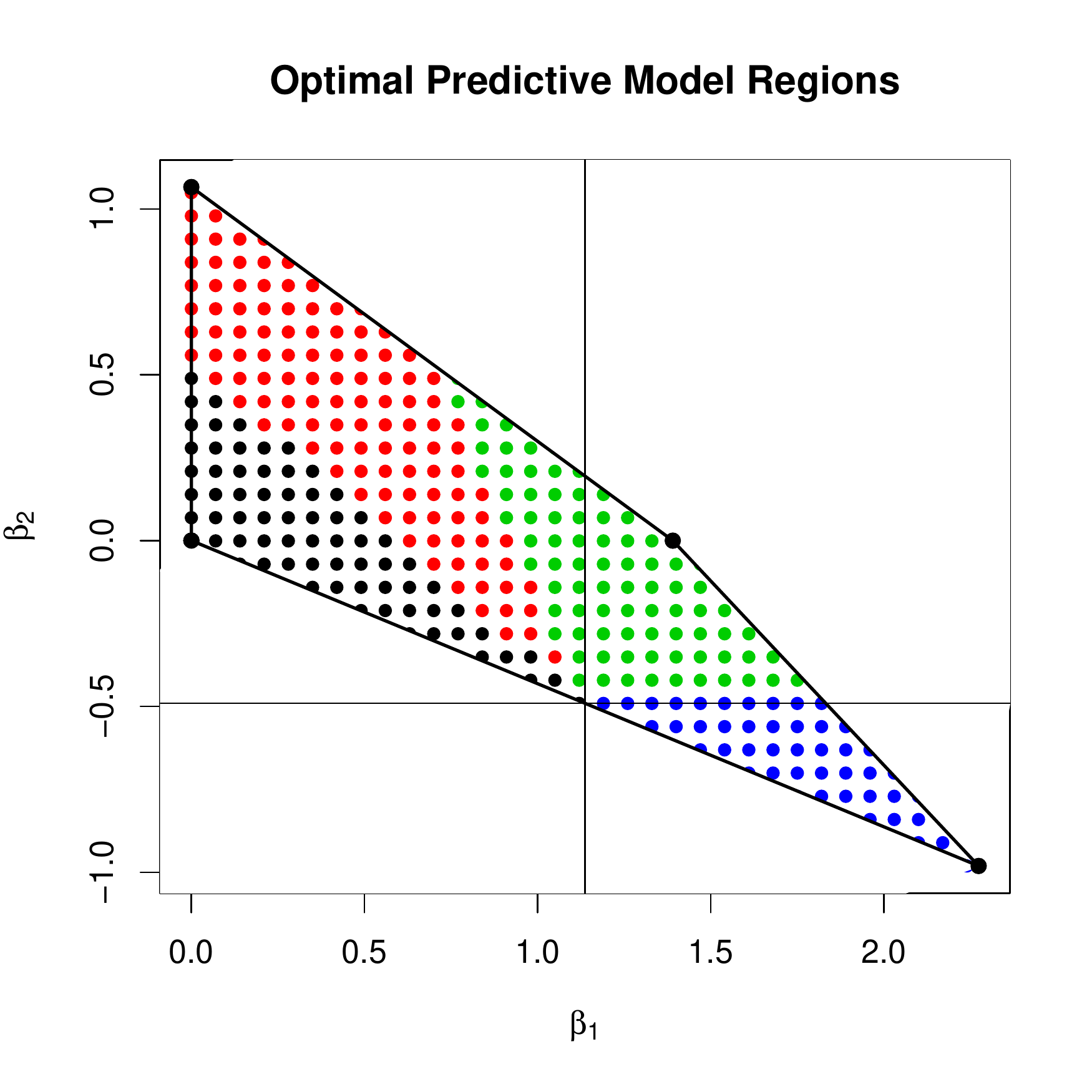}}   \label{fig1c}
    \end{minipage}}
    \caption{Plots of the optimality regions}    \label{fig1}
\end{figure}

\begin{figure}[!t]
     \subfigure[MPM]{
     \begin{minipage}[t]{0.31\textwidth}
       \hskip-1pc  \scalebox{0.25}{\includegraphics{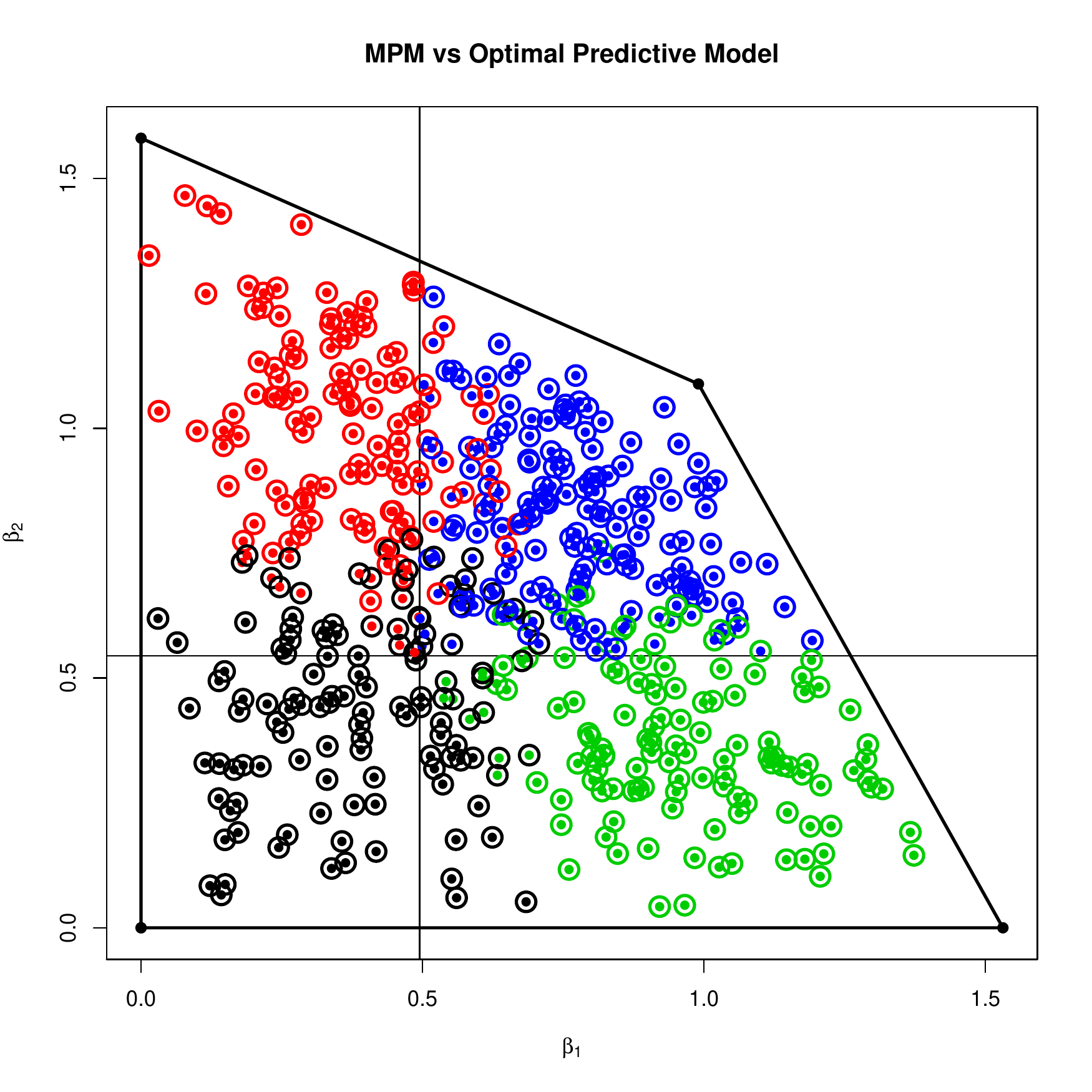}}  \label{fig2a}
    \end{minipage}}
     \subfigure[LASSO  ($\lambda=50$)]{
    \begin{minipage}[t]{0.31\textwidth}
    \hskip-1pc  \scalebox{0.25}{\includegraphics{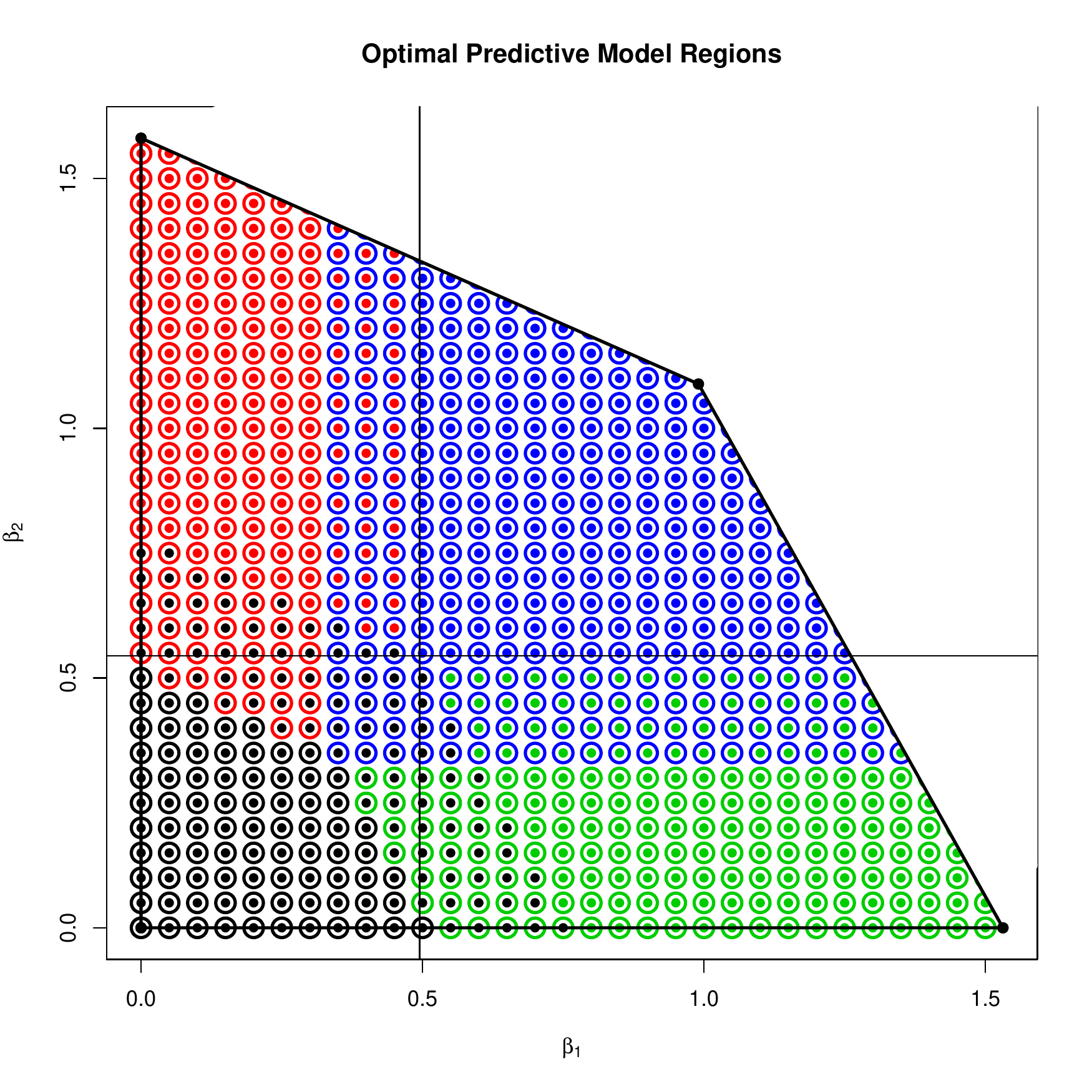}}     \label{fig2b}
    \end{minipage}}
     \subfigure[LASSO ($\lambda=80$)]{
    \begin{minipage}[t]{0.31\textwidth}
     \hskip-1pc \scalebox{0.25}{\includegraphics{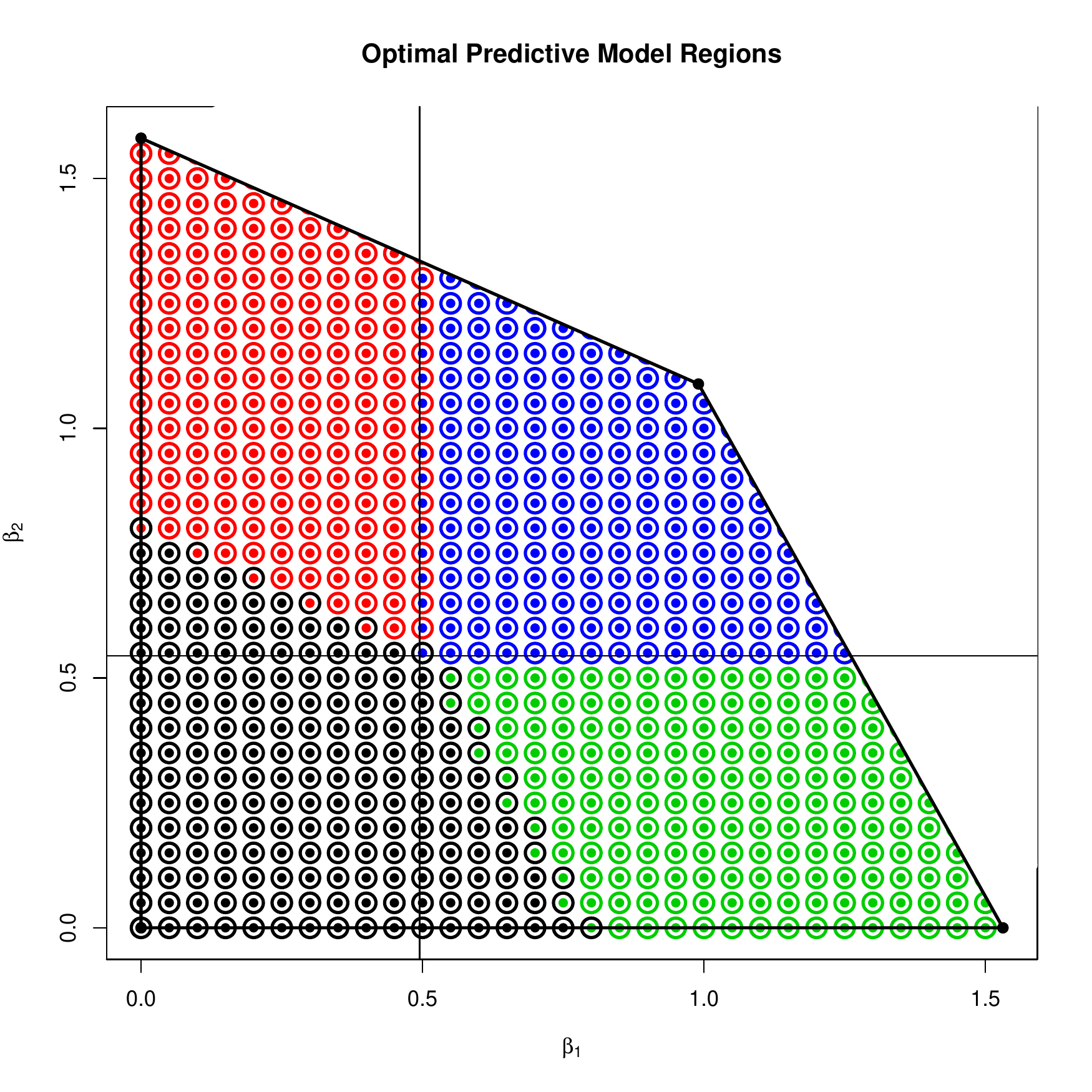}}   \label{fig2c}
    \end{minipage}}
    \caption{Plots of the actual (full dots) and estimated (round circles) optimal predictive regions }    \label{fig2}
\end{figure}

\subsection{The equi-correlated case: optimality can depend on higher order inclusion probabilities}\label{sec:equicor}
If we were to visualize the geometry of optimal predictive model selection in the orthogonal case, we would obtain a rectangular partition of a convex hull of posterior means under each model. Figure \ref{pic:orthogonal} depicts an example of such geometry when $p=2$.  The four black dots correspond to the four posterior means $\wh{\b}_{\bg}$  and comprise a skeleton of a convex hull of  all possible locations of the overall posterior mean $\bar\b$. Each of these hypothetical means is associated with one optimal predictive model, i.e. the model that is closest  to $\bar\b$ in terms of $R(\bg)$. Denote with $\wh{\b}^F=(\wh{\beta}_{1}^F,\wh{\beta}_{2}^F)$ the posterior mean under the full model $\bg=(1,1)'$. When $r_{12}=0$, the optimal predictive model (O) regions are rectangles (marked with 4 colors) where the cuts occur at $\wh{\beta}_{1}^F/2$ (vertical line) and $\wh{\beta}_{2}^F/2$ (horizontal line). The median probability model is known to be optimal in this case and it can be obtained by element-wise thresholding of $|\bar{\b}|$ at $|\wh{\b}^F|/2$.  When the predictors are correlated ($r_{12}=0.5$ in Figure \ref{fig1b}), the regions are no longer rectangular, where simple thresholding of $\bar\b$ is no longer enough to describe the optimal model. It is worthwhile to note, however, that the full model is the optimal model iff  $|\bar\b|\geq |\wh{\b}^F|/2$. Note, also, that
Figure \ref{fig1} (and the later Figure \ref{fig2}) correspond to the Case 2 situation in Section 3.

We focus on the example with $r_{12}=0.5$ a bit more closely in Figure \ref{fig2}. On the left, Figure \ref{fig2a}, we have a comparison with the MPM. The dots correspond to locations of the posterior mean $\bar\b$, where the posterior model probabilities were sampled from $Dir(1,1,1,1)$. The color of the solid dot designates the optimal predictive model. The color of the round circle surrounding each dot designates the median probability model. We can see an agreement between the MPM and O when the posterior model probabilities put a lot of weight onto one model (corners of the hull). When there is model selection uncertainty (the centre of the hull), the MPM does not have to be optimal. It is interesting to note that the regions of the MPM selection are overlapping suggesting that using only first posterior moments $\E[\bg\C\Y]$ may not be enough to characterize the optimal model. The following lemma provides a full characterization of the optimal model in terms of both the first and the second moments.

\begin{lemma}
Assume $\bm{Q}= \bm{X}'\bm{X} =(1-r)\mathrm{I}_p+r\bm{1}\bm{1}'$ for some $-1<r<1$.
Denote $\bm{\pi}=\E[\bg\C\Y]$ the vector of posterior inclusion probabilities and $\bm{\Pi}=\E\left[\frac{\bm{\gamma}\bm{\gamma}'}{1-r+r\|\bg\|}\C\Y\right]$.
Under \eqref{gprior}, the optimal predictive model minimizes
\begin{equation}\label{eq:equi_R}
R(\bm{\gamma})=(1-r)\sum_{i=1}^p\left(\sum_{j=1}^pb_{ij}({\bm{\gamma}})z_j\right)^2+r\sum_{i=1}^p\left[\left(\sum_{j=1}^pb_{ij}({\bm{\gamma}})\right)z_i\right]^2,
\end{equation}
where $\bm{Z}=\bm{X}'\bm{Y}=(z_1,\dots,z_p)'$ and $\bm B(\bg)=\left[\mathrm{diag}\{\bg-\bm\pi\}-r\left(\frac{\bg\bg'}{1-r+r\|\bg\|}-\bm\Pi \right) \right]$.
\end{lemma}
\begin{proof}
We begin by noting that
$$
\H_{\bg}\wh{\b}_{\bg}=\frac{g}{1+g}\frac{1}{1-r}\left(\mathrm{diag}\{\bg\}-\frac{r\,\bg\bg'}{1-r+r\|\bg\|}\right)\bm Z,
$$
where we used the fact
$$
(\bm{G}+\bm{H})^{-1}=\bm{G}^{-1}-\frac{1}{1+a}\bm{G}^{-1}\bm{H}\bm{G}^{-1},
$$
for $\bm G=(1-r)\mathrm{I}$ and $\bm H=r \bm{1}\bm{1}'$, where $a=\mathrm{trace}(\bm{HG}^{-1})$.  Denote with $\bm\pi=\E[\bg\C\Y].$ Then $\bm H_{\bg}\wh{\b}_{\bg}-\bar{\b}=$
\begin{align*}
&\frac{g}{1+g}\frac{1}{1-r}\left\{\mathrm{diag}\{\bg-\bm\pi\}-r\left[\frac{\bg\bg'}{1-r+r\|\bg\|}-\E\left(\frac{\bg\bg'}{1-r+r\|\bg\|}\C\Y\right) \right] \right\}\bm Z.
\end{align*}
The rest follows from matrix algebra.
\end{proof}

When $r=0$, we obtain the usual criterion $R(\bg)$ minimized by the MPM model. The larger the correlation $r$, the more weight is put on the joint inclusion probabilities, where the optimal model is the one whose matrix $\bg\bg'$ is closest to the posterior mean of $\bg\bg'$ (normalized by the model size).

It is also useful to point out that the MPM no longer corresponds to simple thresholding of $\bar\b$  when $r\neq0$.  Because the predictors are correlated, it would seem natural to threshold some functional of  $\bar\b$ which takes into account the correlation. An example of one possible approach is given in \cite{hahn_carvalho}, who suggest running a lasso regression of $\X\bar\b$ onto $\X$. Such a LASSO post-processing step yields a model which summarizes $\bar\b$ while taking into account the correlation pattern between $\x$'s. The regions of such LASSO selected model are depicted in Figure \ref{fig2b} and \ref{fig2c} (again the solid dots are O and the round circle around are the LASSO selected models). We can see that, indeed, the LASSO selection takes into account the correlation and, interestingly, it can almost exactly match the  optimal predictive regions  for a suitably chosen hyper-parameter $\lambda$. This connection between the LASSO post-processing of $\bar\b$ and optimal predictive model is  curious. However, when the predictors are highly correlated ($r=0.9$ in Figure \ref{fig3}), the LASSO regions do not yield the O regions, not even remotely (Figure \ref{fig3}).

\begin{figure}[!t]
     \subfigure[MPM]{
     \begin{minipage}[t]{0.31\textwidth}
       \hskip-1pc  \scalebox{0.25}{\includegraphics{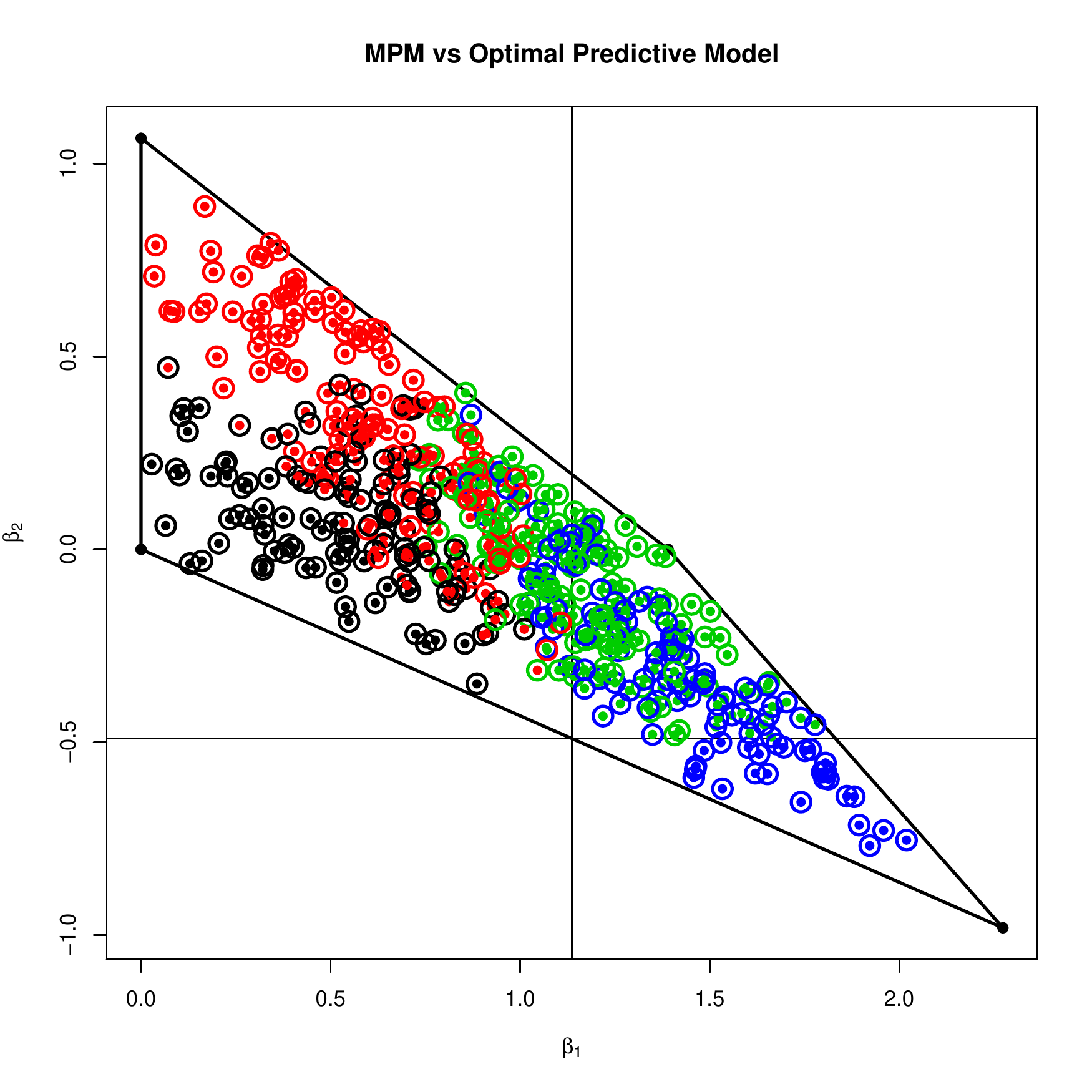}}  \label{fig3a}
    \end{minipage}}
     \subfigure[LASSO  ($\lambda=50$)]{
    \begin{minipage}[t]{0.31\textwidth}
    \hskip-1pc  \scalebox{0.25}{\includegraphics{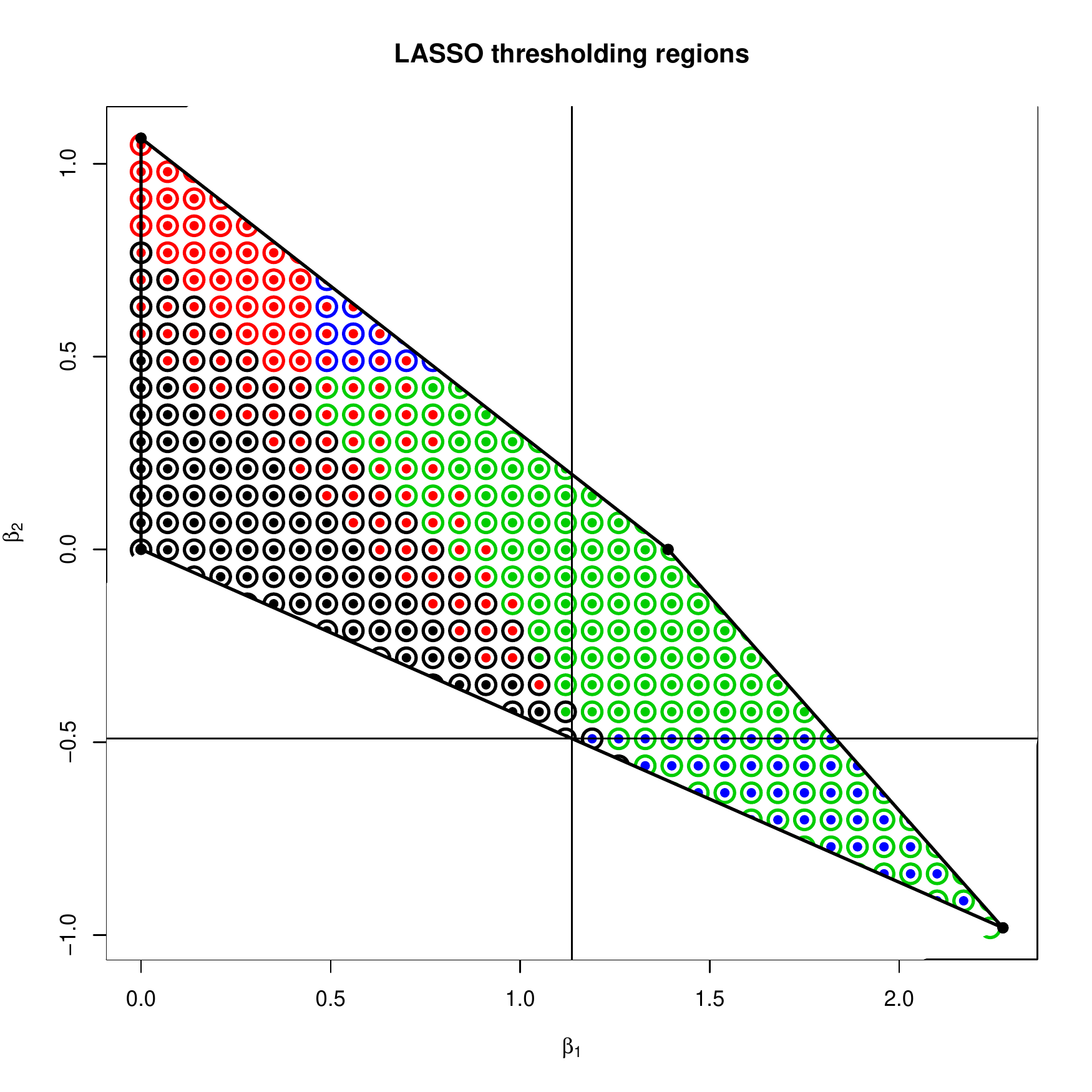}}     \label{fig3b}
    \end{minipage}}
     \subfigure[LASSO  ($\lambda=80$)]{
    \begin{minipage}[t]{0.31\textwidth}
     \hskip-1pc \scalebox{0.25}{\includegraphics{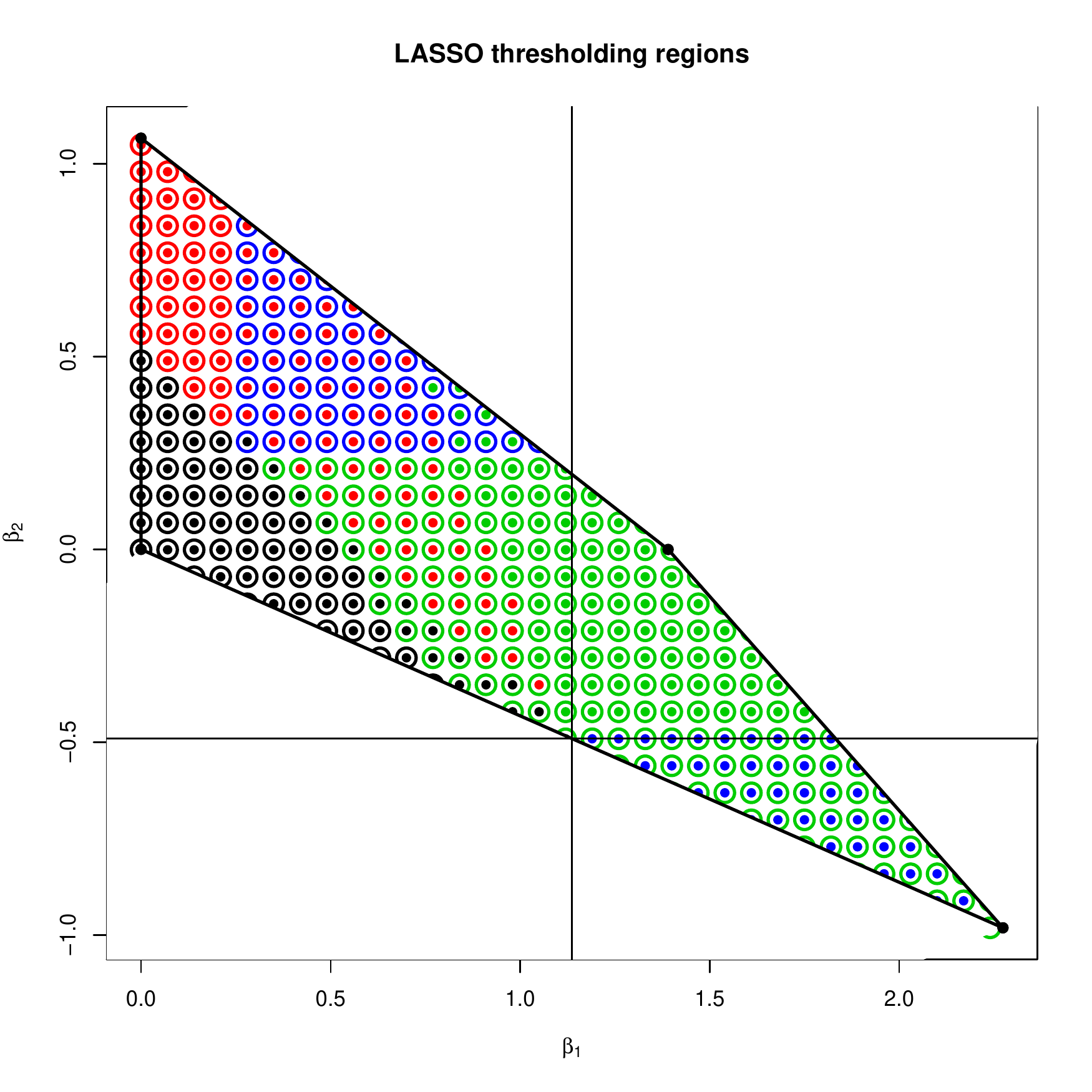}}   \label{fig3c}
    \end{minipage}}
    \caption{Plots of the actual (full dots) and estimated (round circles) optimal predictive regions }    \label{fig3}
\end{figure}

\section{Discussion}

The paper consists of two quite different parts. One part (mostly Section 4) focuses on generalizing previous theorems concerning the
optimality of the median probability model. In addition to the generalizations therein a number of other generalizations are suggested
in the paper, when groups of variables are orthogonal to others. Here are three such results, whose proofs are essentially obvious.
\begin{description}
\item[Result 1.]
If one group of variables is orthogonal to another, then finding the MPM and the optimal procedure can be done separately
for each group of variables.
\item[Result 2.]
If a variable is orthogonal to all others, it can separated from the problem and handled on its own, and will belong in the optimal model if its inclusion probability is bigger than 1/2.
\item[Result 3.]
If two groups of orthogonal variables each have a nested structure, then the median probability model is optimal and can be found separately in each group.
\end{description}

In spite of the considerable generalizations of optimality afforded by Section 4 and these related results, the extent to which the
median probability model is guaranteed to be optimal is still rather limited. Hence the second goal of the paper was to study
the extent to which the MPM failed to be optimal. This was done in two ways: first, by looking at ``worst cases," where the
number of highly correlated variables grows and second, by doing an extensive numerical study to see how often the MPM (and HPM) fail
to be optimal. The conclusions from the numerical study are given in Section 3, and won't be repeated here, except to say that
the performance of the MPM was overall excellent, even in highly correlated situations, and was measurably better than the performance of the HPM.

The MPM can fail, however, and fail badly, so
we finish with a discussion of when this happens, focusing (for simplicity) on the case where there are many
replicates of the covariate vector $\bm x$ in the model; then the median probability model will not include
that covariate. Consider four cases.

\smallskip
\noindent
{\em Case 1. $\bm x$ is not useful for prediction:} Now the median probability model might well do better than the model averaged answer for the original problem, since the median probability model will ignore $\bm x$, while the model averaged answer insists on including it.

\smallskip
\noindent
{\em Case 2. $\bm x$ is crucial for good prediction:} Now the median probability model does very poorly. Unfortunately, the error here, in not including $\bm x$, will typically be much larger than the gain in Case 1.

\smallskip
\noindent
{\em Case 3. $\bm x$ is helpful but not crucial for good prediction:} This is like the situation in Section 1.1. The harm in the median probability model ignoring $\bm x$ may be rather small.

\smallskip
\noindent
{\em Case 4. Nested Models:} If the above arises in a nested model scenario, the median probability model is, of course, the optimal single model. It can still err, however, through the prior probabilities being inappropriate, assigning too much mass to all the duplicate models. (But this is just saying that the model averaged answer then can also err.)

\section*{Appendix 1: Proof of Mini-Theorems}
We denote with $\mbox{\boldmath $\alpha$}_{\mbox{\footnotesize\boldmath $\gamma$}}$  the projection of $\mbox{\boldmath $y$}$ on the space
spanned by the columns of ${\bf X}_{\mbox{\footnotesize\boldmath
$\gamma$}}$. Assume that all variables have been standardized, so that
$$\mbox{\boldmath $\alpha$}_{00}=\left(\begin{array}{c} 0 \\ 0 \end{array} \right), \quad
\mbox{\boldmath $\alpha$}_{10}=\left(\begin{array}{c} a \\ 0 \end{array} \right), \quad
\mbox{\boldmath $\alpha$}_{01}=\left(\begin{array}{c} b \\ c \end{array} \right),  \quad
\mbox{\boldmath $\alpha$}_{11}=\left(\begin{array}{c} a \\ d \end{array} \right),$$
with
$$a=r_{1y}, \quad \quad b=r_{12}\, r_{2y}, \quad \quad c= (1-r_{12}^2)^{1/2}\, r_{2y}, \quad \quad d= \frac{r_{2y}-r_{12}\, r_{1y}}{(1-r_{12}^2)^{1/2}},$$
where $r_{12}=Corr(x_1, x_2)$, $r_{1y}=Corr(x_1, y)$ and $r_{2y}=Corr(x_2, y)$.
Actually the original expression of each coordinate has an irrelevant common factor equal to $\sqrt{n}$, which has been ignored.
The model average point $\bar{\mbox{\boldmath $\alpha$}}$ has coordinates $\bar\alpha_1$ and $\bar\alpha_2$ given by
\begin{displaymath}
 \left(
\begin{array}{c}\bar{\alpha}_1 \\ \bar{\alpha}_2 \end{array} \right)=p_{10} \left(
\begin{array}{c}a \\0 \end{array} \right)+p_{01} \left(
\begin{array}{c}b \\c \end{array} \right)+p_{11} \left(
\begin{array}{c}a \\d \end{array} \right)
\end{displaymath}
where $p_{\mbox{\footnotesize\boldmath $\gamma$}}$ is the posterior probability of model $M_{\mbox{\footnotesize\boldmath $\gamma$}}$.

Suppose that we would like to check if the model average point $\bar{\mbox{\boldmath $\alpha$}}$ lies
inside a particular triangular subregion of the space $\left\{\mbox{\boldmath $\alpha$}_{00}, \mbox{\boldmath $\alpha$}_{10}, \mbox{\boldmath $\alpha$}_{01}, \mbox{\boldmath $\alpha$}_{11}\right\}$.
To this aim, we express the coordinates of $\bar{\mbox{\boldmath $\alpha$}}$ as a linear combination of the coordinates of the vertexes of the triangular subregion. The model average point is inside the triangular subregion if the weights of the vertexes result to be all positive.

In particular, when we refer to the triangular subregion
$S_1=\left\{\mbox{\boldmath $\alpha$}_{00}, \mbox{\boldmath $\alpha$}_{10}, \mbox{\boldmath $\alpha$}_{11} \right\}$, we write the model average point as
\begin{displaymath}
\left(
\begin{array}{c}\bar{\alpha}_1 \\ \bar{\alpha}_2 \end{array} \right)=w^{(1)}_{10} \left(
\begin{array}{c}a \\0 \end{array} \right)+w^{(1)}_{11} \left(
\begin{array}{c}a \\d \end{array}
\right),
\end{displaymath}
with
$w^{(1)}_{00}+w^{(1)}_{10}+w^{(1)}_{11}=1$, and we may find that:
\begin{eqnarray*}w^{(1)}_{00}&=&1-\frac{\bar{\alpha}_1}{a}\\
w^{(1)}_{10}&=&\frac{\bar{\alpha}_1}{a}-\frac{\bar{\alpha}_2}{d}\\
w^{(1)}_{11}&=&\frac{\bar{\alpha}_2}{d}.\end{eqnarray*}

Note that the sign of each weight gives us information on the position of $\bar{\mbox{\boldmath $\alpha$}}$ with respect to the segment joining the other two vertexes. In fact if one of the weight is positive, say $w^{(1)}_{10}$, this means that $\bar{\mbox{\boldmath $\alpha$}}$ lies on the side of $\mbox{\boldmath $\alpha$}_{10}$ with respect to the line through $\mbox{\boldmath $\alpha$}_{00}$ and $\mbox{\boldmath $\alpha$}_{11}$. If $w^{(1)}_{10}<0$ then $\bar{\mbox{\boldmath $\alpha$}}$ lies on the other side, while if $w^{(1)}_{10}=0$ it lies on the segment.

In the same way, when we consider the triangular subregion
$S_2=\left\{\mbox{\boldmath $\alpha$}_{00}, \mbox{\boldmath $\alpha$}_{01}, \mbox{\boldmath $\alpha$}_{11} \right\}$,
we write the model average point as
\begin{displaymath}
\left(
\begin{array}{c}\bar{\alpha}_1 \\ \bar{\alpha}_2 \end{array} \right)=w^{(2)}_{01} \left(
\begin{array}{c}b \\c \end{array} \right)+w^{(2)}_{11} \left(
\begin{array}{c}a \\d \end{array} \right)
 \end{displaymath}
 with $w^{(2)}_{00}+w^{(2)}_{01}+w^{(2)}_{11}=1$ and
\begin{eqnarray*}
w^{(2)}_{00}&=&1+\frac{(d-c)\, \bar{\alpha}_1+(b-a)\, \bar{\alpha}_2}{ac-bd}\\
w^{(2)}_{01}&=&\frac{a\, \bar{\alpha}_2-d\, \bar{\alpha}_1}{ac-bd}\\
w^{(2)}_{11}&=&\frac{c\, \bar{\alpha}_1-b\, \bar{\alpha}_2}{ac-bd}.
\end{eqnarray*}

In case 1 and 2 the triangular subregions
$S_1$ and $S_2$ are disjoint and their union covers the entire space $\left\{\mbox{\boldmath $\alpha$}_{00}, \mbox{\boldmath $\alpha$}_{10}, \mbox{\boldmath $\alpha$}_{01}, \mbox{\boldmath $\alpha$}_{11}\right\}$ (see Figure \ref{fig:subreg-12}).

\begin{figure}[!t]
     \subfigure[Case 1]{
     \begin{minipage}[t]{0.51\textwidth}
       \hskip-1pc  \scalebox{0.4}{\includegraphics{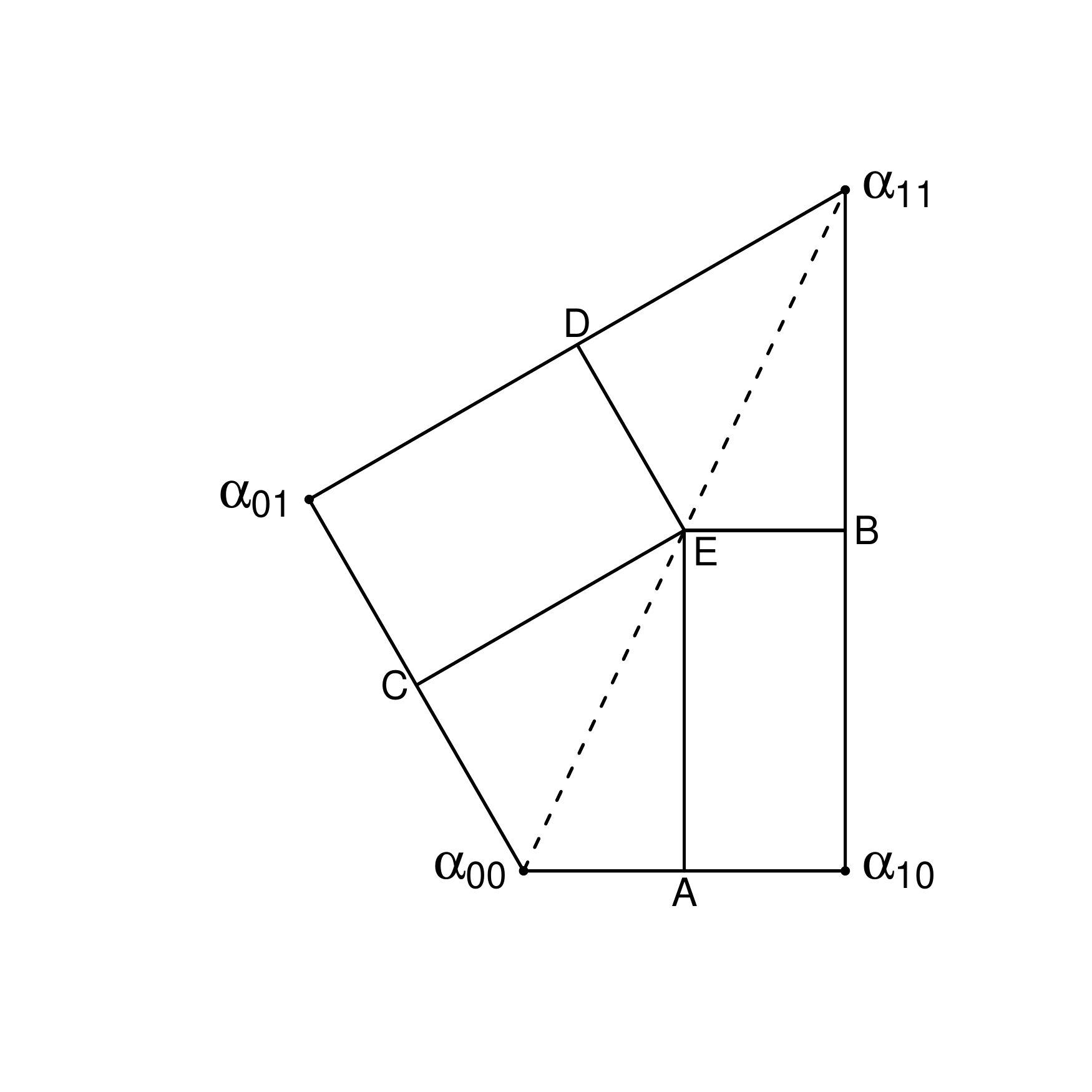}}  \label{fig:srcase1}
    \end{minipage}}
     \subfigure[Case 2]{
    \begin{minipage}[t]{0.51\textwidth}
    \hskip-1pc  \scalebox{0.4}{\includegraphics{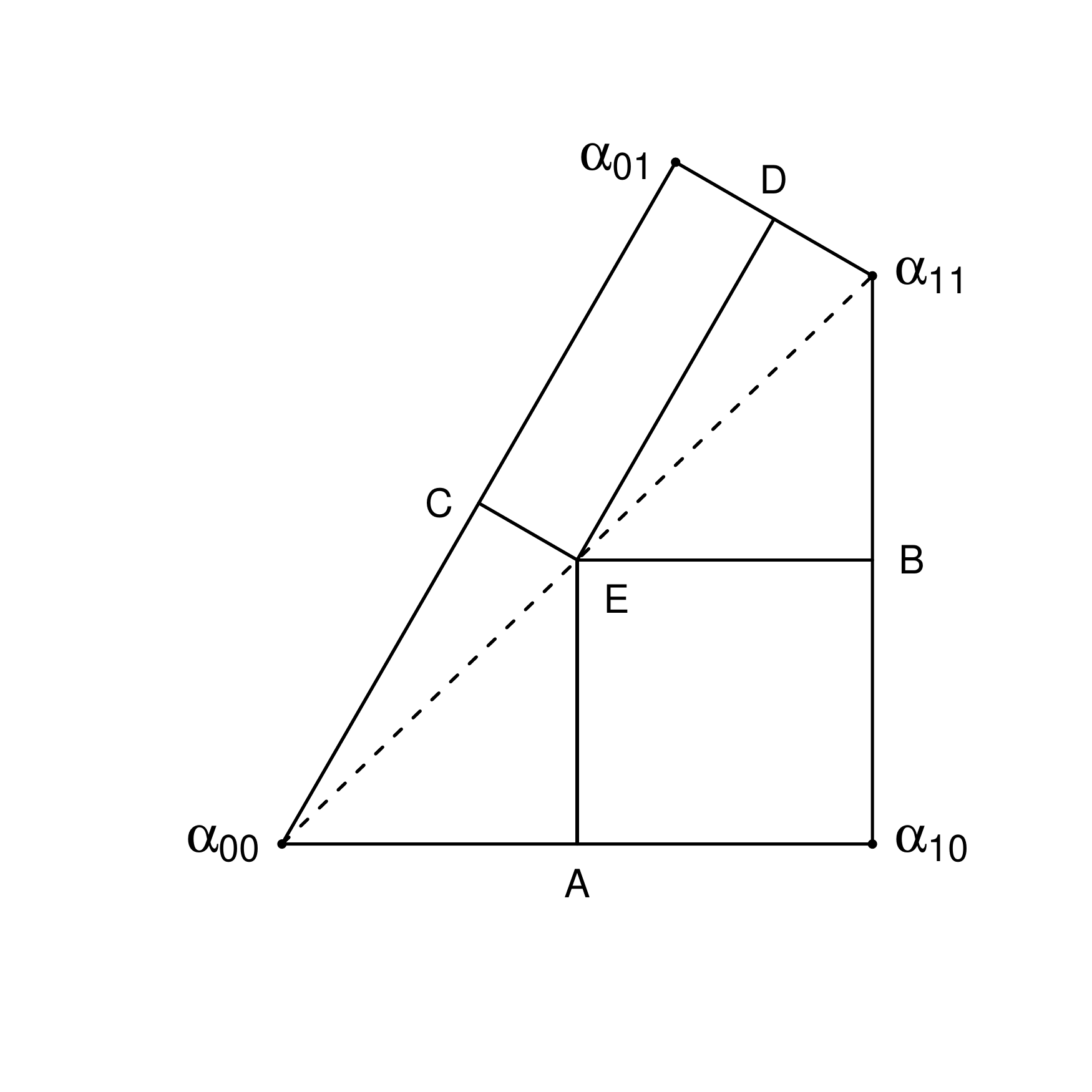}}     \label{fig:srcase2}
    \end{minipage}}
    \caption{Subregions}    \label{fig:subreg-12}
\end{figure}

Note also that to locate the position of the point inside $S_1$ or $S_2$ we just need to check the values of the weights $w^{(1)}$ or $w^{(2)}$. In fact in the nested models case the optimal model is the median. Thus, taking into account $S_1$, we know that if $w^{(1)}_{00}>1/2$ then $\bar{\mbox{\boldmath $\alpha$}}$ lies
inside $\left\{ \bar{\mbox{\boldmath $\alpha$}}_{00}, A, E \right\}$, if $w^{(1)}_{11}>1/2$ inside $\left\{ \bar{\mbox{\boldmath $\alpha$}}_{11}, B, E \right\}$, otherwise inside $\left\{ \bar{\mbox{\boldmath $\alpha$}}_{10}, A, E, B \right\}$.

In case 3 the triangular subregions
$S_1$ and $S_2$ overlap and their union does not cover the entire space $\left\{\mbox{\boldmath $\alpha$}_{00}, \mbox{\boldmath $\alpha$}_{10}, \mbox{\boldmath $\alpha$}_{01}, \mbox{\boldmath $\alpha$}_{11}\right\}$ (see Figure \ref{fig:srcase3-s1} and \ref{fig:srcase3-s2}). However in this case we may refer to $S_3=\left\{\mbox{\boldmath $\alpha$}_{10}, \mbox{\boldmath $\alpha$}_{01}, E \right\}$, $S_4=\left\{\mbox{\boldmath $\alpha$}_{00}, \mbox{\boldmath $\alpha$}_{10}, E \right\}$ and $S_5=\left\{\mbox{\boldmath $\alpha$}_{01}, \mbox{\boldmath $\alpha$}_{11}, E \right\}$, where $E=\left( \begin{array}{c}a/2 \\ d/2 \end{array}\right)$ is the midpoint of the edge linking $\mbox{\boldmath $\alpha$}_{00}$ and $\mbox{\boldmath $\alpha$}_{11}$ (see Figure \ref{fig:srcase3}). To locate the position of the point inside $S_3$, $S_4$ or $S_5$ we just need to check the value of which of the weights of the two vertexes different from $E$ is the largest.

\begin{figure}[!t]
     \subfigure[]{
     \begin{minipage}[t]{0.31\textwidth}
       \hskip-1pc  \scalebox{0.3}{\includegraphics{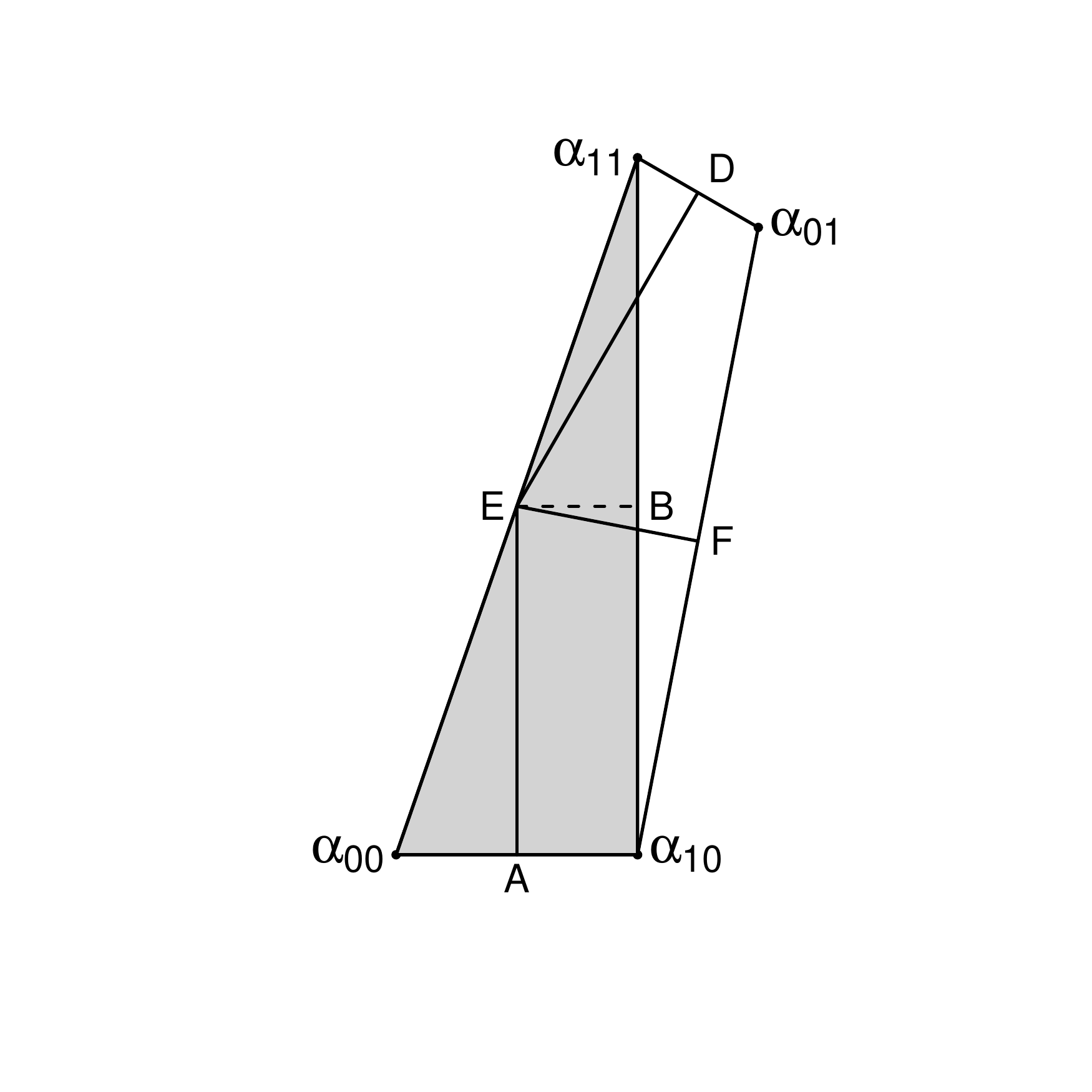}}  \label{fig:srcase3-s1}
    \end{minipage}}
     \subfigure[]{
    \begin{minipage}[t]{0.31\textwidth}
    \hskip-1pc  \scalebox{0.3}{\includegraphics{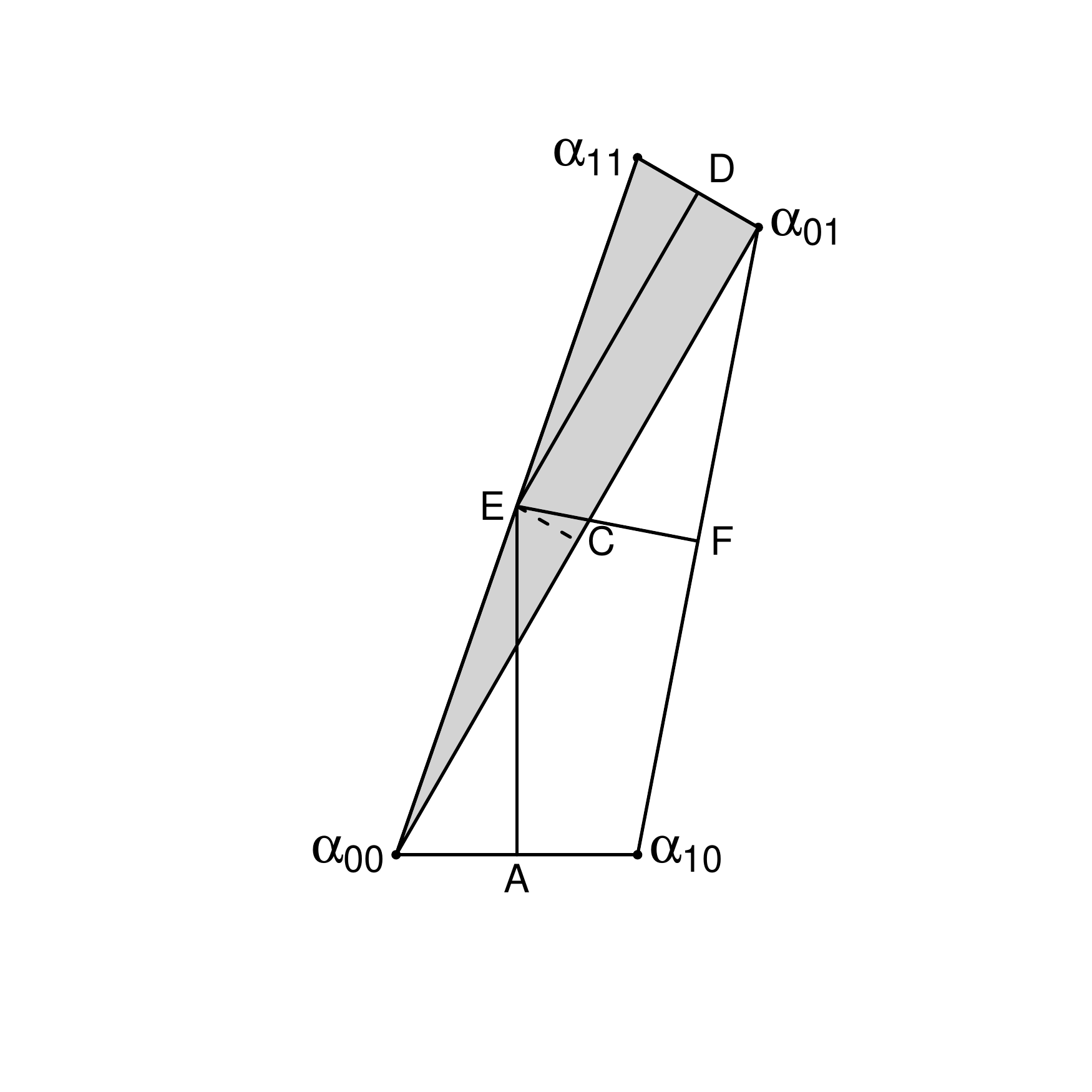}}     \label{fig:srcase3-s2}
    \end{minipage}}
     \subfigure[]{
    \begin{minipage}[t]{0.31\textwidth}
     \hskip-1pc \scalebox{0.3}{\includegraphics{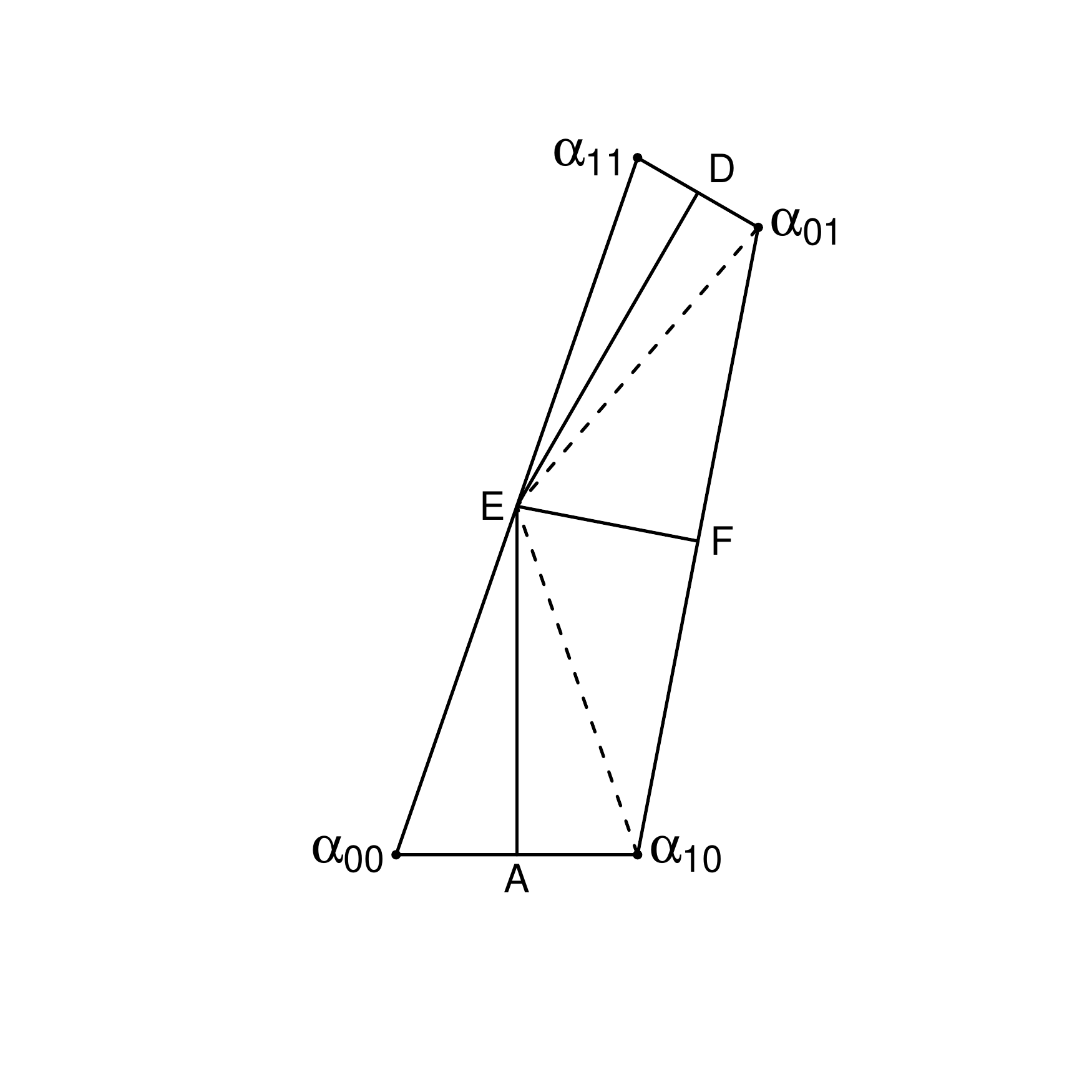}}   \label{fig:srcase3}
    \end{minipage}}
    \caption{Subregions: Case 3}    \label{fig:subreg-3}
\end{figure}

In the rest of the section, the weights for these new subregions are reported. In particular, when we refer to the triangular subregion
$S_3=\left\{\mbox{\boldmath $\alpha$}_{10}, \mbox{\boldmath $\alpha$}_{01}, E \right\}$, from
\begin{displaymath}
\left(
\begin{array}{c}\bar{\alpha}_1 \\ \bar{\alpha}_2 \end{array} \right)=w^{(3)}_{10} \left(
\begin{array}{c}a \\0 \end{array} \right) +w^{(3)}_{01} \left(
\begin{array}{c}b \\c \end{array} \right)+w^{(3)}_{E}\left(
\begin{array}{c}a/2 \\d/2 \end{array}
\right)
\end{displaymath}
and $w^{(3)}_{E}+w^{(3)}_{10}+w^{(3)}_{01}=1$, we obtain
\begin{eqnarray*}
w^{(3)}_{10}&=&\frac{(2c-d)\, \bar{\alpha}_1-(2b-a)\, \bar{\alpha}_2-ac+bd}{ac+bd-ad}\\
w^{(3)}_{01}&=&\frac{d\, \bar{\alpha}_1+a\, \bar{\alpha}_2-ad}{ac+bd-ad}\\
w^{(3)}_{E}&=&2 \frac{ac-c\, \bar{\alpha}_1-(a-b)\, \bar{\alpha}_2}{ac+bd-ad}.
\end{eqnarray*}

When we refer to the triangular subregion
$S_4=\left\{\mbox{\boldmath $\alpha$}_{00}, \mbox{\boldmath $\alpha$}_{10}, E \right\}$, from
\begin{displaymath}
\left(
\begin{array}{c}\bar{\alpha}_1 \\ \bar{\alpha}_2 \end{array} \right)=w^{(4)}_{10} \left(
\begin{array}{c}a \\0 \end{array} \right) +w^{(4)}_{E}\left(
\begin{array}{c}a/2 \\d/2 \end{array}
\right)
\end{displaymath}
and $w^{(4)}_{E}+w^{(4)}_{00}+w^{(3)}_{10}=1$, we obtain
\begin{eqnarray*}
w^{(4)}_{00}&=&1-\frac{ \bar{\alpha}_1}{a}-\frac{ \bar{\alpha}_2}{d}\\
w^{(4)}_{10}&=&\frac{ \bar{\alpha}_1}{a}-\frac{ \bar{\alpha}_2}{d}\\
w^{(4)}_{E}&=&2 \frac{ \bar{\alpha}_2}{d}.
\end{eqnarray*}

When we refer to the triangular subregion
$S_5=\left\{\mbox{\boldmath $\alpha$}_{01}, \mbox{\boldmath $\alpha$}_{11}, E \right\}$, from
\begin{displaymath}
\left(
\begin{array}{c}\bar{\alpha}_1 \\ \bar{\alpha}_2 \end{array} \right)=w^{(5)}_{01} \left(
\begin{array}{c}b \\c \end{array} \right) +w^{(5)}_{11} \left(
\begin{array}{c}a \\d \end{array} \right)+w^{(5)}_{E}\left(
\begin{array}{c}a/2 \\d/2 \end{array}
\right)
\end{displaymath}
and $w^{(5)}_{E}+w^{(5)}_{01}+w^{(5)}_{11}=1$, we obtain
\begin{eqnarray*}
w^{(5)}_{01}&=&\frac{a\, \bar{\alpha}_2-d\, \bar{\alpha}_1}{ac-bd}\\
w^{(5)}_{11}&=&\frac{(2c-d)\, \bar{\alpha}_1-(2b-a)\, \bar{\alpha}_2}{ac-bd}-1\\
w^{(5)}_{E}&=&2 \frac{(d-c)\, \bar{\alpha}_1+(b-a)\, \bar{\alpha}_2}{ac-bd}+2.
\end{eqnarray*}

Conditions under which each model is optimal may be derived using the sets of $w$'s weights.
In particular, $M_{00}$ is optimal if:
$$w^{(1)}_{00} \geq \frac{1}{2} \quad \quad w^{(2)}_{00}\geq \frac{1}{2} \quad \quad w^{(4)}_{00} \geq w^{(4)}_{10}.$$
However, since $w^{(4)}_{00}= w^{(4)}_{10}+2 \, w^{(1)}_{00}-1$, the third condition is equivalent to the first and the first two give:
\begin{eqnarray*}
&&p_1+p_{01}\, r_{12} \frac{r_{2y}}{r_{1y}} \leq \frac{1}{2} \\
&&p_2+p_{10}\, r_{12} \frac{r_{1y}}{r_{2y}} \leq \frac{1}{2},
\end{eqnarray*}
where $p_1= p_{10}+p_{11}$ and $p_2=p_{01}+p_{11}$ are the posterior inclusion probabilities of the two covariates.

Model $M_{10}$ is optimal if:
$$w^{(1)}_{00}\leq \frac{1}{2} \quad \quad w^{(1)}_{00}+w^{(1)}_{10}=1-w^{(1)}_{11}\geq \frac{1}{2} \quad \quad w^{(3)}_{10} \geq w^{(3)}_{01} \quad \quad w^{(4)}_{10} \geq w^{(4)}_{00}.$$
Where, as before, the last condition is equivalent to the first and the other three may be restated as:
\begin{eqnarray*}
&&p_1+p_{01}\, r_{12} \frac{r_{2y}}{r_{1y}}\geq \frac{1}{2} \\
&&p_2+p_{01}\, r_{12} \frac{r_{1y}}{r_{2y}}\frac{1-r_{12}\frac{r_{2y}}{r_{1y}} }{1-r_{12}\frac{r_{1y}}{r_{2y}}} \leq \frac{1}{2}\\
&& \left(\frac{r_{1y}}{r_{2y}}\right)^2 \left[\left(1- r_{12} \frac{r_{2y}}{r_{1y}}\right)
p_1-\frac{1}{2}\right] \geq \left[\left(1- r_{12} \frac{r_{1y}}{r_{2y}}\right)
p_2-\frac{1}{2}\right].
\end{eqnarray*}

Model $M_{01}$ is optimal if:
$$w^{(2)}_{00}\leq \frac{1}{2} \quad \quad w^{(2)}_{00}+w^{(2)}_{01}=1-w^{(2)}_{11}\geq \frac{1}{2} \quad \quad w^{(3)}_{10} \leq w^{(3)}_{01} \quad \quad w^{(5)}_{01} \geq w^{(5)}_{11}.$$
 Since $w^{(5)}_{11}= 2 \, w^{(2)}_{11}+w^{(5)}_{01}-1$, the last condition is equivalent to the second and the first three give:
\begin{eqnarray*}
&&p_1+p_{10}\, r_{12} \frac{r_{2y}}{r_{1y}}  \frac{1-r_{12}\frac{r_{1y}}{r_{2y}} }{1-r_{12}\frac{r_{2y}}{r_{1y}} }\leq \frac{1}{2} \\
&&p_2+p_{10}\, r_{12} \frac{r_{1y}}{r_{2y}} \geq \frac{1}{2}\\
&& \left(\frac{r_{1y}}{r_{2y}}\right)^2 \left[\left(1- r_{12} \frac{r_{2y}}{r_{1y}}\right)
p_1-\frac{1}{2}\right] \leq \left[\left(1- r_{12} \frac{r_{1y}}{r_{2y}}\right)
p_2-\frac{1}{2}\right].
\end{eqnarray*}

Finally $M_{11}$ is optimal if:
$$w^{(1)}_{11}\geq \frac{1}{2}  \quad \quad w^{(2)}_{11}\geq \frac{1}{2}  \quad \quad w^{(5)}_{01} \leq w^{(5)}_{11}.$$
Where, as before, the third is equivalent to the second and the first two may be restated as:
\begin{eqnarray*}
&&p_2+p_{01}\, r_{12} \frac{r_{1y}}{r_{2y}}\frac{1-r_{12}\frac{r_{2y}}{r_{1y}} }{1-r_{12}\frac{r_{1y}}{r_{2y}}} \geq \frac{1}{2}\\
&&p_1+p_{10}\, r_{12} \frac{r_{2y}}{r_{1y}}  \frac{1-r_{12}\frac{r_{1y}}{r_{2y}} }{1-r_{12}\frac{r_{2y}}{r_{1y}} }\geq \frac{1}{2}.
\end{eqnarray*}

The same conclusions may be obtained using the risks. In fact:
\begin{eqnarray*}
R(M_{10})-R(M_{00})&=&2\, a^2 \left(w^{(1)}_{00}-\frac{1}{2}\right)\\
R(M_{01})-R(M_{00})&=&2\, (b^2+c^2) \left(w^{(2)}_{00}-\frac{1}{2}\right)\\
R(M_{11})-R(M_{10})&=&2\, d^2 \left(\frac{1}{2}- w^{(1)}_{11}\right)\\
R(M_{11})-R(M_{01})&=&2\, (a^2+d^2-b^2-c^2) \left(\frac{1}{2}- w^{(2)}_{11}\right)\\
R(M_{01})-R(M_{10})&=&2\, (ac+bd-ad) \left(w^{(3)}_{10}- w^{(3)}_{01}\right)
\end{eqnarray*}
where all multiplying constants are positive.

After setting
$$A_1= r_{12} \frac{r_{1y}}{r_{2y}} \quad \quad \mbox{and} \quad \quad  A_2=r_{12} \frac{r_{2y}}{r_{1y}} , $$
we may restate the optimality conditions of each model as follows.

 $M_{00}$ is optimal if
\begin{eqnarray}
&&p_1+p_{01}\, A_2 \leq \frac{1}{2} \nonumber\\
&&p_2+p_{10}\, A_1 \leq \frac{1}{2}, \label{eq:oc00}
\end{eqnarray}
 $M_{10}$ is optimal if
\begin{eqnarray}
&&p_1+p_{01}\, A_2\geq \frac{1}{2}
\nonumber\\
&&p_2+p_{01}\, A_1\frac{1-A_2 }{1-A_1} \leq \frac{1}{2}
\label{eq:oc10}\\
&& \left(\frac{r_{1y}}{r_{2y}}\right)^2 \left[\left(1- A_2\right)
p_1-\frac{1}{2}\right] \geq \left[\left(1- A_1\right)
p_2-\frac{1}{2}\right],\nonumber
\end{eqnarray}
$M_{01}$ is optimal if
\begin{eqnarray}
&&p_1+p_{10}\, A_2  \frac{1-A_1 }{1-A_2 }\leq \frac{1}{2}
\nonumber\\
&&p_2+p_{10}\, A_1 \geq \frac{1}{2}
\label{eq:oc01}\\
&& \left(\frac{r_{1y}}{r_{2y}}\right)^2 \left[\left(1- rA_2 \right)
p_1-\frac{1}{2}\right] \leq \left[\left(1- A_1\right)
p_2-\frac{1}{2}\right],\nonumber
\end{eqnarray}
 $M_{11}$ is optimal if
\begin{eqnarray}
&&p_2+p_{01}\, A_1 \frac{1-A_2}{1-A_1} \geq \frac{1}{2}
\nonumber\\
&&p_1+p_{10}\, A_2  \frac{1-A_1 }{1-A_2 }\geq \frac{1}{2}.
\label{eq:oc11}
\end{eqnarray}

\begin{table}[ht]
\begin{center}
\begin{tabular}{||c|c|c||}\hline\hline
 Case 1 & Case 2 & Case 3\\\hline\hline
$A_1 <0$ & $0<A_1<1$ & $0<A_1<1$ \\
$A_2 <0$ & $0<A_2<1$ & $1<A_1$ \\
$B_1 <0$ & $0<B_1$ & $B_1<0$ \\
$B_2 <0$ & $0<B_2$ & $B_2<0$ \\
\hline\hline
\end{tabular}
\end{center}
\caption{Characterization of possible scenarios in term of $A_1$, $A_2$, $B_1$ and $B_2$.} \label{tab:char1}
\end{table}

From the optimality conditions and the results in Table \ref{tab:char1}, where
$$B_1=A_1 \frac{1-A_2}{1-A_1} \quad \quad \mbox{and} \quad \quad B_2 =A_2\frac{1-A_1}{1-A_2},$$
the  Mini-Theorems 1-7 in Theorem \ref{thm:mini} from Section 3.2 follow.

\section*{Appendix 2: Details from the Numerical Study}

We first discuss the choice of the correlation ranges adopted in the numerical studies. The idea is to find,
for each possible true model -- null, one-variable and full -- the natural ranges of $r_{1y}$ and $r_{2y}$,
in the sense of spanning the high probability region of data arising from the true model.

We do the computations in this appendix without standardizing variables, so that $\beta_1$ and $\beta_2$ in
the true model do not change with $n$.
Thus  $r_{12}= {\bm x_1}'{\bm x_2}/[\|{\bm x}_1\| \|{\bm x}_2\|]$.  Note that, with $\bm \varepsilon \sim N_n(\bm 0,\bm I)$,
$ Z_i = {\bm x_i}'\bm \varepsilon \sim N(0,\|{\bm x}_i\|^2)$, $ Z_i^* = \frac{Z_i}{\|{\bm x}_i\|} \sim N(0,1)$,
and ${\bm \varepsilon}' {\bm \varepsilon} \sim \chi_n^2$,
\begin{eqnarray*} \|\bm y\|^2 &=& \|{\bm X} {\bm \beta} + {\bm \varepsilon}\|^2
= \|{\bm x}_1\|^2 \beta_1^2 + \|{\bm x}_2\|^2 \beta_2^2 + 2 r_{12}\|{\bm x}_1\| \|{\bm x}_2\| \beta_1  \beta_2
+2 Z_1 \beta_1 + 2 Z_2 \beta_2 + \chi_n^2 \,, \\
r_{1y} &=& \frac{{\bm x_1}' {\bm y}}{\|{\bm x}_1\|\|\bm y\|}
= \frac{{\bm x_1}' [{\bm X} {\bm \beta} + {\bm \varepsilon}]}{\|{\bm x}_1\|\|\bm y\|}
 = \frac{\|{\bm x}_1\|^2 \beta_1 + r_{12}\|{\bm x}_1\| \|{\bm x}_2\| \beta_2 + Z_1}{\|{\bm x}_1\|\|\bm y\|}=\\
& =&\frac{\|{\bm x}_1\| \beta_1 + r_{12}\|{\bm x}_2\| \beta_2 + Z_1^*}{\|\bm y\|} \,, \\
r_{2y} &=& \frac{{\bm x_2}' {\bm y}}{\|{\bm x}_2\|\|\bm y\|}
= \frac{{\bm x_2}' [{\bm X} {\bm \beta} + {\bm \varepsilon}]}{\|{\bm x}_2\|\|\bm y\|}
 = \frac{\|{\bm x}_2\|^2 \beta_2 + r_{12}\|{\bm x}_1\| \|{\bm x}_2\| \beta_1 + Z_2}{\|{\bm x}_2\|\|\bm y\|}=\\
& =&\frac{\|{\bm x}_2\| \beta_2 + r_{12}\|{\bm x}_1\| \beta_1 + Z_2^*}{\|\bm y\|}  \,.
\end{eqnarray*}

\smallskip
\noindent
{\bf When the full model is true:} There is nothing unusual about the behavior of $r_{1y}$ and $r_{2y}$, so
they are allowed to vary independently over the grid {\small $\{0.1, 0.2, 0.3, 0.4, 0.5, 0.6, 0.7, 0.8, 0.9\}$},
but with $r_{1y} \leq r_{2y}$ to eliminate duplicates. Also, only correlations for which
the resulting correlation matrix is positive definite are considered.

\bigskip
\noindent
{\bf When the null model is true:} Now the expressions above become
$$ \|\bm y\|^2 = \chi_n^2, \quad r_{1y}= \frac{Z_1^*}{\sqrt{\chi_n^2}}, \quad  r_{2y}= \frac{Z_2^*}{\sqrt{\chi_n^2}} \,.$$
So, if we want to cover, say, 90\% of the probability range of the $r_{iy}$, we should use a grid such as
$$\{\frac{0.2}{\sqrt{n}},\frac{0.4}{\sqrt{n}},\frac{0.6}{\sqrt{n}},\frac{0.8}{\sqrt{n}},\frac{1.0}{\sqrt{n}},\frac{1.2}{\sqrt{n}},
\frac{1.4}{\sqrt{n}},\frac{1.6}{\sqrt{n}},\frac{1.8}{\sqrt{n}} \} \,,$$
again with $r_{1y} \leq r_{2y}$ and keeping only those for which
the resulting correlation matrix is positive definite.
(For small $n$, one would want to use a grid from the $t$-distribution with $n$ degrees of freedom, since that is the
distribution of the $r_{iy}$ but, for the numerical study, this is not necessary.)

\bigskip
\noindent
{\bf When $\beta_1 = 0$ and $\beta_2 \neq 0$:} Now the expressions above become
\begin{eqnarray*} \|\bm y\|^2 &=&  \|{\bm x}_2\|^2 \beta_2^2 +  2 Z_2 \beta_2 +  \chi_n^2 \,, \\
r_{1y} &=& \frac{ r_{12}\|{\bm x}_2\| \beta_2 + Z_1^*}{\sqrt{|\|{\bm x}_2\|^2 \beta_2^2 +  2 Z_2 \beta_2 +  \chi_n^2|}}
\approxeq \frac{r_{12}\|{\bm x}_2\| \beta_2}{\sqrt{|\|{\bm x}_2\|^2 \beta_2^2 +  2 Z_2 \beta_2 +  \chi_n^2|}} \,, \\
r_{2y} &=&  \frac{\|{\bm x}_2\| \beta_2  + Z_2^*}{\sqrt{|\|{\bm x}_2\|^2 \beta_2^2 +  2 Z_2 \beta_2 +  \chi_n^2|}}
\approxeq \frac{\|{\bm x}_2\| \beta_2}{\sqrt{|\|{\bm x}_2\|^2 \beta_2^2 +  2 Z_2 \beta_2 +  \chi_n^2|}} \,,
\end{eqnarray*}
the last approximations following because the $Z_i^*$ are $O(1)$ and the other terms are $O(\sqrt{n})$.
As in the full model case, both correlations are $O(1)$, so nothing has to go to zero. But note that
$$ r_{1y} \approxeq  r_{12} r_{2y} \,.$$
Since the error in the approximation is $O(1/{\sqrt{n}})$ (and looks to be smaller than $1/{\sqrt{n}}$), this suggests gridding $r_{2y}$ in the usual way (from 0.1 to 0.9) and then using a grid for $r_{1y}$ such as
$$ \left\{\left(r_{12} r_{2y} + \frac{h}{\sqrt{n}}\right), \quad h \in \{-0.9,-0.7,-0.5,-0.3,-0.1, 0.1, 0.3, 0.5, 0.7, 0.9\} \right\}\,,$$
again with $r_{1y} \leq r_{2y}$ and keeping only those for which
the resulting correlation matrix is positive definite.

{\begin{table}[t!]
\centering

\medskip
\scalebox{0.8}{ \begin{tabular}{|l|r|r|r|r|r|r|r|r|} \hline
  & M=H=O & M=H$\neq$O & M=O$\neq$H & H=O$\neq$M &H$>$M &  M$>$H & GM $\frac{R(\bg^{M})}{R(\bg^o)}$& GM $\frac{R(\bg^{H})}{R(\bg^o)}$ \\
   &   &  &  & & both $\neq$ O  & both $\neq$ O & &  \\ \hline\hline
   \multicolumn{9}{c}{Case 1: Full model scenario}\\ \hline\hline
n=10&158 &5 &14 &3 &0 &0 &1.01 &1.03\\
n=50&177 &1 &2 &0 &0 &0 &1.001 &1.006 \\
n=100&178 &0 &2 &0 &0 &0 &1 &1.003\\
    \hline \hline
  \multicolumn{9}{c}{Case 2: Full model scenario}\\ \hline\hline
n=10&156 &52 &10 &0 &4$^\star$ &0 &1.11 &1.15\\
n=50&200 &3 &11 &0 &8$^\star$ &0 &1.04 &1.13 \\
n=100&206 &4 &12 &0 &0 &0 &1.01 &1.13\\
    \hline \hline
  \multicolumn{9}{c}{Case 3: Full model scenario}\\ \hline\hline
n=10&90 &36 &3 &0 &0 &3$^\star$ &1.13 &1.15 \\
n=50&128 &4 &0 &0 &0 &0 &1.01 &1.01 \\
n=100&128 &4 &0 &0 &0 &0 &1.02 &1.02\\
    \hline \hline
       \multicolumn{9}{c}{Cases combined: Full model scenario}\\ \hline\hline
n=10&404 &93 &27 &3 &4$^\star$ &3$^\star$ &1.08 &1.10\\
n=50&505 &8 &13 &0 &8$^\star$ &0 &1.02 &1.06 \\
n=100&512 &8 &14 &0 &0 &0 &1.02 &1.06\\
    \hline \hline
    Overall & 1421 &109 &54 &3 &12$^\star$ &3$^\star$ &1.03 &1.07  \\  \hline \hline
      & 88.7$\%$ & 6.8$\%$ & 3.4$\%$ & 0.2$\%$ & 0.7$\%$ & 0.2$\%$ &&\\ \hline \hline
 \end{tabular}}
\caption{A summary of the numerical study in the two variable case. Legend: H = HPM, M = MPM, O = optimal predictive model;  M$>$H means that MPM has smaller \eqref{Rcrit} than HPM; H$>$M means that HPD has smaller \eqref{Rcrit} than MPM; and GM is the geometric mean of relative risks (to the optimal model) when MPM or HPM is not optimal.} \label{table-full}
{$^*$ Curiously, the optimal model, $O$, is the {\em lowest} probability model in these cases. }
\end{table}}

{\begin{table}[h!]
\centering
\medskip
\scalebox{0.8}{ \begin{tabular}{|l|r|r|r|r|r|r|r|r|} \hline
  & M=H=O & M=H$\neq$O & M=O$\neq$H & H=O$\neq$M &H$>$M &  M$>$H & GM $\frac{R(\bg^{M})}{R(\bg^o)}$& GM $\frac{R(\bg^{H})}{R(\bg^o)}$ \\
   &   &  &  & & both $\neq$ O  & both $\neq$ O & &  \\ \hline\hline
   \multicolumn{9}{c}{Case 1: $\beta_1 = 0$ and $\beta_2 \neq 0$ scenario}\\ \hline\hline
n=10&119 &2 &7 &2 &0 &0 &1.007 &1.012\\
n=50&75 &0 &0 &0 &0 &0 &1 &1 \\
n=100&45 &0 &0 &0 &0 &0 &1 &1\\
    \hline \hline
  \multicolumn{9}{c}{Case 2: $\beta_1 = 0$ and $\beta_2 \neq= 0$ scenario}\\ \hline\hline
n=10&162 &79 &3 &0 &3$^\star$ &0 &1.136 &1.036\\
n=50&288 &29 &2 &1 &8$^\star$ &0 &1.095 &1.085 \\
n=100&325 &25 &0 &0 &0 &0 & 1.058 &1.058\\
    \hline \hline
  \multicolumn{9}{c}{Case 3: $\beta_1 = 0$ and $\beta_2 \neq 0$ scenario}\\ \hline\hline
n=10&143 &88 &8 &0 &0 &4$^\star$ &1.126 &1.151 \\
n=50&298&28 &1 &0 &1 &0 &1.028 &1.026 \\
n=100&312 &40 &1 &0 &0 &1$^\star$ &1.037 &1.039\\
    \hline \hline
       \multicolumn{9}{c}{Cases combined: $\beta_1 = 0$ and $\beta_2 \neq 0$ scenario}\\ \hline\hline
n=10&424 &169 &18 &2 &3$^\star$ &4$^\star$  &1.104 &1.114\\
n=50&661 &57 &3 &1 &9 &0 &1.054 &1.049 \\
n=100&682 &65 &1 &0 &0 &1$^\star$ &1.045 &1.046\\
    \hline \hline
    Overall & 1767 &291 &22 &3 &12 &5$^\star$ &1.065 &1.067  \\  \hline \hline
    & 84.1$\%$ & 13.9$\%$ & 1.1$\%$ & 0.1$\%$ & 0.6$\%$ & 0.2$\%$ &&\\\hline \hline
 \end{tabular}}
\caption{A summary of the numerical study in the two variable case. Legend: H = HPM, M = MPM, O = optimal predictive model;  M$>$H means that MPM has smaller \eqref{Rcrit} than HPM; H$>$M means that HPD has smaller \eqref{Rcrit} than MPM; and GM is the geometric mean of relative risks (to the optimal model) when MPM or HPM is not optimal.} \label{table-onevariable}
{$^*$ Curiously, the optimal model, $O$, is the {\em lowest} probability model in these cases. }
\end{table}}

{\begin{table}[h!]
\centering
\medskip
\scalebox{0.8}{ \begin{tabular}{|l|r|r|r|r|r|r|r|r|} \hline
  & M=H=O & M=H$\neq$O & M=O$\neq$H & H=O$\neq$M &H$>$M &  M$>$H & GM $\frac{R(\bg^{M})}{R(\bg^o)}$& GM $\frac{R(\bg^{H})}{R(\bg^o)}$ \\
   &   &  &  & & both $\neq$ O  & both $\neq$ O & &  \\ \hline\hline
   \multicolumn{9}{c}{Case 1: Null model scenario}\\ \hline\hline
n=10&268 &16 &35 &2 &0 &0 &1.01 &1.04\\
n=50&382 &5 &11 &3 &0 &0 &1.002 &1.008 \\
n=100&397 &2 &4 &1 &0 &1 &1.0009 &1.0036\\
    \hline \hline
  \multicolumn{9}{c}{Case 2: Null model scenario}\\ \hline\hline
n=10&159 &70 &3 &0 &7$^\star$ &0 &1.12 &1.09\\
n=50&233 &6 &0 &0 &0 &0 &1.006 &1.006 \\
n=100&239 &0 &0 &0 &0 &0 &1 &1\\
    \hline \hline
  \multicolumn{9}{c}{Case 3: Null model scenario}\\ \hline\hline
n=10&43 &92 &12 &0 &1 &7$^\star$ &1.38 &1.37 \\
n=50&67&87 &2 &0 &0 &2$^\star$ &1.20 &1.21 \\
n=100&99 &58 &1 &0 &0 &0 &1.03 &1.12\\
    \hline \hline
       \multicolumn{9}{c}{Cases combined: Null model scenario}\\ \hline\hline
n=10&470 &178 &50 &2 &8 &7$^\star$ &1.106 &1.120\\
n=50&682 &98 &13 &3 &0 &2$^\star$ &1.039 &1.045 \\
n=100&735 &60 &5 &1 &0 &1 &1.023 &1.024\\
    \hline \hline
    Overall & 1887 &336 &68 &6 &8 &10 &1.054 &1.060  \\  \hline \hline
    & 81.6$\%$ & 14.5$\%$ & 2.9$\%$ & 0.3$\%$ & 0.3$\%$ & 0.4$\%$ &&\\\hline \hline
 \end{tabular}}
\caption{A summary of the numerical study in the two variable case. Legend: H = HPM, M = MPM, O = optimal predictive model;  M$>$H means that MPM has smaller \eqref{Rcrit} than HPM; H$>$M means that HPD has smaller \eqref{Rcrit} than MPM; and GM is the geometric mean of relative risks (to the optimal model) when MPM or HPM is not optimal.} \label{table-null}
{$^*$ Curiously, the optimal model, $O$, is the {\em lowest} probability model in these cases. }
\end{table}}

\end{document}